\documentclass[12pt]{amsart}

\usepackage{amsmath,amssymb,amscd}
\usepackage[mathscr]{eucal}
\newcommand{\Spec}{{\mathrm{Spec}\, }}

\newcommand{\bbar}[1]{{\overline{#1}}}

\newcommand{\cA}{{\mathcal A}}

\newcommand{\cC}{{\mathscr C}}

\newcommand{\cE}{{\mathcal E}}
\newcommand{\cF}{{\mathcal F}}
\newcommand{\cH}{{\mathcal H}}
\newcommand{\cM}{{\mathcal M}}

\newcommand{\cO}{{\mathcal O}}
\newcommand{\cP}{{\mathscr P}}
\newcommand{\cS}{{\mathscr S}}
\newcommand{\cT}{{\mathscr T}}
\newcommand{\cX}{{\mathcal X}}

\renewcommand{\sp}[1]{{\mathrm{Spec}}}

\newcommand{\scr}[1]{\mathcal {#1}}

\renewcommand{\l}{\lambda}
\renewcommand{\o}{\omega}
\newcommand{\sg}{\mathrm{Sing}}
\newcommand{\ma}{\mathrm{Max}}

\renewcommand{\max}{\mathrm{max} }

\theoremstyle{plain}
\newtheorem{thm}{Theorem}[section]
\newtheorem{lem}[thm]{Lemma}
\newtheorem{pro}[thm]{Proposition}

\theoremstyle{definition}
\newtheorem{Def}[thm]{Definition}
\newtheorem{rem}[thm]{Remark}
\newtheorem{exa}[thm]{Example}
\newtheorem{voi}[thm]{}
\begin{document}

\bigskip 
\title[Algorithmic equiresolution of deformations]{Algorithmic equiresolution of deformations of embedded varieties}

\author{Augusto Nobile}
%\begin{abstract}
%Singularities ...
%\end{abstract}

\address{Louisiana State University \\
Department of Mathematics \\
Baton Rouge, LA 70803, USA}



%\date{\today}
\email{nobile@math.lsu.edu}
\maketitle

\section{Introduction}
\label{S:intro}

The fact that an arbitrary algebraic variety $X$, over a field of characteristic zero,
 admits a resolution (that is, a proper birational morphism $f:X' \to X$, with $X'$ regular) was first established by H. Hironaka in his celebrated 1964 article \cite{Hir}.

 Once resolution is available for a single variety, a natural question is to try to simultaneously resolve all the members of a given family, say parameterized by a variety $T$, a process that might be called {\it equiresolution}. It is clear that this is not always possible: to succeed means that, in some sense,  the singularities  of the different members are not too different. Moreover, there are many conditions that can be imposed to  a simultaneous resolution process, which lead to different notions (see \cite{T}, written in the late 1970's). Another possible approach, using more recent developments, involves 
 {\it algorithmic resolutions}. These theories arose as  an attempt to improve and better understand the article \cite{Hir}.
 
In \cite{Hir} the resolution is achieved by means of a finite sequence of blowing-ups, or monoidal transforms, each with a regular center contained in the singular set (and satisfying some further technical conditions, like normal flatness). But the proof, aside from being very complicated, is existential in the sense that it is not made very clear how to choose each center. For many years there was no improvement upon the original presentation. But starting in  the late 1980's several authors
 (basing their work, in part, on some pioneering contributions of Hironaka himself, e.g., 
 \cite{Hiro}) 
  obtained ``algorithmic resolution theorems'', i.e., processes of resolution where one specifies the center to choose each time we blow-up. See, e.g., \cite{V1}, \cite{EV}, \cite{BEV}, \cite{BM}, \cite{EH}, \cite{Kol}, \cite{W}. This approach not only simplifies the original proof of \cite{Hir}, but affords some additional results. For instance, one obtains {\it equivariant} resolutions, i.e., compatible with the action of a group on the variety.
  
  A common feature of all these works is the substitution of the original problem, where one directly deals with algebraic varieties, by another, seemingly more technical one, where one tries to ``improve'' other objects, involving a sheaf of ideals on a regular ambient scheme. More precisely, one deals with {\it basic objects}. We follow the terminology of \cite{BEV}, although the language and  presentation  vary with the different authors. A basic object (over a field $k$) is a four-tuple $B=(W,I,b,E)$ where $W$ is a $k$-smooth variety, $I$ a never-zero sheaf of  
  ${\cO}_{W}$-ideals, $b$ a positive integer and $E$ a collection of smooth divisors of $W$ with normal crossings. One defines a suitable notion of {\it permissible transformation}, which gives us a new basic object, induced   by $B$ on the blowing-up $W'$ of $W$ with an appropriate regular center. Then essentially the goal is to reach, by means of a sequence of permissible transformations, a basic object $B_r=(W_r,I_r,b,E_r)$ where for all $x \in W_r$ the order of 
 the stalk $(I_r)_x$ in ${\cO}_{{W_r},x}$ is less than $b$. Moreover, one should be able to describe each center, say as the set of maximum value of an upper semicontinuous  function taking values in a suitable totally ordered set (depending on the dimension of $W$ only). If this is done in an appropriate way, the solution to this problem easily implies that of resolving varieties, in such a way that  the centers of the blowing-ups used are explicitly described (\cite {BEV}, Section 5). See also, for other procedures, e.g., \cite{BM}, \cite{Ko} or  \cite{W}. 
 
So, returning to the problem of equiresolution, a possible approach is this one: in the presence of a given algorithm of resolution (say, of embedded varieties, i.e., a closed inclusion of varieties $X \subset W$, with $W$ regular), and given a family of such embedded varieties, introduce and study a reasonable notion of algorithmic equiresolution, i.e., a sequence of monoidal transformations of the ambient total space inducing on each fiber the algorithmic resolution process. When the parameter space is smooth (or, at least, reduced), this problem was studied in \cite{EN}, where several proposed definitions (for families of ideals and of embedded varieties) are seen to be equivalent. 
See also \cite{BMM}, section 6.
 
 But there is something unsatisfactory about the restriction to the use of reduced parameter spaces only. In many problems one is naturally lead to  the consideration of families parameterized by non-reduced schemes. For instance, assuming a suitable resolution algorithm has been fixed, an application discussed in \cite{EN} consists of a stratification of the Hilbert scheme of subvarieties of ${\bf P}^n$ with a given Hilbert polynomial, expressing it  as a union of reduced locally closed subschemes, such that the restriction of the universal family to each stratum  is algorithmically equisolvable, and universal with respect to equisolvable families parameterized by reduced schemes. Clearly the ``reduced'' condition is not entirely satisfactory. Indeed,  the Hilbert scheme itself might be non-reduced, moreover many usual techniques to study features such as  smoothness, tangent spaces, etc., involve the consideration of families parametrized by schemes of the form $S= \Spec (A)$, with $A$ an artinian ring. So, it seems natural to study, at least,  the notion of algorithmic equiresolution for families parametrized by $S$ as above (i.e., {\it infinitesimal families} or  {\it infinitesimal deformations} of  embedded varieties).
 
 In this paper the latter problem is investigated. As said, initially we deal with basic objects over an artinian ring $A$ as above (also called $A$- or $S$-basic objects, see Definition \ref{D:bo}). Of course, the main difficulty here is that we have just one fiber (the naturally induced basic object over the only point of $S$), so that to directly compare the algorithmic resolutions of the various fibers makes no sense. Rather, one should try to introduce conditions that make the notion  ``the algorithmic resolution of the (single) fiber evenly spreads along $S=\Spec(A)$'' rigorous. It seems that, essentially, the only tool we have is to require that, in a suitable sense and working with appropriate completions of local rings, the orders of certain series do not change when we reduce the coefficients modulo the maximal ideal of $A$. This is what we attempt to do. 
 
For instance, the most fundamental notion developed along these lines is that of 
{\it permissible $B$-center} (or permissible center relative to $S$), where 
$B=(p:W \to S, I, b,E)$, ($S=\Spec (A)$, $A$ an artinian local ring whose residue field has characteristic zero and $p$ is a smooth morphism) is a basic object over $A$, see \ref{D:bo} for the complete definition. Namely, by controlling certain orders of ideals, one may impose a condition on a closed subscheme $C \subset W$, flat over $S$ which, were $S$ a reduced scheme rather than an infinitesimal one, would induce on each fiber of $p$ (a regular variety over a field) a permissible center, say in the sense of \cite{BEV}. Then it is possible to define a natural notion of {\it permissible transform of $B$ with center $C$}. If we repeat the process of choosing a permissible center and transforming, we get what we call {\it an $A$-permissible sequence of $A$-basic objects}, and this induces a permissible sequence (in the sense of \cite{BEV}) at the level of closed fibers. If this sequence is a resolution of $B^{(0)}$ (the fiber of $B$)   we say that our $A$-permissible sequence is an {\it equiresolution} 
 of the $A$-basic object $B$. (This is done in sections 3 and 4). A natural question is: if an algorithm for resolution of basic objects over fields (of characteristic zero) is given, when will an equiresolution be {\it algorithmic}? 

In \cite{EN} the problem of algorithmic equiresolution of families parametrized by a reduced scheme is studied. There  one works with an arbitrary resolution algorithm satisfying certain conditions (``good algorithms''). So far I do not know how to treat the present deformation situation with this degree of generality. Trying to get experience, I deal with a specific resolution algorithm, namely essentially that of \cite{EV} (or 
 \cite{BEV}, \cite{Ma}, \cite{Cu}). We review the algorithm, for the reader's convenience, in section 2. 
 
 Let us better explain what is our main result. Let $\cA$ be the class of artinian local algebras, whose residue a field has   characteristic zero. If $B^{(0)}$ is a basic object over a characteristic zero field $k$, write $\ell (B)=r$ if 
 $$(1) \quad B^{(0)}={B_0}^{(0)} \leftarrow \cdots \leftarrow {B_r}^{(0)}$$
 is the algorithnmic resolution of $B^{(0)}$. Then we have:
 
\begin{thm}
\label{T:main}
There is a function associating to each $A$-basic object 
$$B=(W \to \Spec A,I,b,E)$$
($A \in \cA$) with fiber $B^{(0)}$ a non-negative integer $e(B) \le \ell(B^{(0)})$ and $A$-permissible centers $C_0, \ldots, C_{e(B)-1}$, where $C_0$ is a $B=B_0$-center, $C_1$ an $A$-center for the transform $B_1$ of $B_0$ with center $C_0$, and so on (thus an $A$-permissible sequence 
$$(2) \quad B=B_0 \leftarrow B_1 \leftarrow \cdots \leftarrow B_{e(B)}$$
is determined), such that: (i) each center $C_i$ induces on the fiber the $i$-th algorithmic center used in the sequence (1) (and hence (2) induces the truncation at level $e(B)$ of (1)), (ii) this correspondence is functorial with respect to  homomorphisms $A \to A'$ and etale morphisms $W' \to W$. 
\end{thm}
 
 The sequence (2) thus obtained is called the {\it partial algorithmic equiresolution of} $B$. When $e(B)=\ell (B^{(0)})$ we say that $B$ is {\it algorithmically equisolvable} and call the corresponding sequence (2) its {\it algorithmic equiresolution}.
 
 The idea to show this is to try to generalize, or ``spread'', the  different steps of  the algorithmic resolution process of the fiber (a basic object over a field) to the $A$-basic object $B$. The number $e(B)$ indicates how far we can go in the process. Of course, often $e(B)<\ell (B^{(0)})$, not any infinitesimal family will be equisolvable.  This generalization is studied in sections 4 though 8. Theorem \ref{T:main} is proved in section 9, where we introduce  conditions which, if valid, insure that centers used in the algorithmic resolution of the fiber $B^{(0)}$ extend over $S=\Spec(A)$, producing the unique centers $C_i$ of the Theorem.
 
 From this theorem similar results for infinitesimal families of ideals or of embedded varieties may be easily obtained.
 
 Actually, the algorithm we use is not strictly that of \cite{EV}. We use a slight variation of this method. The reason is that 
 there are some technical difficulties  to adapt the whole process  of \cite{EV}  to our situation. This occurs at a key inductive step, where a basic object $B=(W,I,b,E)$ is replaced by another $B_Z$, whose underlying scheme is a regular hypersurface $Z$ of $W$, which is only locally defined, not canonically (what many authors call a {\it hypersurface of maximal contact}). Because of this, there is a serious glueing problem to get globally valid statements. In the references above this is addressed by re-doing the theory in a more general set-up, that of {\it generalized basic objects} where, after proving some technical results, the glueing process is easier. It seems difficult  to adapt this approach to the context of  basic objects over an artinian ring $A$. But,  following a suggestion of S. Encinas, by using techniques of Wlodarczyk in 
 \cite{W} ({\it homogenized basic objects}, specially a suitable version of his ``Glueing Lemma''), it is possible to develop a satisfactory theory over $A$, bypassing the use of  generalized basic objects. See  \ref{V:ae5.1} for more details.

 Most of this article deals with basic objects, the transition to families of ideals or embedded varieties, by now standard (see \cite{BEV}, section 5), is explained in section \ref{S:AE}. Of course, geometrically the most interesting case is that of families of varieties, but at present it does not seem possible to discuss this situation without a rather lengthy previous study of families of basic objects (or something similar).  Also at the end of section 9 it is indicated how our conditions $\cE _i$ afford a natural notion of algorithmic equiresolution for families parametrized by arbitrary noetherian schemes. 
 
 In another article we shall compare this notion with others available when the parameter space is reduced (\cite{EN}). In the future we hope to study applications (like those discussed in \cite {EN}, Section 4) as well as the functors on Artin rings that naturally arise (see  \cite{S}).
 
 In an appendix (section \ref{S:A}) we include some rather basic algebraic results, that we could not find in the literature.

 It is my pleasure to thank O. Villamayor and S. Encinas for their help and encouragement, as well as the referee, for useful suggestions to improve this article.

\section{The algorithm}\label{S:AG}

\begin{voi}
\label{V:ag1}
{\it Terminology and conventions} In general, we shall use the notation and language of \cite{H}. We describe next a few  exceptions. If $W$ is a scheme, a $W$-ideal will mean a coherent sheaf of 
${\cO}_W$-ideals.  If $I$ is a $W$-ideal, the symbol $\mathcal V (I)$ will denote the closed subscheme of $W$ defined by $I$. As usual, $V(I)$ will denote the closed subset of the underlying topological space of $W$ of zeroes of $I$. If $Y$ is a closed subscheme of a scheme $W$, the symbol $I(Y)$ denotes the $W$-ideal defining $Y$. An algebraic variety over a field $k$ will be a reduced algebraic $k$-scheme. If $W$ is a reduced scheme, a never-zero $W$-ideal is a $W$-ideal $I$ such that the stalk $I_x$ is not zero for all $x \in W$, in general $I$ is a never-zero ideal of $W$ if $I{\cO}_{W'}$ is never-zero, with $W'=W_{red}$.

The term \emph{local ring} will mean 
 \emph{noetherian local ring}. In  general, the maximal ideal, or radical, of a local ring $R$ will be denoted by $r(R)$. Often, we write $(R,M)$ to denote the local ring $R$ with maximal ideal $M$. The order of an ideal $I$ in the local ring $(A,M)$ is the largest integer $s$ such that $I\subseteq M^s$. If $W$ is a noetherian  scheme, $I$ is a $W$-ideal and $x \in W$, then $\nu _x(I)$ denotes the order of the ideal 
 $I_x$ of ${\cO}_{W,x}$. 
 
 All the varieties we consider in this paper are defined over characteristic zero fields.
 \end{voi}
 \begin{voi}
 \label{V:1.1}
The goal of this article is to study certain questions on simultaneous resolution, or equiresolution, in the presence of an algorithm of resolution of singularities. The most interesting case is that of algebraic varieties, but in our approach it is more convenient to study first, in detail, the analogous problem for a more formal type of objects, the so-called {\it basic objects} (over fields). From this similar results for varieties easily follow. We shall work with a specific algorithm, essentially that developed in   \cite{V1} and \cite{V2} (and presented more sistematically in \cite{EV} or \cite{BEV}). For the reader's convenience, and also to motivate our work in the succeeding sections, we shall briefly review this algorithm. We omit proofs and many details, found  in the references just mentioned or, extended to the more general setting of basic objects over artinian rings, in further sections of the present paper. 

In this section we work with algebraic schemes over a field $k$, of characteristic zero (although something similar can be done in a more general context, see \cite{EN}, 5.11). 
\end{voi}
\begin{voi}
\label{V:ag2} A basic object (over a field $k$, of characteristic zero) is a four-tuple $B=(W,I,b,E)$, where $W$ is a smooth, equidimensional algebraic variety over $k$, $I$ is a never-zero $W$-ideal, $b >0$ is an integer and $E=(H_1, \ldots, H_m)$ is  a sequence of distinct regular hypersurfaces of $W$ (i.e., each $H_i$ is a regular Weil divisor of $W$) with normal crossings (\cite{BEV}, 2.1). 
 The smooth variety $W$ is the {\it underlying scheme} of $B$, denoted by $us(B)$. The dimension of $B$ is the dimension of the variety $us(B)$.

 In \cite{EV} or \cite{BEV}, the notation $(W,(I,b),E)$ is used for a basic object, we drop the inside parenthesis to simplify. 
\end{voi}
\begin{voi}
\label{V:ag3} The singular set $\sg (B)$ of a basic object is 
$\{ x \in W: \nu _{x}(I) \ge b\}$. This  is a closed set of $W$. Indeed, one may introduce an operation $\Delta ^{i}$ on $W$-ideals, $i \ge 1$, so that $\sg (B)=V({\Delta}^{b-1}(I))$. Concerning $\Delta=\Delta ^1$, if $w$ is a closed point of $ W$ and $x_1, \ldots, x_r$ is a regular system of parameters of $R={\cO}_{W,w}$ and $D_i$ is the derivation associated to $x_i$ (the ``partial derivative'' with respect to $x_i$), then ${\Delta}(I)_w$ is the ideal of $R$ generated by $I_w \cup \{D_i f : f \in I_w\}$, ${\Delta}^i$ is defined by iteration. One defines $\Delta(I)$ more intrinsically by means of suitable Fitting ideals of $\Omega_{{\nu(I)}/k}$, $k$ the base field. See \cite{BEV}, section 13, for more details.

The use of differential methods (such as the operators ${\Delta}^j(I)$) was pioneered by Jean Giraud (see \cite{G}). These  have greatly helped to  clarify and simplify the theory of  resolution of singularities. They give us an alternative way to treat Hironaka's notion of maximal contact, and play and important role in algorithmic techniques. 

\end{voi}
\begin{voi}
\label{V:ag4} 
{\it Permissible transformations}. If $B=(W,I,b,E)$, $E=(H_1, \ldots, H_m)$ is a basic object, a  {\it permissible center} for $B$ is a closed subscheme $C \subset W$ having normal crossings with the hypersurfaces $H_i$, $i=1, \ldots, m$ (in particular, $C$ is regular) such that $C \subseteq \sg(B)$. We define the {\it transform} of $B$ with center $C$ as the  basic object 
$B_1=(W_1,I_1,b,E_1))$ where $W_1$ is the blowing-up of $W$ with center $C$ (so, we have a natural morphism 
$W_1 \to W$), $E'=(H'_1, \ldots, H'_m, H_{m+1})$, where $H'_i$ is the strict transform of $H_i$ ($i=1, \ldots m$) and $H_{m+1}$ is the exceptional divisor, finally $I_1:=I(H_{m+1})^{-b}I{\cO}_{W_1}$ (cf. \cite{BEV}, section 3). The process of replacing $B$ by such new basic object $B_1$ is is called the (permissible) {\it transformation} of $B$ with center $C$,   denoted by $B \leftarrow B_1$.
%or by $B \rightarrow B_1$. 
 We'll also write $B_1:=\cT(B,C)$.

If $W_1 \to W$ is as above, we also define the {\it proper transform} of $I$ to $W_1$ as the $W_1$-ideal 
${\cE}^{-a} I {\cO}_{W_1}$, where ${\cE}$ defines the exceptional divisor and the exponent $a$ is as large as possible. This integer is constant along each irreducible component of the center $C$ used, but in general not globally constant.
\end{voi}
\begin{voi}
\label{V:ag5}
A {\it resolution} of the basic object $B$ is a sequence 
$B_0 \leftarrow \cdots \leftarrow B_r$, where each arrow $B_j \leftarrow B_{j+1}$ is a permissible transformation, such that $\sg (B_r)=\emptyset$. 

An algorithm of resolution (for  basic objects in $\cS$) is a rule that associates to each positive integer $d$  a totally ordered set $\Lambda^{(d)}$  and, for any given basic object   $B_{0}=(W_{0},I_{0},b_{0},E_{0})$ of dimension $d$, functions as follows. First, there is an upper semicontinuous function 
$g_0: \sg (B_0) \to {\Lambda}^{(d)}$, (taking finitely many values), such that 
$C_0 = {\rm Max}(g_0) = \{w \in \sg (B_0): g_0(w) ~ {\rm is ~ maximum }\}$ is a permissible center. If $B_1=\cT (B_0,C_0)$ has $\sg (B_1) \not= {\emptyset}$,  a function 
$g_1: \sg (B_1) \to {\Lambda}^{(d)}$ is given, such that $C_1 = {\rm Max}\,(g_1)$ is a well determined $B_1$-permissible center, which we blow-up, and so on. Eventually we get in this way a permissible sequence 
 $B_0\leftarrow B_1 \leftarrow \cdots \leftarrow B_r$.  
We require that this be a resolution, i.e., $\sg (B_r)=\emptyset$. This is called the algorithmic resolution sequence of $B$. Moreover, this should be stable under etale base change, meaning that if $B'$ is the basic object obtained by pull-back under an etale map $W' \to W$, then the pull-back of the resolution sequence above may be identified to the resolution sequence of $B'$ (and the new resolution functions are induced by the original ones) and, similarly, under extension of the base field. 
\end{voi}
\begin{voi}
\label{V:ag6}
In this paper we shall work with a specific algorithm, a variant of that of \cite{BEV} or \cite{EV} (with some elements from \cite{W}). We shall refer to it as the VW-algorithm. In its  construction, the following auxiliary notions are important.

\smallskip

{\noindent (i) {\it the functions ${\omega}_r$ and $t_r$}.} If 
$$(1) \qquad B_0 \leftarrow \cdots \leftarrow B_r$$
 is a sequence of basic objects and permissible transformations (where we write 
$B_j = (W_j, I_j, b, E_j)$, for all $j$), we define, for $x \in \sg (B_r)$, 
$\omega _r (x) :=   \nu _x ({\bar {I_r}}) / b$ (where ${\bar {I_r}}$ denotes  the proper transform of $I_0$ to $W_r$). It can be proved that if 
in our sequence the center of each transformation is contained in 
${\mathrm {Max}} ({\omega}_j)$ (the set of points where ${\omega}_j$ reaches its maximum value 
${\mathrm {max}} ({\omega}_j)$) then 
${\mathrm {max}} ({\o}_{j-1}) \geq {\mathrm max} ({\o}_j)$, for $j < r$. Such a sequence will be called $\o$-permissible. 

The functions $t_r$  are defined by induction on the length $r$ of a 
$\o$-permissible sequence as above.  If $r=0$, for $x \in \sg (B_0)$ we write 
$t_0(x) = (\omega _0 (x), n_0(x))$, where $n_0(x)$ is the number of hypersurfaces in $E_0$ containing $x$. Assume that $t_j= (\omega _j, n_j)$ was defined (on $\sg (B_j)$) for $j < r$ and that in our sequence (1) is $t$-permissible, i.e., that each center $C_i$ used in the blowing-ups is contained in the subset of $\sg (B_i)$ where $t_i$ reaches its maximum value (in particular then (1) is $\o$-permissible). 
Let $s$ be the smallest index such that 
${\mathrm {max}}(\o  _s) = {\mathrm {max}}(\o _r)$ and $E^{-}_r$ the collection of the hypersurfaces in $E_r$ which are strict transforms of those in $E_s$. Then, for 
$x \in \sg(B_r)$ we set: $t_r(x) = (\omega _r(x), n_r(x))$, where $n_r(x)$ is the number of hypersurfaces in $E^{-}_r$ containing $x$. A $B_r$-center which is contained in 
Max($t_r$) will be called $t$-permissible. It can be proved that in a $t$-permissible sequence the sequence ${\mathrm {max}} (t_j)$ is non-increasing.

\smallskip

(ii)  {\it Monomial objects.} A basic  object
$B=(W,I,b,E)$, where $E=(H_1,\ldots, H_m)$  is 
 {\it monomial} if for each $w \in W$ we have:
$$I_w = I(H_1) ^{\alpha _1 (w)} \ldots I(H_m) ^{\alpha _m (w)}$$ 
with each function 
$\alpha _i: W \to {\bf Z}$
 constant on each irreducible component of  $H_i$ and zero outside $H_i$. If $B$ is monomial, one may define (using combinatorial techniques) a function 
$\Gamma _B $ from $\sg (B)$ to
${\bf Z} \times {\bf Q} \times {\bf Z}^{\bf N}$
 which is upper-semicontinuous (when the target is lexicographically ordered). 
 Then, it turns out that $C:={\mathrm {Max}}(\Gamma_B)$ is a permissible center  and that, if $B_1={\cT}(B,C)$,   
${\mathrm {max}}( \Gamma _{B_1}) < {\mathrm {max}}( \Gamma _{B})$. This center $C$ is the intersection of certain hypersurfaces in $E$, which are explicitly determined from $\Gamma _B$ 
  (see \cite{BEV}, section 20, or \ref{V:det4.0} of this paper).

\end{voi}
\begin{voi}
\label{V:ag7}
Now we briefly describe our resolution algorithm, that will be called {\it the VW-resolution algorithm}.  

($\alpha$) For each integer $d \geq 1$ we must indicate a totally ordered set $\Lambda^{(d)}$  and, for any given basic object $B_{0}=(W_{0},I_{0},b_{0},E_{0})$ over $k$  of dimension $d$,  the corresponding resolution functions $g_j$.
 
This process will be defined inductively on the dimension of $B_0$, as follows. In the sequel, 
$\cS _1:= {\bf Q}\times {\bf Z}$ and $\cS _2:= {\bf Z}\times {\bf Q}\times {\bf Z}^{\bf N}$, in all cases lexicographically ordered.

If ${\rm dim} (B_0) = 1$, let $\Lambda ^{(1)}= \cS _1 \cup \cS _2 \cup  \{\infty _1\}$, where if $a \in \cS _2 $ and $b \in \cS _1 $ then $a >b$  and $\infty _1$ is the largest element of the set. Then we define  for 
$w \in \sg (B_0)$, $g_0(w)= t_0(w)$  and, if $g_i$ is defined for $i < s$, determining a permissible sequence 
$B_0\leftarrow B_1 \leftarrow \cdots \leftarrow B_s$ we define, for $w \in 
\sg(B_s)$, $g_s(w)=t_s(w)$ if $\o _s(w) > 0$ and $g_s(w)=\Gamma _{B_s} (w)$ otherwise.

In the induction step we need the following  auxiliary construction.

\smallskip
  
($\beta$) {\noindent {\it  Inductive step}}. Assume that we have an algorithm of resolution  defined for basic objects of dimension $< d$. 
Consider a $t$-permissible sequence of basic objects and transformations 
$$ (1) \qquad B_0 \leftarrow B_1 \leftarrow \cdots \leftarrow B_s $$ 
Let $w \in \ma (t_s)$ and suppose that, near $w$,  $\dim \ma(t_s) \le d-2$. 
   Then there is an open neighborhood $U$ of $w$ (in $W_s=us(B_s)$), a hypersurface $Z_s$ on $U$, containing $w$, and a basic object 
${B_s}^*=(Z_s, {I_s}^*, {b_s}^*,{E_s}^*)$, having the following properties:

\smallskip
 
 (i) $\sg ({B_s}^*) = {\rm Max} \, ({t_s} _{|U})$. 
 
 \smallskip

(ii) The algorithmic resolution sequence corresponding (by the induction hypothesis) to ${B_s}^*$:
 $$ (2) \qquad {B_s}^* \leftarrow ({B_s}^*)_{1} \leftarrow \cdots \leftarrow 
 ({B_s}^*)_{p}$$ 
(determined, say, by resolution functions $\widetilde{g}_i$) 
 induces a $t$-permissible sequence 
 $$ (3) \qquad  {\widetilde B_s} \leftarrow  {\widetilde B_{s+1}} \leftarrow 
 {\widetilde B_{s+p}}             $$ 
 (obtained by using the same centers $C_i=\ma (\widetilde{g}_i)$, and denoting by  ${\widetilde B_s}$ the restriction of $B_s$ to $U$).
 
 \smallskip
 
 (iii) If 
 ${\rm max} (t _s)={\rm max} (t _{s+j})$ ($j=1, \ldots,  p$) then,  
  for all such indices $j$,  
 $us(({{B_{s}}^*})_j)$ gets identified to $Z_{s+j}$, the strict transform of $Z_s$ to $us({\widetilde B}_{s+j})$ and 
$\sg(({{B_{s}}^*})_j)= {\rm Max} (\widetilde{t_{r+j}})$ (where $\widetilde{t_j}$ are the $t$-functions of the sequence (3)). 

\smallskip

(iv) Under the assumption of (iii) for all $j=0, \ldots , p$, if 
$w_j \in {\rm Max}(\widetilde{t_{s+j}})$
 is in the pre-image of $w$ (under the morphism $\widetilde B_{s+j} \to \widetilde B_{s}$ arising from (3)), the resolution function $\widetilde{g_j}$ of ${B_s}^*$ defines a function (still denoted by $\widetilde{g_j}$) on a neighborhood (in ${\rm Max}(\widetilde{t_{s+j}})$) of $w_j$. The process does not uniquely determine the neighborhood $U$ nor the hypersurface $Z_s$, but the value $\widetilde{g_j}(w_j)$ is independent of the choices. 
 
 In \ref{V:ag8} we are going to indicate how to make these constructions. 

 \smallskip

($\gamma$) Now, assuming the resolution functions given for dimension $< d$,  we'll define resolution functions $g_j$ for objects of dimension $d$ as follows. In this case,  
the totally ordered set of values will be: 
 $\Lambda ^{(d)}=(\cS _1 \times {\Lambda}^{(d-1)})\cup \cS _2 \cup \{ \infty  _d\}$, where 
 $\cS _1 \times {\Lambda}^{(d-1)}$ is lexicographically ordered, any element of $\cS _2$ is larger than any element of $\cS _1 \times {\Lambda}^{(d-1)}$ and 
 $\infty _d$ is the largest element.
 
Consider first a single basic object $B_0$. Let $M:=\ma (t_0)$ and $M(1)$ the union of the one-codimensional components of $M$. Given $x \in \sg(B_0)$, necessarily we have $\o _0(x) > 0$ and there 
  are three cases.  
(a) $x \in M(1)$. Then set $g_0(x)=(t_0(x), {\infty}_{d-1})$. 
(b) $x \in M \setminus M(1)$. Then take a neighborhood $U$ of $x$ (in $W$) such that the basic object $B_0^*$ above  is defined 
 and the function 
 $\widetilde {g_0}:Z_0 \to {\Lambda}^{(d-1)}$ as in (iv) above (with $s = j=0$)
 Then set   $g_0(x)=(t_0(x),\widetilde {g_0}(x))$. This value is independent of the choices made. (c) $x \notin M$
 Then set $g_0(x)=(t_0(x),\infty _{d-1})$.

Assume now  that resolutions functions $g_i$, $i=0, \ldots, j-1$ have been defined, determining  centers $C_i = \ma \,{(g_i)}, i=0, \ldots, j-1$\ldots, leading to a permissible sequence $B_0 \leftarrow \cdots \leftarrow B_j$,  $
 B_i=(W_i,I_i,b,E_i)$, $i=0, \ldots, j$, $j \ge 0$. We assume that if $B_{j-1}$ is not a monomial object, then this is a $t$-sequence.
  
  There are two basic cases: (a) $\max (\o _j) =0$, (b) $\max (\o _j) > 0$. 
  %Let $x \in \sg (B_j)$. There are two cases: (a) $\o _j(x) =0$, (b) $\o _j(x) > 0$.
  
   In case (a), $B_j$ is monomial, for $x \in \sg(B_j)$ let $\Gamma _j$ be its $\Gamma$-function and set $g_j(x):=\Gamma _j (x)$. In case (b), let $M_1(j)$ denote the union of the one-codimensional components of $M(j):= \ma (t_j)$ and $H$ the exceptional divisor of the blowing-up (with center $C_i$) $W_{j-1} \leftarrow W_j$.  For $x \in \sg (B_j)$ there are three  sub-cases: 
   
   ($b_1$) $x \in M_1(t_j) \cap H$. Then we set  $g_j(x)=(t_j(x), {\infty}_{d-1})$                                                         %_d$ if $x \in M_1(j)$ and $g_j(x) = (t_j(x), \infty _{d-1})$ if $x \in \sg(B_j)$ but $x \notin M_1(j)$.
   
   ($b_2$) $x \in (M(j) \setminus M_1(j))\cap H$. Then consider the smallest index $s$ such that $t_s(x_s)=t_j(x)$, where $x_s$ is the image of $x$ in $\sg (B_s)$ induced by the sequence above. Using 
  the construction of ($\beta$), applied to $x_s \in W_s$, we obtain resolution functions of $B^*_s$, $\widetilde{g_{0}}, \widetilde{g_{1}}, \dots$. Then it makes sense to take $\widetilde{g_{j-s}}(x)$, and it can be proved that this value is well-defined. We set 
 $g_j(x)=(t_j(x),\widetilde{g_{j-s}}(x)) \in \cS _1 \times {\Lambda}^{(d-1)}$. 
 
   ($b_3$)  $ x \notin H$. Then, if $x'$ is the image of $x$ in $W_{j-1}$, set $g_j(x)=g_{j-1}(x')$

 With this definition, if $B_j$ is not monomial  then the center $C_j={\mathrm {Max}}(g_j)$ is contained in ${\mathrm {Max}}(t_j)$.  
 
 It can be proved that the sequence $\{  {\mathrm {max}}(g_j)       \}$ is strictly decreasing, which  leads to a resolution of $B$ (\cite{BEV} or \cite{EV}).
\end{voi} 
\begin{voi}
\label{V:ag8}
We shall better explain some details of this process, specially the crucial inductive step of \ref{V:ag7}. For this, we must review some other concepts. 

\smallskip

($\alpha$) {\it Adapted hypersurfaces, nice objects}. 
 A hypersurface $Z \subset W$ 
is {\it adapted} to $B$ (or  $Z$ is {\it $B$-adapted}) if
 the following conditions hold:
 (A1) $I(Z) \subseteq \Delta ^{b-1} (I)$ (an inclusion of sheaves of
${\cO}_W$-ideals),
 (A2) $Z$ is transversal to $E$ (see \ref{V:is1.0}). If, moreover: 
(A3) ``Whenever $D$ (resp. $D'$) is an irreducible component of $Z$ (resp. of
$V({\Delta}^{b-1}(I))$) then $D \not= D'$'' holds, we say that $Z$ is inductive. 

\smallskip
We shall say that $B$ (a basic object in $\cS$) is {\it nice} if either
Sing($B$) is empty or $B$ admits an adapted hypersurface. An adapted hypersurface is necessarily regular.

\smallskip

($\beta$) {\it Inductive objects} 
If  $B$ is a nice basic object 
  we define a $W$-ideal, called the {\it coefficient ideal} and denoted by
${\cC}(I)$, as follows:
 $${\cC}(I) : = \sum_{i=0}^{b-1} \, [{\Delta}^i(I)]^{b!/b-i}  $$
 If $Z$ a $B$-inductive hypersurface, then 
 the {\it coefficient ideal relative to $Z$}, or the 
$Z$-{\it coefficient ideal}, denoted by
${\cC}(I,Z)$, is the restriction of  
${\cC}(I)$ to $Z$. This is a never-zero $Z$-ideal.

The basic object 
$B_Z := (Z , {\cC}(I,Z), b!, E_Z)$, 
$E_Z = ( H_1 \cap Z,\ldots, H_m \cap Z )$, 
is called the {\it inductive object of $B$, relative to the inductive hypersurface $Z$}.
\smallskip

$(\gamma)$ Consider a $t$-permissible sequence of basic objects and transformations 
$B_0 \leftarrow \cdots \leftarrow B_r$ (where $(B_j=(W_j,I_j, b,E_j))$ and a point $w \in W_r$ such that, near $w$, 
$M_w=\{ x : t_r(x)=t_r(w)  \}$ has codimension (in $W_r$) $>1$. Then there is a nice object $B_r  ''$, defined on a suitable neighborhood $U$ of $w$ in $W_r$, admitting an inductive hypersurface $Z$, such that $\sg (B'' _r) = M_w$. The definition and properties of this object are presented in \cite{EV}, 9.5 or, in a more general context, in section 8 of this paper.

($\delta$) {\it Homogenized ideals and objects}. 
If $W$ is a variety, a $W$-{\it weighted ideal} is a pair $(I,b)$, where $I \subset {\cO}_W$ is $W$-ideal  and $b$ is a non-negative integer. 

 The associated homogenized ideal of $(I,b)$ is the the $W$-ideal 
$$(1) \qquad \cH (I,b)= I + \Delta (I) T(I) + \cdots +  \Delta ^i(I) T(I) ^i + \cdots + \Delta ^{b-1}(I) T(I)^{b-1}$$ 
where we have written $T(I):= {\Delta}^{b-1}(I)$ (see \cite{W}, section 2).

If $B=(W , I,b, E)$ is a  basic object in $\cS$, the basic object  ${\cH}(B):=(W, \cH (I,b), b,E)$ is the homogenized object associated to $B$. This is discussed in detail (in a more general setting) in section \ref{S:HI}.
\end{voi}
\begin{voi}
\label{V:ag9} 
We return to the discussion of the VW-resolution algorithm. In the notation of \ref{V:ag7}, we take  as the open set $U$ a neighborhood of $w$ over which the  nice object $B_s''$ of \ref{V:ag8} ($\gamma$) is defined. Then its associated homogenized object  $\cH ({B_s}''):=(HB_s'')$  is again nice and it admits an adapted hypersurface $Z_s$ containing $w$, defined on $U$. This will be the $Z_s$ of 
 \ref{V:ag7}.  Our object $B^*_s$ of \ref{V:ag7} ($\beta$) will be  $({\cH}(B''_s))_{Z_s}$. 
 
 In \ref{V:an2.0} of this paper we shall check that properties (i)-(iv) of \ref{V:ag7} ($\beta$) are valid.
\end{voi}
 \begin{voi}
 \label{V:ag10} The algorithm discussed in \cite{EV} or \cite{BEV} is very similar to the VW-algorithm just described. It proceeds as in \ref{V:ag7}, the only difference is that in the inductive step the auxiliary object $B^*_s$  now is 
 $(B_s'')_{Z_s}$ rather than $({\cH}(B''_s))_{Z_s}$. This looks simpler, however with this approach it is more difficult to check that the process is independent of the choice of the adapted hypersurfaces $Z_s$ we choose (as mentioned in the introduction, see also \ref{V:ae5.1}).
 
 Actually, it can be proved that both the VW-algorithm and that of \cite{EV}  are  the same, in the sense that the resolution functions in either case coincide. We shall  not check this fact in this paper.
 
 We want to emphasize that we do not claim that the VW-algorithm (or that  of \cite{EV}) and that of \cite{W} are the same. They are not, as is proved in 
 \cite{BMM} Remark 6.16. The VW-algorithm is not that of \cite{W}, but essentially that of \cite{EV}, with some details changed by using some concepts from \cite{W}.
\end{voi}
\begin{voi}
\label{V:ag10.1} The algorithms just discussed enjoy some important additional properties. For instance, they are functorial with respect to etale morphisms $W' \to W$ (where $ W=us(B)$, $B$ a basic object) and extension of the base field. For the precise statements and a proof see \cite{EV} or \cite{BEV}.
\end{voi}\begin{voi}
 \label{V:ag11} In the following sections (3 through 8) we attempt to generalize, as much as possible, the theory just described to the case where we do not work over a base field $k$ but rather a (suitable) artinian ring $A$. Our point of view is to regard a basic object defined over $A$ (see \ref{D:bo}) as an infinitesimal family of basic objects, or an infinitesimal deformation of the only (closed) fiber that we have. In our approach, when adapting a  given notion is adapted to this situation, the intuitive idea is that it applies to the fiber and it ``spreads well'' along the (infinitesimal) parameter space $\Spec (A)$. This won't be always possible. When the whole resolution process of the fiber can be extended to an object over $A$ we'll say that we achieved algorithmic equiresolution. 
 The extension of this theory to the mentioned relative situation sometimes is straightforward, sometimes not. In general, when the translation to this relative situation is simple we shall omit the proofs or certain details.
 
 Concerning the algorithm of resolution, working over an artinian ring $A$ the resolution functions $g_i$ of \ref{V:ag5} do not seem very useful. Indeed, these functions  are really defined on the underlying topological space $W$ (notation of \ref{V:ag2}), or a subspace thereof. If $W$ is a scheme over $A$ (artinian), $W$ and the fiber share the same topological space. Hence, for our purposes the VW-algorithm should be regarded as a rule selecting the appropriate centers $C_0, C_1, \ldots $ involved in the algorithmic resolution sequence of \ref{V:ag5}.  This is our point of view.
 
\end{voi}

\section{Basic Notions}\label{S:BN}

\begin{voi}
\label{V:bn1.0}
In general, we shall use the notation and terminology introduced in \ref{V:ag1}. In addition,   
 throughout, the symbol ${\mathcal A}$ will denote the collection of artinian local rings $(A,M)$ such that the residue field $k=A/{M}$ has characteristic zero. Since such a ring is necessarily complete and equicharacteristic, $A$ will contain a (unique) {\it field of representatives}, i.e.  a subfield mapping onto $k$ via the canonical homomorphism $A \to A/{M}=k$. Thus, $A$ is automatically a $k$-algebra.  If $A \in {\mathcal A}$ we'll usually write $S=\Spec (A)$.
\end{voi}
\begin{voi}
\label{V:bn2.0}
 If $p:W \to S$ is a smooth morphism, $S=\Spec(A)$, the fiber (or closed fiber) of $p$ is the fiber over the only point of $S$ (regarded as a closed subscheme of $S$).  It is denoted it by $W^{(0)}$. Then $W^{(0)}=W_{red}$ and this  is an algebraic variety, smooth over the field $A/M$. If $I \subset {\cO}_{W}$ is a $W$-ideal, 
 $I^{(0)}:=I{\cO}_{W^{(0)}}$ is called the fiber of $I$.

\smallskip

A hypersurface on $W$ over $S$, or an $S$-hypersurface, or an $A$-hypersurface, is a positive Cartier divisor $H$, flat over $S$, inducing over the fiber $W^{(0)}$  a regular codimension one subscheme $H^{(0)}$. Thus, the ideal $I(H)_y$ (the stalk at $y$ of the ideal sheaf defining  $H \subset W$) is defined by a single element, inducing at ${\cO}_{W^{(0)},y}$ an element $a$  of order one. This element $a$ is a non-zero divisor of ${\cO}_{W,x}$. This is a consequence of Theorem 23.2 (page 179) in \cite{M} (or use \ref{L:regse}, witn $n=1$). It can be proved that the naturally induced projection morphism $H \to S$ is smooth (see \ref{L:regse}).
\end{voi}
\begin{voi} 
\label{V:bn3.0}
Let $p:W \to S$  be as in \ref{V:bn1.0}, $w$ a point of $W$, $R={\cO}_{W,w}$,
$R'={\cO}_{W^{(0)},w}$ (the local ring of the special fiber at $w$, which is regular.) A system of elements $a_1, \ldots, a_n$ in $R$ is called a regular system of parameters of $R$ relative to $p$, or simply (if the morphism $p$ is clear) an $A$-regular (or $S$-regular) system of parameters  if the induced elements  $a_1^{(0)}, \ldots,  a_n^{(0)}$ in $R'$ form a regular system of parameters, in the usual sense. Elements $a_1, \ldots, a_r$ of $R$ are
{\it part of an $A$-regular system of parameters}, or a
{\it partial  $A$-regular system of parameters}, if they are contained in an  $A$-regular system of parameters $a_1, \ldots, a_n$, $r \leq n$, of $R$. Then necessarily 
$a_1, \ldots, a_n$ is a regular sequence in the local ring ${\cO}_{W,w}$ (see \ref{L:regse}, or use the result in 
\cite{M} cited at the end of \ref{V:bn2.0} and induction).

\end{voi}
\begin{voi} 
\label{V:bn4.0}
An collection $E = \{H_1, \ldots, H_m \}$ of  $S$-hypersurfaces of $W$ is said to have normal crossings if for all points $w \in W$ there is an $A$-regular system of parameters $a_1, \ldots, a_n$ of
${\cO}_{W,w}$ such that $H_1 \cup \cdots \cup H_m$ is defined at $w$ by a product of elements $a_i$, without multiple factors. This product might be empty, hence $=1$ (this happens when no hypersurface $H_i$ contains $w$.)

\smallskip
A subscheme $C$ of $W$ (with defining sheaf of ideals $I(C)$) is said to have normal crossings with $E$ over $S$ (or relative to $S$) if
$E$ has normal crossings and for all points $w \in C$, there 
is an $A$-regular system of parameters  
$a_1, \ldots, a_n$ in ${\cO}_{W,w}$ such that
 the stalk $I(C)_w$ is generated by ($a_1, \ldots, a_r) {\cO}_{W,w}$ (for some $r \leq n$) and for each
$H_i$ containing $w$, the ideal  $I(H_i)_w$ is generated by a suitable element $a_{j}$, $j \in \{1, \ldots, n \}$.

It can be proved  that the induced projection $C \to S$ is smooth, and  the
blowing-up $W_1$ of $W$ with center $C$ is also $S$-smooth (see \ref{L:regse}).

\end{voi}
\begin{voi} 
\label{V:bn5.0}

An $S$-pair (or $A$-pair) is an ordered pair $(p: W \to S, E)$, with $p$ smooth and $E = ( H_1, \ldots, H_m )$ an ordered $m$-tuple  of distinct $S$-hypersurfaces on $W$ having normal crossings. We call $W$ the {\it scheme of the pair} $(p,E)$. When $A$ is  a field which should be clear from the context (e.g., when taking a closed fiber), often we shall write $(W,E)$ for the pair.

\smallskip
A  {\it permissible center}
for the pair $(p, E)$ (as above) is a subscheme $C$ of $W$,  having normal crossings with $E$. Then $C$ is automatically smooth over $S$ (\ref{V:bn4.0}).
\end{voi}
\begin{voi} 
\label{V:bn6.0}
We use the notation of \ref{V:ag1}. Let $B=(p:W \to S, E)$ be a pair over  $S$, $I$ a 
 never zero $W$-ideal and 
 $C \subset W$ be an irreducible $S$-permissible center for the pair $(p,E)$, defined by the $W$-ideal
$J \subset {\cO}_W$,  with generic point  $y \in C$.

We shall say that the order of $I$ along
$C$  is $\geq m$,
written
$\nu (I,C) \geq m$, if (in the local ring ${\cO}_{W,y}$) we have
$I_y \subseteq (J_y) ^m$.

Finally, we write
$\nu (I,C) = m$ if $m$ is the largest integer such that $\nu (I,C) \geq m$.

\end{voi}
\begin{exa}
\label{E:or}
Let 
$W={\bf A}^2_A=\Spec A[x,y]$ with $A=k[\epsilon]=k[t]/(t^2)$ (and, say, $k= {\bf C}$), $S=\Spec A$, $W \to S$ the natural projection,
$I=(\epsilon x + y^2 +x^3)A[x,y]$, $E=\emptyset$, $C=V(x,y)$, $w$ the ``origin'', i.e. the maximal ideal $(\epsilon,x,y)A[x,y],~ b=2$. Then $C$ is $S$-permissible and $I_w$ has order $2$, but $\nu (I,C)=1$.
\end{exa}
\begin{voi}
\label{V:bn7.0} 

Let $W \to S = {\rm Spec}\,A$ (with $A \in {\mathcal A}$) be a smooth morphism, $w \in W$, $a_1, \ldots a_r$ elements of $R:={\cO}_{W,w}$ which are part of an $A$-regular system of parameters, 
${J}=(a_1, \ldots, a_r)R$, $R'=R/J$. Then, 
the completion $R^*$ of $R$ with respect to $J$ is isomorphic to the power series ring in $r$ variables
$ R'[[x_1, \ldots, x_r]]$, where $x_i$ corresponds to $a_i$, for all $i$ (see \ref{P:preries}). 

\smallskip

Now, with the notation of \ref{V:bn6.0},
let $w$ be the generic point of the center $C$ and ${J}=I(C)_w$. Then, by 
definition of permissible center, there is a sequence    $a_1, \ldots , a_r$ which is  part of an $A$-regular system of parameters (see \ref{V:bn3.0}) and generates $J$. Then, again with the notation of \ref{V:bn6.0}, 
$\nu (I,C) = m$ if and only if
 each $f \in I_w$, when regarded  as an element of the completion 
$R^*=R'[[x_1, \ldots, x_n]]$,  $R'={\cO}_{C,w}$, can be written as a power series in
$x_1, \ldots, x_r$ of order $\geq m$, with coefficients in 
$ R'$, and for some $f$ that order will be {\it exactly} $m$.
\end{voi}
\begin{voi}
\label{V:bn8.0}
The notions of \ref{V:bn6.0} can be described  in a more global way as follows. (cf. \cite{EV} or \cite{BEV}, section 13).  We retain the assumptions and notation of \ref{V:bn6.0}.
Let $J$ be a $W$-ideal. Define $\Delta (J/S):=J + {\cF}_{d-1}(\Omega _{Y/S})$, where $Y:={\mathcal V}(J)$ and
${\cF}_{d-1}$ denotes the $(d-1)$-Fitting ideal. Then, using well known properties of the objects involved and the remarks in \ref{V:bn7.0}, the following facts are easily verified.

If $w \in C$ is a closed point the stalk $I(C)_w$ is generated by an $A$-regular sequence  
$a_1, \ldots, a_n$ of ${\cO}_{W,w}$, let  $R^*=A'[[x_1, \ldots,x_n]]$ (with $A'={\cO}_{W,w}/(a_1,\ldots,a_n)$) denote the completion of $\cO _{W,w}$ with respect to the ideal $(a_1, \ldots, a_n)$. Then  
$ {\Delta (J/S)}_w R^*$ is the ideal generated by 
 the elements $f \in J_w$ and the partials $\partial f / \partial x_i$,
$i=1, \ldots, n$, for all $f \in J_w$. 

We may iterate this construction, getting $W$-ideals
 $\Delta ^i (J/S)$, $i=1, 2, \ldots$, such that 
$ \Delta ^i (J/S)R^*$ is the ideal generated by elements of $I$ and their partial derivatives of order $\leq i$.
\end{voi}
The connection between this object and the notion of \ref{V:bn6.0} is given in the following result.
\begin{pro}
\label{P:ekiv}
If $(W \to S,E)$ is an $S$-pair, $C$ is an irreducible permissible center, with generic point $y$ and defining ideal $I(C)=J$, and $I$ a never-zero $W$-ideal, then the following statements are equivalent:

(i) $I_y \subseteq (J_y)^b$ (i.e., $\nu(I,C)\ge b$).

(ii) $I_x \subseteq (J_x)^b$ for every closed point $x$ in a dense open subset of $C$ 
 (i.e.,  a non-empty subset of $C$).

(iii) ${\Delta}^{b-1}(I/S)_x \subseteq J_x$ for every closed point $x$ in a dense open subset of $C$ (say, for $x$ in $U \cap C$, $U$  a  suitable open set of $W$.)

(iv) ${\Delta}^{b-1}(I/S) \subseteq J$
\end{pro}
\begin{proof}

The other implications being well-known, we discuss  the implication $ (iii) \Rightarrow (iv)$. Recall the following basic facts. Here, if $Y$ is a locally closed subscheme of a scheme $W$, $cl(Y)$ denotes the scheme-theoretic closure of $Y$ in $W$.

\smallskip
{\noindent
(a) If $Y$ is a closed subscheme of $W$, $U$ an open of $W$, $Y' = Y \cap U$ then 
$cl(Y') \subseteq Y$.}

\smallskip
{\noindent
(b) In the notation of (a) if, moreover, the closed subscheme $Y$ is irreducible, has no embedded points (e.g., if it is C-M, that is a Cohen-Macaulay scheme) and $U$ and $Y$ have at least one point in common, then $ cl(Y') = Y $ 
}

\smallskip
Now we prove   $ (iii) \Rightarrow (iv)$ as follows.
Let $Y = {\mathcal V}({\Delta}^{b-1}(I/S))$, note that 
${\mathcal V}(I(C))=C$. Observe, moreover, that $W$ is C-M (use the fact that $W$ is smooth over $ S$ which is C-M, now use  the Corollary to Thm. 23.3, page 181 of \cite{M}). Since $C$ is locally defined by a regular sequence, $C$ is also C-M, in particular with no embedded points. Let  $U$ be an open set in $W$ such that the open subset of $C$ in (iii) is of the form $C'=C \cap U$. Let $Y' = Y \cap U $, and $C'= C \cap U$. The assumption of (iii) implies that 
the restriction of 
${\Delta}^{b-1}(I/S)$ to $U$ is contained in the restriction of $I(C)$ to $U$ 
(because  the closed points are dense,  $W_{red}$ being an algebraic scheme over a field). This implies, taking ${\mathcal V}$, that $C' \subseteq Y'$. Taking closure we get $cl(C')\subseteq cl(Y')$. But, by (a) and (b), $cl(C')=C$ and $cl(Y') \subseteq Y$. Thus, $C \subseteq Y$, which is equivalent to (iv).

\end{proof}
\begin{voi}
\label{V:bn9.0}
Note that the notion of $S$-hypersurface given in \ref{V:bn2.0} may be equivalently expressed as follows. A subscheme $H$ of $W$ is an $S$-hypersurface if, for each 
$x \in H$, $I(H)_x$ is generated by an element $a \in {\cO}_{W,x}$ whose image in 
${\cO}_{W^{(0)},x}$ is in $M \backslash {M}^2 $ (with $M= r({\cO}_{W^{(0)},x})$). Indeed, as remarked in 
\ref{V:bn2.0}, the definition given there implies this notion. Conversely, if $H \subset W$ is as indicated, then $H$ is flat over $S$ (see the end of \ref{V:bn4.0}) and, using \ref{V:bn7.0}, the completion of $R={\cO}_{W,x}$ with respect to $(a)$ is isomorphic to $A'[[a]]$ (with $a$ analytically independent over $A'\cong A/(a)$), showing that $a$ is not a zero divisor in ${\cO}_{W,x}$. Hence, $H$ is a relative Cartier divisor over $S$.
\end{voi}
\begin{voi}
\label{V:bn9.1}
Let $(p:W \to S, E)$ be an $S$-pair, $C$ an $S$-permissible center of it. 
If $W'$ is the blowing up of $W$ with center $C$, $p':W' \to S$  the induced morphism  and  $E'=
(H'_1, \ldots, H'_m , H'_{m+1})$,
where $H'_i$ is the strict transform of $H_i$, for $i=1, \ldots, m$ and $H'_{m+1}$ is the exceptional divisor, then  $(p',E')$ is a new $S$-pair, called the {\it transform} of $(p,E)$ with center $C$. The fact that $H'_{m+1}$ is an $A$-hypersurface (in the sense of \ref{V:bn2.0}) is seen by using \ref{V:bn9.0}, the other points are easy. The closed fiber of $(p',E')$ may be identified to the transform of $(W^{(0)},E^{(0)})$ with center $C^{(0)}:=C \cap W^{(0)}$ (use \ref{P:bloui}).

\end{voi}

\section{Basic objects}\label{S:BO}

\begin{Def}
\label{D:bo}
A basic object over $S$ is a four-tuple $B=(p:W \to S, {I}, b, E)$, where 
$(p, E)$ is an $S$-pair, ${I}$ a never-zero $W$-ideal (\ref{V:ag1})
and $b$ is a non-negative integer.  

The pair $(p,E)$ is   the {\it underlying pair} of the basic object $B$ and $W$ is the {\it underlying scheme of the basic object $B$}, denoted by $us(B)$. The integer $b$ is the {\it index} of the $A$-basic object $B$. The dimension of $B$, $\dim\,(B)$, is the dimension of the scheme $us(B)$.

\smallskip
There is a naturally defined notion of the {\it  fiber} of a basic object $B$, usually denoted by $B^{(0)}= (W^{(0)},{I}^{(0)}, b, E^{(0)})$.
We let $\sg \, (B^{(0)}) $ denote the closed set 
$\{ w \in W^{(0)}: \nu _w ({I}^{(0)}) \geq b \}$ and 
$\sg \, (B) :=  \sg \, ({B}^{(0)}) $.
\end{Def}
\begin{voi}
\label{V:bo1.0}
If $B=(p:W \to S, I,b, E)$ is an $S$-basic object, a subscheme $C \subseteq W$ is a {\it permissible center} for $B$ (or $B$-permissible, or just a $B$-center)  if it is a permissible center for its underlying $S$-pair $(p, E)$  and moreover, for each irreducible component $C'$ of $C$ we have, 
letting $C'^{(0)}$ be the  fiber of the natural projection $C' \to S$, that  
$$ (1) \quad \nu (I,C') = \nu (I^{(0)},C'^{(0)}) \geq b $$ 

{\noindent with the notation of \ref{V:bn6.0}.}
\end{voi}

Note that if $A=k$ (a field) then a $B$-permissible center is a permissible center in the sense of \cite{EV}, \cite{BEV} or \ref{V:ag4}. A $B$-permissible center induces a permissible center on the fiber $B^{(0)}$.
\begin{exa}
\label{E:nohay} It is possible to have $A$-basic objects without any permissible centers. This is a simple example.

Let $B=(W \to S,I,2,\emptyset)$, where $S=\Spec (A)$, with $A=k[\epsilon]$ ($k$ a characteristic zero field,   ${\epsilon}^2=0$, ), $W=\Spec (A[x])$ ($x$ an indeterminate ) $I=(x^2,\epsilon x)$. Were $C$ a $B$-permissible center, it should induce the only permissible center of the fiber $B^{(0)}$, namely the origin ${\mathcal V}(x)$. Hence, the ideal $I(C)$ is generated by an element of the form $y=x+h$, $h \in (\epsilon)A[x]$ and  
$A[x]=A[y]$. Then 
$I=
((y-h)^2,{\epsilon}(y-h))
=(y^2,{\epsilon}y)$ and $C={\mathcal V}(y)$.  Now it is clear that 
$\nu(I,C)=1 $ while $\nu(I^{(0)},C^{(0)})=2 $, impossible were $C$  a $B$-permissible center.
\end{exa}
\begin{pro}
\label{P:ceng}
(a) Let $C$ be an irreducible $B$-center, with $B$ as in \ref{D:bo}, $Z \subset W$ an irreducible $S$-hypersurface, both $C$ and $Z$ having a common  point $y$. Assume $I(Z)_y \subseteq I(C)_y$. Then,  
$C \subset Z$.

\smallskip

(b) Assume $C, \,C'$ are irreducible $B$-centers, having a common point $y$, and $I(C)_y = I(C')_y$. Then $C=C'$.

\end{pro}

\begin{proof} It is similar to that of \ref{P:ekiv}. (a) Let $cl$ indicate scheme-theoretic closure in $W$. The hypothesis implies that there is an open dense neighborhood $U$ of $y$ in $W$ such that, letting  
$C_1:=C_{|U}, ~ Z_1:=Z_{|U}$, we get   $C_1 \subseteq Z_1$. Hence    $cl(C_1) \subseteq cl(Z_1)$. But since both $C$ and $Z$ are Cohen-Macaulay, hence without non-trivial embedded points, necessarily $cl(C_1) = C, ~ cl (Z_1) = Z$.

\smallskip

(b) Use the same argument as in (a), with $C'$ rather than $Z$, to get 
$C \subseteq C'$. Similarly we get $C' \subseteq C$.
\end{proof}
\begin{voi}
\label{V:bo2.0}  
If $B$ is an $A$-basic object,  $C$ an $A$-permissible center for $B$ and $\pi: W_1 \to W$ the blowing-up of $W$ with center $C$, the sheaf ${I}$ induces several important sheaves of ${\cO}_{W_1}$-ideals. Namely, we have:

\medskip
 \noindent(i) $I'_1={I}{\cO}_{W_1}$ (the {\it total transform} of $I$ to $W_1$),  

\smallskip
\noindent(ii) $I_1:= {\mathcal E}^{-b} I'_1$, where ${\mathcal E}$ defines the exceptional divisor (the {\it controlled transform} of $I$ to $W_1$),

\noindent(iii) $\overline {{I}_1} := {\mathcal E}^{-a}{I'}_1$, with $a$ as large as possible (the {\it proper transform} of ${I}$.) (If $C$ is not connected, the exponent $a$ is constant along $p^{-1}(C')$, for each connected component $C'$ of $C$, but not necessarily globally constant).

\end{voi}

\begin{voi}
\label{V:bo3.0}
Given a $A$-basic object $B$ (as in \ref{V:bo1.0}), the 4-tuple  
${B}_1 := (W_1 \to S,{I}_1,b,E_1)$ (where $I_1$ is the controlled transform of $I$ and 
$E_1$ is defined as in \ref{V:bn5.0}) is a new $A$-basic object,  called the {\it transform} of the basic object ${B}$ with center $C$ (or at $C$). The process of replacing a basic object $B$ by its transform $B_1$ (with a permissible center $C$, as above) will be called the {\it transformation} of $B$ with center $C$, indicated by 
$B \leftarrow B_1$, if $C$ is clear. Sometimes we shall write $\cT (B,C)$ to denote the transform of $B$ with center $C$.

Using the statement at the end of \ref{V:bn9.1}, one may verify that the transform of an $S$-basic object with a $B$-permissible center over $S$ induces (by taking  fibers) the transform of the  fiber $B^{(0)}$ with center $C^{(0)}:=C \cap W^{(0)}$ (notation as in \ref{V:bn9.1}).
 Moreover, 
${\overline {{I}_1}}^{(0)}=
 \overline {{I^{(0)}}_1}$. 
 
 When $A=k=$ a field, the notions of \ref{V:bo2.0} and \ref{V:bo3.0} reduce themselves to those of \ref{V:ag4}.
\end{voi}

\begin{voi}
\label{V:bo4.0}
A permissible sequence of transformations over $S$ is one of the form:
$$ (1) \qquad B_0 \gets \cdots \gets B_r$$ 
where each arrow $B_{i} \leftarrow B_{i+1} $ stands for a transformation of basic objects, with a $B_i$-permissible center, say $C_i \subset W_i=us(B_i)$. Note that, by taking  fibers, such a sequence induces one:
$$\qquad B_0^{(0)} \gets \cdots \gets B_r^{(0)}$$
where the $i$-th arrow is the transformation with permissible center $C_{i-1}^{(0)}\subset us(B_i^{(0)})$, where $C_{i-1}^{(0)}$ is  the restriction of $C_{i-1}$ to the  fiber. 
\smallskip
The sequence obtained from (1) by deleting $B_{j+1}, \ldots, B_r$ and the corresponding arrows is called the $j$-{\it truncation}  of the sequence (1).

\smallskip

The sequence (1) is called an {\it equiresolution} of the basic object $B_0$ if ${\mathrm Sing} \, (B_r) = \emptyset$ 
\end{voi}
\begin{voi}
\label{bo5.0}
{\it  W-equivalence.} Let $B=(W \to S, I, b, E)$ and $B'=(W \to S, J, c, E')$ be basic objects ($S =\sp \, (A) $, $A$ as in \ref{V:bn1.0}) be basic objects over $S$. We shall say that they are {\emph {pre-equivalent}} if the following conditions hold: (0) $C \subset W$ is a $B$-permissible center if and only if it is a $B'$-permissible center, (1) If $B_1$ (resp. $B'_1$) denotes the transform of $B$ (resp. $B'$) with center $C$, then $C_1$ is a $B_1$-permissible center if and only if it is a $B'_1$ permissible center, consider the transforms with center $C_1$, ..., ($n$) if we repeat this process $n$ times ($n$ any natural number), obtaining basic objects $B_n$ and $B'_n$ respectively, $C_n$ is a $B_n$-permissible center if and only if it is a $B'_n$-permissible center. 

We say that $B$ and $B'$ are {\emph {W-equivalent}} if, in addition, the special fibers 
$B^{(0)}$ and $B'^{(0)}$ are also pre-equivalent. 

In case where the base is a field, pre-equivalence is the same as equivalence. This notion is what in \cite{W} is called equivalence. The notion of equivalence used in  
\cite{EV} requires a further condition, for that reason we use our terminology of W-equivalence.

\end{voi}
\begin{exa} \label{E:eq}
Let $B=(W \to S,I,b,2, \emptyset)$ and
$B=(W \to S,I',b,2, \emptyset)$, where $S=\sp (A)$, $A=k[\epsilon]$, $k$ a field,
$\epsilon ^{2}=0$, $W=\Spec \, (A[x])$, $I=(x^2+{\epsilon}x) $,
$I'=(x^5, {\epsilon}x)$. Then, the $A$-basic objects $B$ and $B'$ are pre-equivalent but not equivalent.
\end{exa}

\section{Some resolution tools}
\label{S:DT}
In this section we study, in the context of $A$-basic objects ($A \in \cA$), some notions discussed in \ref{V:ag6} (where we worked over a field).

\begin{voi}
\label{V:dt2.0} {\it $\o$- and t-permissible centers}.
 Consider a  permissible sequence of $A$-basic objects 
$$\qquad (1) \qquad   B_0 \leftarrow B_1 \leftarrow \cdots \leftarrow B_r$$
where $B_{i+1}=\cT(B_i)$, $C_i$ a $B_i$-permissible center, for all $i$. 
We want to define, recursively on the length $r$, the notion ``(1) is an $\o$-permissible sequence, or an $\omega$-sequence'' (see 
\ref{V:ag6}, where $A=k$, a field). If $r=0$, by definition (1) is always $\omega$-permissible. Assume we know, by induction, what is an $\omega$-permissible sequence when the length is $r$, and that  this implies that the induced sequence of special fibers 
$$(2) \qquad  B_0^{(0)} \leftarrow B_1^{(0)} \leftarrow \cdots \leftarrow B_r^{(0)}$$
is such that each center $C_i^{(0)}$ that was used satisfies 
$C_i^{(0)} \subseteq {\mathrm{Max}}({\omega}_i) $ (i.e., is $\o$-permissible). Given a sequence of length $r+1$, we say it is $\o$-permissible if its $r$-truncation is so (then we have a function 
${\omega}_{r}$ with domain $\sg (B_r^{(0)}):= \sg(B_r)$. Let $b_r/b := 
{\mathrm max}(\o _r)$) and if the center $C_r$ used to obtain $B_{r+1}$ satisfies  
$\nu(\overline{I_r},C_r) = \nu(\overline{I_r^{(0)}},{C_r}^{(0)})= b_r$ (where, e.g., 
$\overline{I_r}$ demotes proper transform, see \ref{V:bo2.0}).  Then the induced sequence of  fibers satisfies 
$C_i^{(0)} \subseteq {\mathrm{Max}}({\omega}_i) $ for all $i$, with notation as above. Given a $\o$-permissible sequence (1), each center $C_j$ used in it is said to be $\o$-permissible for $B_j$.
 If (1) is an $\o$-permissible sequence, by the function $\o _j$ of (1) we mean the function $\o _j$ of the corresponding sequence of  fibers.

\smallskip

Now we study $t$-permissible sequences. Consider a sequence of $A$-basic objects and transformations (1), where we write  $B_j=(W_j \to S, I_j, b, E)$, for all $j$. We shall define, by induction on the length $r$, what it means that (1) is $t$-permissible.

If $r=0$ (i.e. there is just one basic object), the sequence (reduced to one object) is $t$-permissible.

Next assume the notion of $t$-permissible center is defined, by induction, if the sequence has length $\leq r$, in such a way that it  induces a sequence  of  fibers which is $t$-permissible, in the sense of \ref{V:ag7}. We declare a sequence of length $r+1$ $t$-permissible if the following conditions (a) and (b) hold: (a) the $r$-truncation of (1) is $t$-permissible . Hence 
 by looking at  fibers we have functions $t_i$, $i=1, \ldots, r$ satisfying 
${\rm max} ~(t_i) \geq {\rm max}~(t_{i+1}) $, $i=0, \ldots , r$.  
 Let $s$ be the smallest index such that 
${\rm max}\,(\o _s) = {\rm max} (\o _r)$. Let $E_r^{-}$ consist of the hypersurfaces in $E_r$ which are strict transforms of those in $E_s$ and $(\bbar{\o},\bbar{n})= {\mathrm  max}(t_r)$. Then we demand: (b) any component $C$ of the center $C_r$ used to obtain $B_{r+1}$ satisfies:  
 $\nu (\bbar{I_r},C) =  \nu (\bbar{I^{(0)}_r},C^{(0)})= b_r$ and  for each closed point $y \in C$,  
 the number of hypersurfaces in $E^{-}_r$ containing $y$ is as equal to $\bbar{n}$. If (1) is 
  a $t$-permissible sequence, each center $C_j$ used in it is said to be $t$-permissible for $B_j$, and by 
  the function $t _j$ of (1) we mean the function $t _j$ of the corresponding sequence of  fibers. 

\end{voi}
\begin{voi}
\label{V:dt3.0}

We shall consider, in the present situation, the analogue of certain numerical invariants introduced in \cite{V1} (see also \cite{V2}, where it is proved that the algorithm is equivariant, or \cite{EV}, which is an exposition of these two papers).  But we shall use, to simplify, a different notation. 
Consider a permissible sequence of $A$-basic objects and transformations as in (1) of 
\ref{V:dt2.0}, which is  also $t$-permissible.

If $C$ is any irreducible permissible center for $B_j$, $j=0, \ldots r$, we define 
$${\o}_j (C):={{\nu} ({\bar I_j},C)}/b ~ , \qquad 
 {\sigma}_j (C):={{\nu} ({ I_j},C)}/b$$
where we have used the notation of  \ref{V:bn6.0} and \ref{V:bo2.0} . In \cite{EV} these numbers are denoted by
w-ord and ord respectively.

\end{voi}
\begin{pro} 
\label{P:nuom}

With the notation and assumptions introduced in \ref{V:dt3.0}, 
let $C$ be an irreducible permissible center for $B_r$ 
 $D'_i$ its image by the natural morphism $W_r \to W_i:=us(B_i)$, assume $D'_i \subseteq C_i$ and let $D_i$ be the irreducible component of the center $C_i$ containing $D'_i$, $i=0, \ldots , r-1$.  Then,         

$$ {\omega}_r (C) = {\sigma}_r (C) - {\sigma}_{r-1} (D_{r-1}) - \cdots - {\sigma}_0 (D_0) + r$$
\end{pro}
{\begin{proof} Note that if $B_j=(W_j \to S, I_j,b,E_j)$, with 
$E_j= (H_1, \ldots, H_{m+j})$, where $H_1, \ldots, H_m$ are the strict transforms of the hypersurfaces in $E_0$, 
then we have an expression 
$$(1) \quad I_j = I(H_{m+1})^{a_1} \ldots  I(H_{m+j})^{a_{j}} \overline{I_j}$$
where each exponent $a_i$ is constant on each irreducible component of $H_i$. From this (for $j=r)$, taking orders along $C$ (see \ref{V:bn6.0}) we get:
$$\nu (I_r,C) = \nu (\overline{I_j}, C) + \sum_{j=1}^{r} a_j(y)$$ 
where $y$ is the generic oint of $C$. Note that if $y_j$ is the image of $y$ in $W_{j}$ via the appropriate morphism,   
$$  a_j(y_j) = \nu  (I_{j-1},D_{j-1}) - b $$
From this , iterating, we obtain:
$$\nu  ( \overline{I_r},C) = \nu  ({I_r},C) - (\sum_{i=0}^{r-1} \nu (I_i, D_i)) + r b$$
Dividing by $b$ we obtain the desired formula.
\end{proof}

\medskip
We shall discuss the notion of {\it monomial objects} (\ref{V:ag6} (ii) or \cite{EV}, section 5) in the context of basic objects over a ring $A \in {\cA}$.
\begin{voi}
\label{V:det4.0}

{\it Monomial objects.} Using the notation of \ref{V:ag6} (ii), let $B=(W,I,b,E)$ (where $E=(H_1, \ldots, H_m)$) be a monomial object (with $W$ a 
$k$-variety, $k$ a characteristic zero field). We 
 define functions ${\Gamma _i}$, $i=1,2,3$, with domain
$S:= \sg \,(B)$, as follows.

If $w \in S$, $\Gamma _1 (w)$ is the smallest integer $p$ such that there are indices $i_1, \ldots, i_p$ such that
$$(1) \qquad \alpha_{i_1} (w) + \cdots + {\alpha_{i_p}(w)} \geq b $$
Consider, for $w \in S$,  the set $P'(w)$ of sequences  $i_1, \ldots, i_p$ satisfying $(1)$ above, and let  $\Gamma _2 (w)$ be the maximum of the rational numbers
$({\alpha _{i_1} (w)} + \cdots + {\alpha _{i_p}(w)})/b$, for
$(i_1, \ldots, i_p) \in P'(w)$.

If $w \in S$, let $P(w)$ be the set of all sequences 
$(i_1, \ldots, i_p,0,0, \ldots)$ such that 
$({\alpha _{i_1} (w)} + \cdots + {\alpha _{i_p}(w)})/b = \Gamma _2 (w)$, and define 
$\Gamma _3 (w)$ to be the maximum of the set $P(w)$, when we use the lexicographical order.

Finally, one defines a function $\Gamma$ (or $\Gamma _B $) from $S$ to
${\bf Z} \times {\bf Q} \times {\bf Z}^{\bf N}$
 by the formula 
$ \Gamma (w) = (- \Gamma _1 (w), \Gamma _2 (w), \Gamma _3 (w))$. When the target is lexicographically ordered the function $\Gamma$ is upper semicontinuous.

Let max $(\Gamma_3) = (i_1, \ldots, i_p,0,0, \ldots)$ and take
$C= H_{i_1} \cap \cdots \cap H_{i_p}$. Then, it turns out that $C$ is a permissible center for the pair $(W,E)$ and that the transform $B_1$ of $B$ is again monomial, satisfying ${\rm max}~ (\Gamma _{B_1}) < {\rm max} ~(\Gamma_B) $. Thus, iterating this process, after a finite number of steps we reach a situation where the singular locus is empty.
 (See \cite{EV}, section 5 or \cite{BEV}, section 20).

\smallskip

Now consider $A$-basic objects $B=(W \to S, I,b,E), \,E=(H_1, \dots, H_m)$ $ S=\Spec (A) \, A \in {\cA}$. We shall say that such an object is {\it premonomial} if its closed fiber $B^{(0)}$ is monomial. Let ${\Gamma} := {\Gamma}_{B^{(0)}} $. We say that our $A$-basic object $B$ is {\it monomial} if it is premonomial and, letting
$(i_1, \ldots, i_p,0,0, \ldots) = {\rm max}~ (\Gamma_3)$, then
$C:= H_{i_1} \cap \cdots \cap H_{i_p}$ is a $B$-permissible center. This is called the {\it canonical center } of the monomial $A$-basic object $B$.
\end{voi}

\smallskip

The following proposition is an easy consequence of the definitions.
\begin{pro}
\label{P:canctr}

Let $B$ be a monomial $A$-basic object $B$, 
$C:= H_{i_1} \cap \cdots \cap H_{i_p}$ its canonical center. Let $H_i$ induce $H^{(0)} _i$ on the special fiber $B^{(0)}$. Then, if  $B_1$ is the transform of $B$ with center $C$, the closed fiber $B^{(0)}$  is naturally isomorphic to the transform of $B^{(0)}$ with center
$H^{(0)}_{i_1} \cap \cdots \cap H^{(0)}_{i_p}$, and this is the canonical center of $B^{(0)}$.
\end{pro}

It immediately follows from the proposition that if $B$ is monomial, with canonical center $C$, then the transform $B_1$ of $B$ at $C$ is again premonomial.
\begin{Def}
\label{D:gamper} If $B$ is a basic object over $A$, we say that $B$ admits a {\it $\Gamma$-permissible center} if: (a) $B$ is monomial, (b) the transform $B_1$ of $B$ at the canonical center is again monomial.
\end{Def}
\begin{voi}
\label{V:det4.0'}
{\it ${\rho}$-permissible  sequences}. A sequence $B_0 \leftarrow \cdots \leftarrow B_r$ of $A$-basic objects and transformations is called $\rho$-permissible if there is an integer $s \ge 0$ such that: 
(a) $B_0 \leftarrow \cdots \leftarrow B_s$  is $t$-permissible, (b) $B_j$ is monomial if $s \le j$ and,  for all such $j$, $B_j \leftarrow B_{j+1}$ is the transformation with the canonical center of $B_j$. In particular, it could be $s >r$, in this case the sequence is $t$-permissible, or $s=0$, in this case all the objects are monomial and the transformations have canonical centers.
\end{voi}

\section{The inductive object}\label{S:IS}

\begin{voi}
\label{V:is1.0}
Throughout this section, $A$ denotes an artinian ring in $\mathcal A$ (\ref{V:bn1.0}) and  $S=\Spec (A)$. Given an $A$-basic object
$B=(W \to S, I,b,E)$, 
$E=(H_1, \ldots, H_m)$,  we say that  an $S$-hypersurface $Z \subset W$ is 
 transversal to $E$ if for each closed point $w$ in $Z$, there is an $A$-regular system of parameters 
$a_1, \ldots, a_n$ of ${\cO}_{W,w}$ 
 such that $I(Z)_w$ is defined by $a_1$ and for every $H$ in $E$ containing $w$, the ideal $I(H)_w$ is defined by some $a_i$, with $i>1$. In particular, $Z$ has normal crossings with $E$, in the sense of \ref{V:bn4.0}.
\end{voi}
\begin{voi} 
\label{V:is2.0}
{\it Adapted and inductive hypersurfaces, nice objects.}
Given an $A$-basic object 
$B=
{(W \to S, I,b,E)}$ as above, we say that an $S$-hypersurface $Z \subset W$ 
is {\it adapted} to $B$ (or that $Z$ is {\it $B$-adapted}) if
 the following conditions hold:\\
\smallskip
(A1) $I(Z) \subseteq \Delta ^{b-1} (I/S)$ (an inclusion of sheaves of
${\cO}_W$-ideals),\\
\smallskip
(A2) $Z$ is transversal to $E$. \\
\smallskip
If $Z$ is adapted and, moreover, it satisfies\\ 
(A3) Whenever $D$ (resp. $D'$) is an irreducible component of $Z$ (resp. of
$V({\Delta}^{b-1}(I/S))$) we have $D \not= D'$, \\
then we say that $Z$ is $B$-inductive (or just inductive, if $B$ is clear). 

In  \cite{EV} the analogues of (A1) and (A2) are called (LC) and (IA),
respectively. Condition (A3) does not follow from the others. For instance, if $B$ is the basic object $(\Spec ({\bf C}[x,y],(y),1, \emptyset)$
and $Z$ is the line $V(y)$, then $Z$ satisfies (A1) and (A2)
but $D=Z$ is a component of both $Z$ and $V({\Delta}^{b-1}(I/S))$. But if $\sg(B)$ has codimension at least two, then any $B$-adapted hypersurface is automatically inductive.

\smallskip
We shall say that $B$ (a basic object over $A$) is {\it nice} if either
Sing($B$) is empty or $B$ admits an adapted hypersurface. We say that $B$ {\it
is nice at} $w \in W$ if there is an open neighborhood $U$ of $w$ such that
the restriction of $B$ to $U$ is nice; $B$ is {\it locally nice} if $B$ is nice at $w$, for any $w \in W$.
\end{voi}
\begin{voi}
\label{V:is3.0}
{\it Remarks on hypersurfaces}. We keep the previous assumptions and notation.
Recall that, according to  \ref{V:bn2.0} and \ref{V:bn9.0}, if $Z \subset W$ is a regular
$A$-hypersurface and $w \in Z$ is a closed point, then the stalk $I(Z)_w$ is generated by an
element $a \in {\cO}_{W,w}$ which is part of an $A$-regular system of
parameters 
\ref{V:bn3.0}. In  particular, if
$a_1, \ldots, a_d$ form an $A$-regular system of parameters of
${\cO}_{W,w}$, $J:= (a_1, \ldots, a_d)$ and $a \in J$, then the order of the generator $a$
with respect to $J$ is =1.

Now assume that the regular hypersurface $Z$ also satisfies (A1), i.e.,
$I(Z) \subseteq {\Delta}^{(b-1)}:={\Delta}^{b-1}(I/S)$. Then, if $w \in Z$, the generator $a \in I(Z)_w$ defining $Z$ near $w$ belongs to the stalk 
$[{\Delta}^{(b-1)}]_w$. If $C$ is a $B$-permissible center, then by (A1) and
\ref{P:ekiv} 
necessarily $C \subseteq Z$. If
$w \in C$,  then
$J:=I(C)_w$ is generated by part of an $A$-regular system of parameters. Since $C \subseteq Z$, $a \in J$ and hence the order of $a$ with respect to $J$ is 1. Thus 
we see that $[{\Delta}^{(b-1)}]_w$ contains an element which is part of an
$A$-regular system of parameters.

Conversely, if $w \in W$ and
$[{\Delta}^{(b-1)}]_w$
contains an element $a$ which is part of an $A$-regular system of parameters,
then $a$ defines, near $w$, an $A$-regular hypersurface
satisfying condition (A1).

\end{voi}
\begin{lem} 
\label{L:fibZ}
If $B$ is nice basic object over $A$, then its fiber $B^{(0)}$ is
also a nice basic object (over the residue field of $A$).
\end{lem}
The proof is simple and will be omitted.

\begin{voi}
\label{V:is4.0}If $B$  is a nice object and $C \subset W$ is an irreducible $B$-permissible
center, with generic point $y$, then $\nu (C,I)=b$. To see this, it suffices to show that if $w$ is a closed point of $\sg B$, then the stalk 
$[{\Delta}^{(b-1)}]_{w}$ contains contains an element which is part of an
$A$-regular system of parameters. But
 if $Z$ is any
adapted hypersurface, from condition (A1) and the inclusion
$\Delta ^{b-1} (I/S) \subseteq I(C)$ we obtain $I(Z) \subseteq I(C)$. From
this, our assertion follows from \ref{V:is3.0}, taking $J=I(C)_y$.
\end{voi}
\begin{voi}
\label{V:is4.1}
Recall that if $B=(W,I,b,E)$ is a basic object over a field, then it is called {\it
good} if $\nu _w(I)=b$ for any $w \in \sg(B)$ (see  \cite{EV} ). Working over $A$ artinian, we shall say that an $A$-basic object $B$ is {\it good} if its special fiber $B^{(0)}$ is good, in the sense just mentioned. The result of \ref{V:is4.0} says that, working over a field, a nice object is good. Since the fiber of a nice $A$-object is a nice $k$-object ($k$ being the residue field of $A$), the same is true working with $A$-basic objects.
\end{voi}
\begin{rem}
\label{R:cegood}
Given a good $A$-basic object $B=(W \to S ,I,b, E)$, a closed subscheme $C$ of $W$ which is a center for the pair $(W \to S, E)$ (see \ref{V:bn5.0}) is also a $B$-permissible center if and only if 
${\Delta}^{b-1}(I/S) \subseteq I(C)$. In other words, the extra condition 
$\nu(I,C')=\nu(I^{(0)},C'^{(0)})$, for any component $C'$ of $C$, automatically follows. Indeed, one implication is clear, for the other note that from the inclusion and the assumption that $B$ is good we get 
$b \leq \nu(I,C') \leq \nu(I^{(0)},C'^{(0)}) = b$, hence all of these are equalities.
\end{rem}
\begin{voi}
\label{V:is5.0}
The notion of adapted hypersurface is stable under permissible
transformation, in the sense that if
$B \leftarrow B_1 = (W_1 \to S, I_1,b,E_1)$ is a transformation along a
$B$-permissible center $C$ and $Z_1$ is the strict transform of $Z$ to $W_1$,
then $Z_1$ is an adapted hypersurface for $B_1$. The stability of condition
(A1) follows from
Lemma \ref{L:gir} below (Giraud's Lemma), part (b), and that of (A2) is well known.   
 Lemma \ref{L:gir} follows from Lemma \ref{L:preg}, to be given next. The proofs of both
lemmas are entirely analogous to those given in 9.1 and 9.2 of \cite{EV}, and
will be omitted. Indeed, thanks to the remarks in \ref{V:bn7.0}, the necessary
calculations in completions of rings (discussed in \cite{EV}, section 9) can
be carried out in our case.
\end{voi}
\begin{lem}
\label{L:preg}
Using the notation of \ref{V:is5.0}, let $B \leftarrow B_1$ be a permissible
transformation of $A$-basic objects, $H \subset W_1$ the exceptional divisor, with defining ideal $I(H)$.  Then,

\smallskip
(i) $\Delta ^{b-i} (I/S) {\cO}_{W_1} \subseteq I(H)^i$,

\smallskip
(ii) $ I(H)^{-i}\Delta ^{b-i} (I/S) {\cO}_{W_1} \, \subseteq  \Delta ^{b-i}
(I_1/S)$.
\end{lem}
\begin{lem}
\label{L:gir}
Consider an $A$-basic object
$B=(W \to S,J,b,E)$, an  $A$-adapted hypersurface $Z \subset W$, an irreducible
$B$-permissible center $C \subset Z$ and $B_1=(W_1 \to S, J_1,b,E_1)$, the
transform of $B$ with center $C$. Let
$B'=(W \to S,I(Z),1,E)$. Then :

\smallskip
(a) $C$ is  a $B'$-permissible center, and if
$B'_1=(W_1 \to S, I(Z)_1,1,E_1)$ is the the transform of $B'$ with center
$C$, then
$I(Z)_1 = I(Z_1)$, where $Z_1$ is the strict transform of $Z$ to $W_1$,

\smallskip
(b) $I(Z_1) \subset \Delta ^{b-1}(J_1 /S)$.
\end{lem}
\begin{voi}
\label{V:is6.0}
{\it The coefficient ideal and the object $B_Z$.}  We use the notation of \ref{V:is1.0}. Let  $B$ be a nice $A$-basic object and $Z$ a $B$-inductive hypersurface.
 
 \smallskip 
 
 (a) We define a
sheaf of ${\cO}_W$-ideals, called the {\it coefficient ideal} and denoted by
${\cC}(I/S)$, as follows:
 $${\cC}(I/S) : = \sum_{i=0}^{b-1} \, [{\Delta}^i(I/S)]^{b!/b-i}  $$
 %For $i=0,\ldots, b-1$ let $I_Z^{(i)}:= \Sum[{\Delta}^i(I/S)]^{b!/b-i}$.
 %For $i=0,\ldots, b-1$ let $I_Z^{(i)}:= [{\Delta}^i(I/S)]^{b!/b-i}{\cO}_Z$.
 Often, when $A$ is a field we'll simply write 
${\cC}(I)$.
 If $B$ is nice and $Z$ a $B$-inductive hypersurface, we define a
sheaf of ideals on $Z$, called the {\it coefficient ideal relative to $Z$}, or the 
$Z$-{\it coefficient ideal}, denoted by
${\cC}(I/S,Z)$, as the restriction of  
${\cC}(I/S)$ to $Z$. 

\smallskip

(b) The $A$-basic object 
$B_Z := (Z \to S, {\cC}(I/S,Z), b!, E_Z)$, where we write $E_Z := ( {H_1 \cap Z},
\ldots, {H_m \cap Z})$ is called the {\it inductive object of $B$, relative to the inductive hypersurface $Z$}. Indeed, this is an $A$-basic object. The fact that $E_Z$ is a
system of $A$-hypersurfaces with normal crossings follows from (A2); what is
left is to check that ${\cC}(I/S,Z)$ is a never zero sheaf of ideals, and this follows
from (A3).
\end{voi}
\begin{voi}
\label{V:is7.0}
Condition (A3) is important. For instance, if the basic object is $B=(\Spec ({\bf C}[x,y],(y^2),
2, \emptyset)$ and $Z$ is the line $V(y)$, then ${\cC}((y^2))$ restricts the zero sheaf
of ideals on $Z$. So, $Z$ is not a $B$-inductive hypersurface.

But, as remarked in \ref{V:is2.0}, if $\sg(B)$ has codimension at least two, any $B$-adapted hypersurface is automatically inductive, so for any  $B$-adapted hypersurface  the inductive object is defined. We shall use this observation several times, for instance in sections \ref{V:ae3.0} and \ref{V:ae4.0}.

It is well known that the object $B_Z$ is not necessarily nice, even if the base ring $A$ is a
field $k$. 
\end{voi}

\begin{voi}
\label{V:is8.0}
In case we are working over a field $k$ (characteristic zero), if
$B=(W,J,b,E)$ is a nice basic object and $Z$ an inductive hypersurface, it is proved in \cite{EV} that the 
object $B_Z$ has the following properties:

\smallskip
(i)
$ \sg (B) = \sg (B_Z)$. From this it follows that a subscheme  $C$ of $W$ is a
$B$-permissible center if and only if $C$ is a $B_Z$-permissible center.

\smallskip
(ii) If $B \leftarrow B_1$ is a permissible transformation, by (i) using the
same center we obtain a permissible transformation
$B_Z \leftarrow (B_Z)_1$. By \ref{V:is5.0}  the strict transform $Z_1$ of $Z$ to
$B_1$ is an inductive hypersurface of $B_1$. Although we cannot prove that
$(B_1)_{Z_1}$ is isomorphic to
$(B_Z)_1$, it is true that
$\sg ((B_1)_{Z_1}) = \sg ({(B_Z)_1})$. We may repeat the process, taking a
transformation of $B_1$ with a permissible center $C_1$ (which will also be
$(B_Z)_1$-permissible.) After iterating $k$ times we get (in obvious notation):
$$(1) \qquad \sg ((B_k)_{Z_k}) = \sg ({(B_Z)_k})$$
The  analog of (i) in the context of basic objects over an artinian ring
$A$ would be the assertion: if $B$ is a nice $A$-basic object and $Z$ is a
$B$-inductive hypersurface, then a center $C$ is $B$-permissible if and only if
it is $B_Z$-permissible. More precisely, if $C$ is a permissible center for $B_Z$, then by (A2) $C$ will be a permissible center for the pair $(W \to S, E)$ (see \ref{V:bn5.0}). But, will it  be $B$-permissible (i.e., will 
$\nu(I,C)=\nu(I^{(0)}, C^{(0)})\ge b$ hold?) Conversely, if $C$ is a $B$-permissible center, then by (A1) it will be permissible for the pair $(Z \to S, E_Z)$. But, will it be $B_Z$-permissible? 

The answer to both questions is negative. Next we present  examples.

\end{voi}
\begin{exa}
\label{E:accidenti}
(Showing that $C$ is $B$-permissible does not imply $C$ is $B_Z$-permissible). 
Let $B=(W \to S,I,2, \emptyset)$ where $S=\Spec (A)$, $A=k[\epsilon]$, $k$ a field,
$\epsilon ^{2}=0$, $W=\sp \, (A[x,z])$, $I=(z^2 + \epsilon x^2, z^3 + x^3) $.
 Then, $\Delta ^{1}(I/S) = (z, \epsilon x, x^2)$ and one easily checks that $Z$, defined by the ideal $(z) \, A[x,z]$ is a $B$-inductive hypersurface, and
 $B_Z = (\sp \, (A[x]) \to \sp \,(A), (\epsilon x^2, x^3), 2, \emptyset)$. Then we see that $C \subset W$, defined by the ideal $(x,z)$ is a $B$-permissible center but not a
 $B_Z$-permissible center.
\end{exa}
\begin{exa}
\label{E:accidento} 
(Showing that $C$ is $B_Z$-permissible does not imply $C$ is $B$-permissible). Let $A$ and  $W \to S$ be as in Example \ref{E:accidenti},  
 $B=(W \to S,I,4, \emptyset)$, where $I=(x^5+\epsilon x^2 \, z + z^4)$. Then one verifies that $Z$, the subscheme of $W$ defined by $(z) \, A[x,z]$ is a $B$-inductive hypersurface, and $B_Z = (\sp \, (A[x]) \to \sp \,(A), ( x^{30}), 24, \emptyset)$. The subscheme $C$ of $Z$, defined by 
 $(x) A[x]$ on $Z$ and by $(x,z) $ on $W$, is a $B_Z$-permissible center but not a $B$-permissible one.

\end{exa}
\begin{Def}
\label{D:strong}
 We use the notation of \ref{V:is6.0}. A closed subscheme $C$ of $Z$ is a {\it strongly permissible} center for $B_Z$ (or a {\it strong $B_Z$-center}) if  it is a permissible center for both $B_Z$ and $B$. 
\end{Def}
\begin{voi} 
\label{V:is10.0}
We extend this notion to sequences as follows. Let $B$ be a nice $S$-basic object, $Z$ a $B$-inductive hypersurface. We use the notation of \ref{V:is6.0}. A permissible sequence of $S$-basic objects: 
$$(1) \qquad B_Z:= (B_Z)_{0} \leftarrow  (B_Z)_{1} \leftarrow \cdots \leftarrow \ (B_Z)_{m}$$ 
(with centers $C_i \subset us((B_Z)_i)$, $i=0, \ldots m-1$) is said to be {\it strongly permissible} if: $C_0$ is a strongly permissible
 $(B_Z)_0$-center (hence, $C_0$ is a $B$-center and if $B_1:={\cT}(B,C_0)$,  $us((B_Z)_1)$
 may be identified, to the strict transform $Z_1$ of $Z$ to $us(B_1)$, and $C_1$ to a closed subscheme of $us(B_1)$),  $C_1$ is a 
 $B_1$-permissible center; if $B_2:={\cT}(B_1,C_1)$, with similar identifications $C_2$ is a permissible $B_2$-center, and so on. So, eventually, with these identifications, $C_i$ must be a $B_i$ center, $i=0, \ldots, m-1$. 
 
 Thus, the strongly permissible sequence (1) induces a permissible sequence of $S$-basic objects 
 $$(2) \quad B \leftarrow  B_1 \leftarrow \cdots \leftarrow \ B_m $$
 using the same centers $C_i$ as in (1).
 
 When $A=k=$ a field, according to \ref{V:is8.0} a center is permissible if and only if it is strongly permissible.
 \end{voi}

\section{The homogenized ideal and applications}
\label{S:HI}

In this section we present, adapted to  our needs, some notions and results due to Wlodarzcyck (cf. \cite{W}, where one works over a field). Throughout, $S= {\rm Spec} \, (A)$, $A \in \cA$, as in \ref{V:bn1.0}. 

\begin{Def} 
\label{D:wid} Given a scheme $W$, a {\it weighted $W$-ideal} is a pair $(I,b)$, where $I \subset {\cO}_W$ is $W$-ideal and $b$ is a non-negative integer. Often, if $W$ is clear, we'll talk about a weighted ideal.

\end{Def}
\begin{Def} 
\label{D:ho} 
Let $(I,b)$ be a weighted $W$-ideal, where $W$ is a scheme, smooth over $S$. The {\it homogenized ideal, relative to $S$,  associated to} $(I,b)$ is the the $W$-ideal $\cH (I/S,b)= I + \Delta (I/S) T(I/S) + \cdots +  \Delta ^i(I/S) T(I/S) ^i + \cdots + \Delta ^{b-1}(I/S) T(I/S)^{b-1}$,  
where we have written $T(I/S):= {\Delta}^{b-1}(I/S)$. If $B=(W \to S, I,b, E)$ is a  $A$-basic object, we use the notation $\cH (B):=(W \to S, \cH (I/S,b), b,E)$ (a new $A$-basic object).
\end{Def}
\noindent
{Note that passing to the  fiber we obtain, in the usual notation,  
$\cH(I^{(0)},b)=\cH(I/S,b)^{(0)}$.}
\begin{lem}
\label{L:Tdelta}
If $(I,b)$ is a weighted ideal on an $S$-smooth scheme $W$, then 
${\Delta}^{b-1}(I/S) = {\Delta}^{b-1}({\cH}(I/S,b))$
\end{lem}
The proof is elementary and we leave it to the reader.
\begin{pro} 
\label{P:wide} If $B=(W \to S, I,b, E)$ is a good $A$-basic object (see \ref{V:det4.0}), then $\cH (B)$ is W-equivalent to $B$.
\end{pro}
\begin{proof} 
By Remark \ref{R:cegood} and  Lemma \ref{L:Tdelta} a center $C$ is $B$-permissible if and only if it is $\cH (I)$-permissible.  By induction, by the same reason it suffices to show that, for all $ r > 0$, if 
$B \leftarrow B_1 \leftarrow  \cdots \leftarrow B_r$ and 
$\cH (B) \leftarrow [\cH (B)]_1 \leftarrow  \cdots \leftarrow [\cH (B)]_r$ 
are permissible sequences of basic objects and transformations, where in both cases the same permissible centers have been used, then 
$$(1) \qquad {\Delta}^{b-1}(I_r/S) = {\Delta}^{b-1}([{{\cH}(I/S,b)}]_r)$$
(where 
${{\cH}(I/S,b)}_r$ is the ideal in the basic object 
$\cH (B)_r$). To simplify, sometimes we'll drop ``$b$'', writing, e.g., $\cH (I/S,b)=\cH (I/S)$. 
 
First we shall check that 
$$(2) \qquad [{\cH}(I/S)]_r \subseteq {\cH}(I_r/S)$$ 
To see this,  consider 
 $ [{\cH}(I/S]_{r-1} = I_{r-1} + \sum_{i=1}^{b-1} {\Delta}^i (I_{r-1}/S) T(I_{r-1}/S)^i  $. The controlled transform of this ideal to $W_r$ is 
$$[{\cH}(I/S)]_r = I_{r} + \sum_{i=1}^{b-1} [{\Delta}^i (I_{r-1}/S) T(I_{r-1}/S)^i]_1.$$ 
But, letting $H$ be the exceptional divisor of the blowing up $W_{r-1} \leftarrow W_r$ and $\cE :=I(H)$, 
 $[{\Delta}^i (I_{r-1}/S) T(I_{r-1}/S^i]_1= {\cE}^{-(b-i)}{\Delta}^i(I/S)[{\cE}^{-1}T(I/S)]^i \subseteq 
 {\Delta}^i(I_r/S) [T(I_r/S)]^i$, the last inclusion by \ref{L:preg}. From this (2) clearly follows. So, we have inclusions 
 $I_r \subseteq {\cH}(I/S,b)_r   \subseteq \cH (I_r/S,b)$. Applying the operator 
 ${\Delta}^{b-1}$ to each of these and using \ref{L:Tdelta} we see that the resulting first and third terms are equal, hence we obtain (1).
 \end{proof}
 
  We shall need the following basic result.
 \begin{pro}
 \label{P:twoids}
 Let  $B=(W \to S, I, b, E)$ be an $A$-basic object, $y$ a point of $W$, 
 $R={\cO}_{W,y}$, $a_1, \ldots, a_r$ an $A$-regular system of parameters of $R$, $M$ = r($A$) and  
 $J = \\
 MR + (a_1, \ldots, a_r) R$. Then, $J$ is the maximal ideal of $R$.
 \end{pro}
 \begin{proof}
 Let $J':={\mathrm{max}}(R)$. We have $J \subseteq J'$. To see that equality holds, letting $R'$ denote the completion of $R$ with respect to $J$, and using the fact that $R'$ is a flat $R$-algebra, it suffices to show that $JR=J' R'$. Now, we have the following identifications: 
 $R'/JR' = (R'/MR')/(a_1, \ldots,a_r)(R'/MR') = {\widehat{(R/MR)}}/(a_1, \ldots, a_r){\widehat{(R/MR)}}=
 {\widehat{(R^{(0)})}}/{\mathrm {\mathrm max}}{\widehat{(R^{(0)})}}$, where the hat indicates completion with respect to the maximal ideal. The last equality holds because  
 $a_1, \ldots, a_r$ induces a regular system of parameters in the regular local ring 
 $R^{(0)}={\cO}_{W^{(0)},y}$, in particular it generates the maximal ideal of $R^{(0)}$. So, 
 $R'/JR' $ is a field, hence $JR'$ must be equal to the maximal ideal $J'R'$, as claimed.
 \end{proof}
  \begin{pro}
 \label{P:twocomp}
 With the notation of the previous proposition, letting $\widehat{R}$ denote the completion of $R$ with respect to its maximal ideal and $R'$ the completion with respect to 
 $(a_1, \ldots, a_r)R$, we have 
 $R'=  {
 \widehat{R}}$.
 \end{pro}
 \begin{proof}
 Using the fact that $A$ is artinian and hence a power of its maximal ideal $M$ is equal to zero, it is easy to verify that there is a positive integer $s$ such that for all sufficiently large integers $t$ we have  
 $Q^t \subseteq J^t \subseteq Q^{t-s}$, 
 where $Q=(a_1, \ldots, a_r)R$ and $J = MR + (a_1, \ldots, a_r) R=r(R)$ (use \ref{P:twoids}). It is well known that these inclusions imply that the topologies defined by $Q$ and $J$ are the same, whence the claimed equality of completions. 
 \end{proof}
 \begin{pro}
 \label{P:elauto}
 Let $B$, the point $y \in W$ and $R={\cO}_{W,y}$ be as in \ref{P:twoids}, let both $u,u_2, \ldots, u_n$ and $v, u_2, \ldots, u_n$ be $A$-regular systems of parameters of $R$, $\widehat R$ the completion of $R$ with respect to its maximal ideal ${\mathcal M}$. Then, there is an $A$-linear automorphism $\phi$ of $\widehat R$ such that 
 $\phi (u) = v$ and $\phi (u_i)=u_i$ for $i=2, \ldots, n$ and, moreover, such that if $h=u-v$ and 
 $p:{\widehat R} \to     {\widehat R}/(h)$ is the canonical quotient homomorphism, then $
 p ~ \phi = p$.   
  \end{pro}
 \begin{proof}
Consider the completion $R^*$ of $R$ with respect to the ideal generated by $(u,u_2, \ldots, u_n)$. We know: 
(i) $R^* = A'[[u,u_2,\ldots, u_n]]$, for a suitable over-ring $A'$ of $A$ (see 
\ref{V:bn7.0} or \ref{P:preries}), 
(ii) $R^* = \widehat R$ (see \ref{P:twocomp}.) Using the identification (ii), to prove the proposition it suffices to define an automorphism $\phi$ of $R^*$ satisfying the required properties. Define $\phi$ by the conditions: $\phi (a) = a$ for $a \in A'$, 
$\phi (u) =v$ and  $\phi(u_j)=u_j$, $j = 2, \ldots, n$. This is correct, since $v$ , regarded as an element of $R^*$, must be of the form 
$v=m + \alpha $, with $m \in MR^*$ ($M=r(A)$), $\alpha \in (u,u_2, \ldots,u_n)R^*$, and $M^s=0$ for $s$ large enough. Modulo $MR^*$, this induces a homomorphism of rings of formal power series 
$K[[u,u_2, \ldots, u_n]] \to K[[v,u_2, \ldots, u_n]]$ (with $K$ a field), which clearly is an isomorphism. By the ``nilpotent'' Nakayama's Lemma, $\phi$ must be an isomorphism. The statement about the quotient map $p$ follows from the fact that 
$ u - \phi (u) = h$.
\end{proof}
\begin{pro}
\label{P:hompl}
Keep the assumptions and notation of \ref{P:elauto}, but also assume that  $u \in T(I/S)_y$ and $v \in T(I/S)_y$. Then, the automorphism $\phi$ satisfies: 
$\phi (\cH(I/S) {\widehat {R}})=\cH(I/S) {\widehat {R}} $. 
\end{pro}
\begin{proof}
Above, we wrote $\cH(I/S) {\widehat {R}}:=\cH(I/S)_y {\widehat {R}} $.  Recall the definitions in \ref{P:elauto}: we identify ${\widehat {R}}$ with 
$R^*:=A'[[u,u_2, \ldots, u_n]]$, then 
$\phi:R^* \to R^*$ is defined by: $\phi(u)=v, \, \phi(u_i)=u_i$, for $i=2, \ldots, n$, $\phi(a) =a$ for $a \in A'$. Thus, for $f(u,u_2,\ldots, u_n) \in R^*$, 
$$(1) \qquad \phi \,  f=
f(u)+ \dfrac{\partial f}{\partial u} h + \frac{1}{2!}\dfrac{\partial ^{2}f}{\partial u^2} \,   u^2 h^2 +
\frac{1}{3!}\dfrac{\partial ^{3}f}{\partial u^3}  u^3 \, h^3+ \cdots$$
(With this set-up, what follows imitates the proof of \cite{W}, Lemma 2.9.4). To check the desired equality, since $\phi$ is an automorphism of a noetherian ring, it suffices to show: 
$\phi ({\cH}(I/S) R^*) \subseteq {\cH}(I/S) R^*$. \\
To check the latter inclusion, in view of the definition of 
$\cH (I/S)$ in \ref{D:ho}, it suffices to show: 
$\phi (\Delta ^j (I/S)[T(I/S)]^j R^* \subseteq \cH (I/S) R^*$, $j=0, \ldots, b-1$, where as in \ref{D:ho} we have written $T(I/S) = \Delta ^{b-1} (I/S)$. It is easy to verify that this inclusion follows if we can prove the following assertions: 
$(a_j): \phi({\Delta}^j(I/S) R^*) \subseteq {\cH}({\Delta}^j(I/S) R^*)$, $j=0, \ldots, b-1$, $(b): \phi(T(I/S)R^*) \subseteq {\cH}(I/S) R^*$.
Now, $(a_0)$ is a consequence of formula (1) above. Since this is valid for any ideal $I$, we may substitute (in $(a_0)$) $I$ by $\Delta ^j (I/S)$, and we obtain $(a_j)$, for all $j$. Concerning $(b)$, if $g \in T(I/S)_y$, by (1) (with $f=g$) and the fact that $h = u-v\in 
T(I/S)$ (hence $h^i \in [T(I/S)_y]^i$ for all $i$), it follows that $\phi(g) \in T(I/S)$, as desired.
\end{proof}
 
Let us recall some terminology and notation to be used in the next theorem. Given a scheme 
$W$ and a point $w \in W$, an etale neighborhod of $y$ in $W$ consists of a scheme $V$, a point $v \in V$ and an etale morphism 
$\pi:V \to W$ such that $\pi(v)=w$. We use the notation $\pi:(V,v) \to (W,w)$ to indicate this etale neighborhood. To shrink the etale neighborhood $\pi$ means to take the composition of $\pi \, \pi '$ of $\pi$ with another etale neighborhood  $\pi':(V',v') \to (V,v)$. 
 If $f$ and $g$ are morphisms from a scheme $V$ to a scheme $W$ and  $Y$ is a closed subscheme of $W$, we say that $f$ and $g$ {\it agree over $Y$} if 
$f^{-1}(Y) = g^{-1}(Y):=Y'$ and $fj = gj$, where $j:Y' \to V$ is the inclusion.

If $B=(p:W \to S,I,b,E)$ is an $S$-basic object, $f:V \to W$ an etale morphism, $f^*(B)$ denotes the $S$-basic object $(V,f^*(I),b,f^*(E))$, where $f^*(I):=I{\cO}_V$ and $f^*(E)$ is the sequence of inverse images by $f$ of the hypersurfaces that appear in $E$.

\begin{thm}
\label{T:etal} 
Let $B=(W \to S, I, b, E)$ be a nice basic object, $Z$ and $Z'$ adapted hypersurfaces, 
$y \in Z \cap Z'$. Let 
 $Z$ (resp. $Z'$) be defined by 
$u \in \Gamma(W, \cO _W)$ (resp. $v \in \Gamma(W, \cO _W)$) and $J=\cH(I/S)$. Then, there are etale neighborhoods  
$\pi _u:(V,z) \to (W,y)$ and $\pi _v:(V,z) \to (W,y)$ such that: 
 (a) $ \pi_u^*(J)=\pi_v^*(J)$ ,  
 (b) $\pi_u^*(u)=\pi_v^*(v)$,   (c) if $H$ is any hypersurface in $E$, then 
 $\pi _u ^{-1}(H)=\pi _v ^{-1}(H)$ (hence, $\pi_u^*(E) = \pi_v^*(E)$).  
 (d) $\pi_u=\pi_v$ over ${\mathcal V}(u-v)$
\end{thm}
\begin{proof}
Since $W$ is of finite type over $A$, we may find an affine open neighborhood of $y$, still denoted by $W$, of the form $\Spec (D)$, where $D = A[{\bf x}]/K$, with 
${\bf x} = (x_1, \ldots, x_m)$ ($x_1, \ldots, x_m$ algebraically independent over $A$) and  $K$ and ideal of $A[{\bf x}]$, say generated by 
 $f_i \in A[{\bf x}]$, $i=1, \ldots, r$. The point $y$ corresponds to a prime ideal $P$ of $D$, so we may identify $R=\cO _{W,y}=D_P$. Moreover, shrinking $W$, if necessary, we may assume that: 
 (i) $J \, R = (g_1, \ldots, g_t) R$, where $g_j \in A[{\bf x}]$, for all $j$, 
 (ii) there are elements $u_2, \ldots, u_n$ in $R$ such that both $u, u_2, \ldots, u_n$ and $v, u_2, \ldots, u_n $ are $A$-regular systems of parameters of $R$, where each hypersurface $H$ in $E$ containing $y$ is defined at $y$ by some $u_i$, $i > 1$ and, moreover, $u$ and  $u_j,\, j = 2, \ldots, n$ are images  in $R=D_P$ of  polynomials (in $A[{\bf x}]$)
 $q({\bf x})$ and 
 $q_j({\bf x}), \, j = 2, \ldots, n$ respectively.
 
 Now, we have natural homomorphisms 
 $$A[{\bf x}] \to A[{\bf x}]/K \to (A[{\bf x}]/K)_P = R \to {\widehat R}$$
 where ${\widehat R}$ denotes the completion of $R$ with respect to its maximal ideal.
  Let $\psi:A[{\bf x}] \to {\widehat R}$ 
 be the composition of these homomorphisms, and  $a_i$ the image of $x_i \in A[{\bf x}]$ in 
 ${\widehat R}$,
  $i=1, \ldots, m$. The fact that $\psi$ factors through $K$ means:
 $$(1) \qquad f_i(a_1, \ldots, a_m ) = 0, ~ i=1, \dots, r$$
Now consider the automorphism $\phi$ (of ${\widehat R}$) associated to 
 $u,v,u_2, \ldots, u_n$, as in \ref{P:elauto}, and the composition 
 $\phi \, \psi : A[{\bf x}]\to  {\widehat R}$. Since $\phi$ is $A$-linear and the coefficients of $q({\bf x})$ are in $A$, the fact that $\phi(u)=v$ means:
$$(2) \qquad q(a_1, \ldots, a_m) = v \in R$$
For the same reasons, the fact that $\phi(u_i)=u_i,~ i = 2, \ldots, n$, means
$$(3) \qquad q_i(a_1, \ldots, a_m) = u_i \in R, ~ i = 2, \ldots, n$$
Let  $J \, R$ be generated by elements 
 $h_1, \ldots, h_p$ in $R$. Then, as above, from $\phi(JR)=JR$ we obtain:
$$(4) \qquad g_j = c_{j1}\, h_1 + \cdots, c_{jp}\, h_p, ~ j=1, \ldots, t $$
for suitable elements $c_{ji} \in {\widehat R}$. 
 
 We may view (1) through (4) as polynomial equations with coefficients in $R$ in unknowns 
 $\{A_k\}$ and $\{C_{ji}\}$, which have a solution  $\{a_k\}$ and $\{c_{ji}\}$ in  
${\widehat R}$. By Artin's approximation Lemma (see \cite{A} 1.10, note that since $A$ is a finite dimensional vector space over $k=A/r(A)$, the hypothesis are valid) for any positive integer 
$s$ we find solutions $\{b_k\}, ~ \{d_{ji}\}$ in $\widetilde R$, the henselization of $R$ with respect to its maximal ideal. Choose $s \geq 2$. By sending 
$X_i$ to $b_i \in \widetilde R  \subset \widehat R$, 
$i=1, \ldots, m$, we  obtain an $A$-linear homomorphism 
$\alpha ':A[{\bf x}] \to \widetilde R  \subset \widehat R$, which induces a homomorphism 
$\alpha : A[{\bf x}]/K \to \widetilde R $ and, localizing at $P$ and completing, another 
$\phi ': \widehat R \to \widehat R$. By our choice $s \geq 2$, $\phi '$ agrees with 
the isomorphism $\phi$ modulo 
${\mathcal N}^{2}$ (${\mathcal N}= r({\widehat R}$). Hence, $\phi '$ is also an automorphism of ${\widehat R}$, satisfying $\phi ' (u) = v$ and $\phi '(u_i) = u_i$, $i=1, \ldots, n$.

Returning to $\alpha: D \to {\widetilde R}$, from the definition of henselization (involving a limit of etale neighborhoods) for some etale neighborhood 
$\pi_v:(V,z) \to (W,y)$ there is an induced morphism 
$\bar{\alpha}: V \to W $ such that $\bar{\alpha}(z)=y$. If we consider the induced homomorphism 
$\widehat{\cO _{W,y}} \to \widehat{\cO _{V,z}}$, by our construction this may be identified to the isomorphism 
$\phi ': \widehat{R} \to \widehat{R}$. Thus, $\bar{\alpha}$ is etale at $z$. Hence, shrinking $V$ if necessary, $\bar{\alpha}$ defines an etale neighborhood 
$\pi_u:(V,z) \to (W,y)$. Shrinking again this neighborhood if necessary, by using \ref{P:elauto} we see $\pi_u$ and $\pi_v$ satisfy properties (a) through (d). 
\end{proof}
 \begin{voi}
\label{V:hi1.0} 
Theorem \ref{T:etal} can be extended to sequences of permissible transformations of $A$-basic objects as follows. Suppose we are in the conditions of \ref{T:etal}. We retain the notation of that theorem, but we also write  
$B_0:= (W \to S, J,b,E)$ (recall $J= \cH (I/S)$). Take a $B_0$-permissible center $C$. Since  $C \subseteq {\mathcal V}(u-v)$, by \ref{T:etal} (d) 
$\pi _u ^{-1}(C)=\pi _u ^{-1}(C) := {\overline C} \subseteq V$. Let us write 
${\overline J}:=\pi _u ^* (J) = \pi _v ^* (J)$ (the equality by \ref{T:etal}, (a)), and 
${\overline B}_0= (V, {\overline J}, b, {\overline E})$, where ${\overline E}:=\pi_u^*(E)=\pi_v^*(E)$. 
 %Since ${\mathcal V}  (\Delta ^{b-1}(J/S)) \subseteq {\mathcal V}(u-v)$ and both 
 Since $\pi _u$ and  $\pi _v$ are etale, we have that both 
 $\pi  _u ^* ( \Delta ^{b-1}(J/S))$ and $  \pi  _u^* ( \Delta ^{b-1}(J/S)) $ agree with 
 $\Delta ^{b-1}({\overline J}/S)$, i.e, 
  $\pi  _u^* ( \Delta ^{b-1}(J/S))
=  \pi  _u^* ( \Delta ^{b-1}(J/S))$
 $=\Delta ^{b-1}({\overline J}/S)$.
 Now  we may easily check that ${\overline C}$ is a 
$\overline B$-permissible center. We transform $B$ and $\overline{B}$ with centers $C$  and $\overline{C}$ respectively, getting 
$B \leftarrow B_1 = (W_1 \to S, J_1, b, E_1)$ and 
$ {\overline B_0} \leftarrow {\overline B_1} = ({\overline W}_1, {\overline J}_1, b, {\overline E}_1) $ respectively. By  Proposition \ref{P:bloui}, both the pull-backs of 
$\pi _u : V \to W$ and  $\pi _v : V \to W$ via $W \leftarrow W_1$ (the blowing-up of $W$ with center $C$) may be identified to the blowing-up $V_1$ of $V$ with center 
$\overline{C}$. Hence we get etale morphisms 
$\pi_{1u}, ~ \pi_{1v}$ from $V_1$ to $W_1$ induced by $\pi_{u}$ and $\pi_{v}$ respectively. Notice that again, by the etaleness, both
$\pi_{1u}^{*}(J_1)$ and  $\pi_{1v}^{*}(J_1)$ may be identified to 
${\overline J}_1$, hence they are equal. Similarly, if $u_1$ (resp. $v_1$) define the strict transform of $Z$ (resp. $Z'$) to $W_1$, then 
$\pi_{1u}^{*}(u_1) = \pi_{1v}^{*}(v_1) = \overline{u}_{1}$, where 
$\overline{u}_{1}$ is the proper transform of $\overline{u}$ to $V_1$. Note that (by  \ref{V:is5.0}) these hypersurfaces are again adapted.

A formal argument with cartesian (i.e., fiber product) squares shows that 
$\pi_{1u}=\pi_{1v}$ over ${\mathcal V}(u_1 - v_1)$ (in the sense explained after the statement of \ref{T:etal}) and moreover, if we choose a point $y_1 \in W_1$ lying over $y_0:=y$, then there is a point $ z_1 \in V_1$ lying over $z_0$ such that 
$\pi _{u1}(z_1)=\pi _{v1}(z_1)=y_1$.  

Now the procedure may be iterated: if $C_1$ is a permissible $B_1$-center, then 
$\pi_{1u}^{-1}(C_1) = \pi_{1v}^{-1}(C_1) : = \overline{C_1}$ and this is a permissible center for $\overline{B}_1$, we take the transforms of $B_1$ and $\overline{B}_1$ with centers $C_1$ and $\overline{C_1}$ respectively, and proceed as before. Iterating, we obtain the following result:
 \end{voi}
\begin{thm}
\label{T:etaler}
Let 
$$(1) \qquad B_0 \leftarrow B_1 \leftarrow \cdots \leftarrow B_r$$
be a permissible sequence of $A$-basic objects and transformations, where $B_0$ is the homogenized basic object associated to a nice basic object $ B=(W \to S,I,b,E)$, we write 
$B_i=(W_i \to S, J_i,b,E_i)$ (so, $J_0 = \cH (I/S), \, W_0=W$), $p_i:W_i \to W_{i-1}$ the corresponding blowing-up morphism, $i=1, \ldots, r$. Let $Z$ and $Z'$ be $B$-adapted hypersurfaces, defined by elements $u$, $v$ of $\Gamma (W, \cO _{W})$ respectively, $y_0 \in Z \cap Z'$ a point of $W$. Then there are a permissible sequence of $A$-basic objects and transformations 
$\overline{B_0} \leftarrow \overline{B_1} \leftarrow \cdots \leftarrow \overline{B_r} $
(we write $\overline{B_i}=(V_i \to S, \overline{J_i},b,\overline{E_i})$, $q_i$ denotes the $i$-th blowing-up morphism $V_i \to V_{i-1}$), and for each $i$ etale morphisms 
$\pi _{iu}$ and $\pi _{iv}$ from $V_i$ to $W_i$ such that:
 
 (i) $\pi _{iu}^{*}(B_i) = \overline{B_i} = \pi _{iv}^{*}(B_i)$, in particular 
 $\pi _{iu}^{*}(J_i) = \pi _{iv}^{*}(J_i)$
 
 (ii) $\pi _{iu}^{*}(u_i) = \pi _{iv}^{*}(v_i)$, where $u_i$ and $v_i$ denote the strict transforms of $u$ and $v$ to $W_i$, respectively.
 
 (iii) $\pi _{iu}=\pi _{iv}$ over ${\mathcal V}(u_i-v_i)$
 
 (iv) for all $i$, $\pi _{iu}^{-1}(C_i)=\pi _{iv}^{-1}(C_i)$ and this is the center of the $i$-th transformation $\overline{B_i} \leftarrow    \overline{B_{i+1}}$
 
 (v) for all $i$, the squares 
 
\begin{displaymath}
\begin{array}{ccccccccc}
%(1)
{V_{i-1}}&{{\stackrel{q_i} \leftarrow}}&{V_i}&{}&{}&{V_{i-1}}&{{\stackrel{q_i}      \leftarrow}}&{V_i} \\

%(2)
{\enskip{\downarrow \!{\scriptstyle{\pi _{u,i-1}}}}}&{}&{\enskip
{\downarrow \! {\scriptstyle{\pi _{ui}}}}}&{}&{}&     
{\enskip{\downarrow \!{\scriptstyle{\pi _{v,i-1}}}}}&&{\enskip{\downarrow \! 
{\scriptstyle {\pi _{vi}}}}}\\

%(3)

{W_{i-1}}&{{\stackrel{p_i} \leftarrow}}&{W_i}&{}&{}&{W_{i-1}}&{{\stackrel{p_i}      \leftarrow}}&{W_i}

\end{array}
\end{displaymath}
are cartesian.

(vi) If $y_i \in W_i$ is such that $p_i(y_i)=y_{i-1}$, $i=1, \ldots, r$, then there are points $z_i \in V_i$, $i=1, \ldots, r$, such that 
$\pi _{ui}(z_i)=\pi _{vi}(z_i)=y_i$ and $q_i(z_i)=z_{i-1}$, C. Moreover, if 
$z_0, \ldots, z_{r-1}$ as above are given, 
%satisfying $\pi _{ui}(z_i)=\pi _{vi}(z_i)=y_i$ and $q_i(z_i)=z_{i-1}$, $i=1, \ldots, r-1$, 
then we can find $z_r \in V_r$, lying over $z_{r-1}$ such that  
$\pi _{ur}(z_r)=\pi _{vr}(z_r)=y_r$ and $q_r(z_r)=z_{r-1}$
\end{thm}

\begin{voi}
\label{V:hi2.0}
Next we discuss  a version of \ref{T:etaler} relative to inductive objects. In it, we use the 
%notation of \ref{T:etaler} and, in addition
following notation and assumptions. Let 
$B_0=(W_0 \to S, J_0,b,E_0)$ be as in \ref{T:etaler}, i.e., the homogenized $A$-basic object associated to a nice basic object $ B={(W_0 \to S,I,b,E)}$, 
$Z$ and $Z'$ be $B$-inductive hypersurfaces, defined by elements $u$, $v$ of $\Gamma (W_{0}, \cO _{W_{0}})$ respectively. Let $C_0\subset Z \cap Z'$ be a strongly permissible center for both ${B^*}_0:=(B_0)_Z$ and ${B'^*}_0:=(B_0)_{Z'}$. Hence, writing $B_0:=B$,  $C_0$ is also a $B_0$-permissible center. Consider $B^*_1={\cT}(B^*_0,C_0)$, $B'^*_1={\cT}(B'^*_0,C_0)$  and  $B_1={\cT}(B_0,C_0)$. By \ref{V:is5.0} we may identify 
$us(B^*_1)$ and  $us(B'^*_1)$ with the strict transforms $Z_1$ and $Z'_1$ of $Z_0:=Z$ and $Z'_0:=Z'$ to $W_1=us(B_1)$ (via the blowing-up $W_1 \to W$ with center $C_0$) respectively; again $Z_1$ and $Z'_1$ will be $B_1$-inductive hypersurfaces. Let $C_1 \subset Z_1 \cap Z'_1$ be a permissible center simultaneously for $B^*_1$, $B'^*_1$ and $B_1$. Transform these objects with center $C_1$, obtaining $A$-basic objects 
$B^*_2$, $B'^*_2$ and $B_2$ respectively. Again by \ref{V:is5.0} we may identify $us(B^*_2)=Z_2$, $us(B'^*_2)=Z'_2$, where $Z_2$ and $Z'_2$ are the strict transforms of $Z_1$ and $Z'_1$ to $W_2=us(B_2)$, and so on. We assume this can be repeated, eventually obtaining strongly permissible sequences of $A$-basic objects (\ref{V:is10.0}):
$$(1) \qquad (B_0)_{Z} = {B^*}_0 \leftarrow \cdots \leftarrow {B^*}_r$$  and
$$(2) \qquad (B_0)_{Z} = {B'^*}_0 \leftarrow \cdots \leftarrow {B'^*}_r $$       
where in each case we have used (in the sense just explained) 
the same centers 
$C_i \subset Z_i \cap Z'_i$, where  $Z_i = us({B^*}_i)$,  $Z'_i = us({B'^*}_i)$, $i=0,\ldots, r-1$, as well as one:
 $$(3) \qquad {B}_0 \leftarrow \cdots \leftarrow {B}_r$$
where the center of $B_i \leftarrow B_{i+1}$ is again $C_i$. Since, for all $i$, $Z_i$ and $Z'_i$ are $B_i$-inductive hypersurfaces in $W_i$, 
it makes sense to take a point $y \in Z_i \cap Z'_i$.  We still denote by $p_i$ the morphisms $Z_i \to Z_{i-1}$ and $Z'_i \to Z'_{i-1}$ induced by $p_i:W_i \to W_{i-1}$ of \ref{T:etaler}. Let 
${\pi}_i :Y_i \to Z_i$ and ${ \pi '}_i:Y_i \to Z'_i$ be  the morphisms induced by $\pi _{u,i}$  and $\pi _{v,i}$ respectively,  $i=0, \ldots, r$. Then we have:
\end{voi}
\begin{thm}
\label{T:etalez} 
The hypotheses and notation are  those of   \ref{T:etaler} and  \ref{V:hi2.0}, so we have the sequences (1), (2), (3) there introduced. 
 Then, there is a permissible sequence of $A$-basic objects and transformations ${\tilde B}_0 \leftarrow \cdots \leftarrow {\tilde B}_r$ (we write 
${\tilde B}_i=(Y_i \to S, {\tilde J}_i, b!, {\tilde E}_i)$, ${q^*}_i:Y_i \to Y_{i-1}$ the induced morphism, for all $i$), and for each $i$ etale morphisms  
${\pi}_i:Y_i \to Z_i$, ${\pi}'_i:Y_i \to Z'_i$, such that 
${{\pi}_i}^*({B^*}_i)= {\tilde B}_i  ={{\pi}'_i}^*({B'^*}_i)$ and the squares 

\begin{displaymath}
\begin{array}{ccccccccc}
%(1)
{Y_{i-1}}&{{\stackrel{{q^*}_i} \leftarrow}}&{Y_i}&{}&{}&{Y_{i-1}}&{{\stackrel{{q^*}_i}      \leftarrow}}&{Y_i} \\

%(2)
{\enskip{\downarrow \!{\scriptstyle{{\pi} _{i-1}}}}}&{}&{\enskip
{\downarrow \! {\scriptstyle{{\pi} _{i}}}}}&{}&{}&     
{\enskip{\downarrow \!{\scriptstyle{{\pi}' _{i-1}}}}}&&{\enskip{\downarrow \! 
{\scriptstyle {{ \pi}' _{i}}}}}\\

%(3)

{Z_{i-1}}&{{\stackrel{p_i} \leftarrow}}&{Z_i}&{}&{}&{Z'_{i-1}}&{{\stackrel{p_i}      \leftarrow}}&{Z'_i}

\end{array}
\end{displaymath}

{\noindent are cartesian.} \\
Moreover, if $y_i \in Z_i \cap Z'_i$ is such that $p_i(y_i)=y_{i-1}$, $i=1, \ldots, r$ and 
$z_i \in Y_i$, $i=1, \ldots, r-1$ is such that 
${\bar \pi}_{i}(z_i)={\bar \pi}'_{i}(z_i)=y_i$ 
and ${q^*}_i(z_i)=z_{i-1}$,  
 then there is a point $z_r \in V_r$, such that 
 $q^*_r(z_r)=z_{r-1}$ and 
${{ \pi}_r} (z_r)={{{ \pi}'}_r}(z_{r})=y_r$.
\end{thm}

\begin{proof} Let $Y_0:={\pi_{u,0}}^{-1}(Z_0)={\pi_{v,0}}^{-1}(Z'_0)$. Then, the induced morphisms 
$\pi_{0}:Y_0 \to Z_0$ and 
$\pi'_{0}:Y_0 \to Z'_0$ 
are etale. From this fact, 
${{\pi}_0}^*({B^*}_0)=  {{\pi}'_0}^*({B'^*}_0)= {\overline{B_0}}_{Y_0} $ 
and $ \overline{C_0}$ is a strongly permissible center for the inductive object 
${\overline{B_0}}_{Y_0}$. Let 
$({\overline{B_0}}_{Y_0})_1:={\cT}({\overline{B_0}}_{Y_0},\overline{C_0} )$ (\ref{V:is6.0} (b)). Then, using the notation of \ref{V:hi1.0}, 
$Y_1:=us( ({\overline{B_0}}_{Y_0})_1)$  
may be identified to the hypersurface of $V_1$ defined by ${\bar u}$ and $\overline{C_1}$ to a closed subscheme of $Y_1$. The induced morphisms 
$\pi _1:{Y_1 \to Z_1}$ and  
$\pi '_1:Y_1 \to Z'_1$ are again etale, which implies that 
${{\pi}_1}^*({B^*}_1)   ={{\pi}'_1}^*({B'^*}_1)= ({\overline{B_0}}_{Y_0})_1 $  and 
$\overline{C_1}$ is  a permissible center for both
$({\overline{B_0}}_{Y_0})_1$ and $\overline{B_1}$. Continuing in this way, we get a permissible sequence of $A$-basic objects 
$${\overline{B_0}}_{Y_0} \leftarrow ({\overline{B_0}}_{Y_0})_1 \leftarrow \cdots \leftarrow ({\overline{B_0}}_{Y_0})_r$$ 
obtained by using the centers $\overline{C_i}$, $i=0, \ldots, r-1$ of \ref{T:etaler}. Let  $Y_i := us(({\overline{B_0}}_{Y_0})_i)$,  
$i=0, \ldots r$. Then
${ \pi}_i :Y_i \to Z_1$ and 
 ${ \pi '}_i:Y_1 \to Z'_i$  (the morphisms induced by $\pi _{u,i}$  and $\pi _{v,i}$),  respectively)  are etale,  and 
 ${{\pi}_i}^*({B^*}_i)= {(\overline{B_i})}_{Y_i}   ={{\pi}'_i}^*({B'^*}_i)$, for all $i=1, \ldots, r$. 
The other assertions are consequences of \ref{T:etaler}
\end{proof}
 
\section{The associated locally nice object}
\label{S:AN} 

In this section we discuss how to associate to a basic object $B$ over an artinian local ring $A \in {\cA}$, locally,   a nice object $B''$, extending the results of \ref{V:ag8} ($\gamma$). The presentation  follows that of 
\cite{EV}, where one works over a base field. In particular, the construction of $B''$ is done via an intermediate object $B'$, which has some but not all the desired properties. Since  the discussion in \cite{EV} is rather sketchy, and some changes are necessary in the present context, we provide some details.

Throughout this section, $A \in {\mathcal A}$ is an artinian ring as in \ref{V:bn1.0}, $S=\Spec (A)$. Given an $A$-basic object $B=(W \to S, I,b,E)$, to say ``$C$ is a center of $B$'' means: $C$ is a center for the underlying pair $(W\to S, E)$, see \ref{V:bn5.0}. (To be a $B$-center is stronger, see \ref{V:bo1.0}). 

We shall use the notation and results of \ref{V:ag6} and \ref{V:dt2.0} about the functions  $\omega_{r}$ and $t_r$.
\begin{Def} 
\label{D:B'} Consider a  $\o$-permissible $A$-sequence ${ B_0 \leftarrow \cdots \leftarrow B_r}$, let ${{\rm max}(\o_{r})=b_r/b}$. We associate to $B_r=(W_r \to S,I_r,b,E_r)$ the following basic object 
$B'_r$. There are two cases: 
\\
($\alpha$) $b_r \geq b$ . We define 
$B_r'=(W_r \to S, \bbar{I_r}, b_r, E_r)$ (see \ref{V:bo2.0} for $\bbar{I_r}$),
\\
($\beta$) $b_r < b$. In the notation of Proposition \ref{P:nuom}, let    
$E_r=(H_1,\ldots, H_{m+r}$) (where $H_{1}, \ldots, H_{m}$ are the strict transforms of the hypersurfaces that appear in $E_0$), 
$\scr{C}_{r}=I(H_{m+1})^{a_1}\ldots I(H_{m+r})^{a_r}$; we set 
$B_r'=(W_r \to S, {\bbar{I_r}}^{b-b_r}+\scr{C}\,^{b_r}, b_r(b-b_r), E_r)$. 
\end{Def} 
 \begin{pro}
\label{P:proB'} With notation and assumptions as in \ref{D:B'}, the $A$-basic object $B'r=(W_r \to S, J_r, b', E_r)$ has the following properties:
\\
(a) $(B_r')^{(0)} = (B_r^{(0)})'$ 
\\
(b) A center $C \subset W_r$ is $\o$-permissible (\ref{V:dt2.0}) if and only if it is $B_r'$-permissible.
\\
(c) For every $B_r'$-center $C$, $\nu(J,C)=\nu(J^{(0)},C^{(0)})=b'$, where $b'$ is the index of $B_0'$.
\\
(d) Let $C \subset W_r$ be a center which is $\o$-permissible (or $B_r'$-permissible, by (b) it is the same.) Consider the transformations $B_r' \leftarrow (B_r')_1$ and 
$B_r \leftarrow B_{r+1}$ with center $C$, and the associated object $B_{r+1}'$. Assume 
$\mathrm{max}(\o_r) = \mathrm{max}(\o_{r+1})$. Then, 
$(B_r')_1=B_{r+1}'$
\end{pro}
\begin{proof}
(a) This readily follows from the definitions.

\smallskip

(b) From \cite{EV} (or by direct verification, which is easy) we have: 
$$(1) \qquad \nu({\bbar{I_r}}^{(0)},C^{(0)}) = b_r$$
Thus, to prove (b) we have to show: for a center $C$, 
$\nu(\bbar{I_r},C) = \nu(\bbar{I_r},C) \geq b_r ~ \Leftrightarrow ~
\nu({J_r},C) = \nu({J_r},C) \geq b' $. In case ($\alpha$), $J_r=\bbar{I_r}$, so the equivalence is obvious (and the $\geq$ symbols are equalities). So, consider case ($\beta$), where $J_r=(\bbar{I_r})^{b-b_r} + {\mathcal C}_r ^{b_r}$ and $b'=b_r(b-b_r)$. 

Check the implication $\Rightarrow$ first. By our assumption and (1), we have 
$\nu({\bbar{I_r}}^{b-b_r},C)=  \nu({\bbar{I_r} ^{(0)}}^{b-b_r},C^{(0)})=b_r(b-b_r)$. On the other hand, since our sequence of $A$-basic objects of \ref{D:B'}
 is permissible, 
$\nu({\mathcal C}_r,C) = \nu({\mathcal C}^{(0)}_r,C^{(0)})=b-b_r $ (the exponents $a_i$ are the same over $A$ or at the level of the special fiber). Then 
$\nu({\mathcal C}_r^{b_r},C) = \nu(({\mathcal C}^{(0)}_r)^{b_r},C^{(0)})=b_r(b-b_r) $ and the implication $\Rightarrow$ follows. 

Now we check the implication $\Leftarrow$. First, we claim that 
$ \nu(\bbar{I_r},C) =  \nu(\bbar{I_r}^{(0)},C^{(0)}) = b_r $. We always have 
$ \nu(\bbar{I_r},C) \leq  \nu({\bbar{I_r}}^{(0)},C^{(0)
}) = b_r $. To check that the sign $<$ cannot hold, we may assume $C$ is irreducible, with generic point $y$, and we work in the ring $R$, the completion of ${\cO}_{W,y}$ with respect to $(a_1, \ldots, a_m)$, an $A$-regular system of parameters which generates the ideal $I(C)_y$. We know 
$R = A'[[a_1, \ldots, a_r]]$ (a power series ring), where $A'$ is an appropriate local $A$-algebra. By (1), we may find in 
${\bbar{I_r}}R$ a series $h= M_0 + \cdots + M_{b_r}+ \cdots$ (sum of homogeneous parts in the power series ring), where $M_{b_r}$ is not a zero-divisor of $R$. Were, by contradiction, 
$\nu(\bbar{I_r},C) < b_r$, then we may find in $\bbar{I_r}$ a series $g$, $g=M'_q + \cdots$ (sum of homogeneous parts) where $M'_q \not= 0$ and $q < b_r$. But then the series 
$g \,(h^{b_r-1}) \in {\bbar I_r}^{b-b_r} \subset J_r$ and (because $M_{b_r}  
$ is not a zero-divisor) has order $< b_r(b-b_r)$, a contradiction. So, the equality holds.
But $B'_r$ induces over the special fiber the object ${{B_r}^{(0)}} '$, in particular 
$ C \subset \sg (B_r):=\sg(B_r^{(0)}$. In 
\cite{EV} it is proved that 
$\nu(I_0^{(0)},C^{(0)}) \ge b_r$. Thus, the implication $\Leftarrow$ is proved. It follows 
  $C$ is a $B_r$ center. 

\smallskip

(c) By (b), such a center $C$ is also a $\o$-center. Then, in case ($\alpha)$ the equality follows from the definition of $\o$-center. In case $(\beta)$, from this definition and Proposition \ref{P:nuom}. Note that in  this case $a_1 + \cdots + a_m = b-b_r$.

\smallskip

(d) The proof is a calculation entirely analogous to that for the case where $A=k$ is a field, and we shall omit it.
\end{proof}

If (using the notation above) $B'_r=(W_r \to S, J_r,b',E_r)$ is the object associated to an $A$-basic object $B_r=(W_r \to S, I_r,b,E_r)$,  given a closed point 
$y \in \sg(B'_r)={\rm Max}(\o _r)$ (\ref{P:proB'}) (b),   
we may find  an element 
$f \in {{\Delta}^{b'-1}(J_r/S)}_{y}$ defining, on a neighborhood of $y$, an $S$-hypersurface $Z$ satisfying condition (A1) of \ref{V:is2.0}. But in general (A2) of \ref{V:is2.0} won't be satisfied, hence $Z$ won't be adapted to 
${B_r '}$. To overcome this difficulty we introduce another $A$-basic object, denoted by  ${B_r}''$ (see \cite{EV}, 9.5). This object $B_r''$ will be essential in inductive arguments.
 
\begin{voi} 
\label{V:an1.0}
Here we use the notation of \ref{V:dt2.0}. Consider a $t$-permissible sequence of $A$-basic objects  
$$ (1)\quad B_0 \leftarrow \cdots \leftarrow B_r$$
where we write $B_j=(W_j \to S, I_j, b,E_j)$. 

(a) We shall say that an open set $U \subseteq W_r$ is {\it amenable} for $B_r$ if the following property holds. As usual, we write $t_r=({\o}_r,n_r)$. Let $\mathrm{max} ({t_r}_{|U}) = (\bbar{\o},\bbar{n})$ (where $\bbar{\o}= \mathrm{max} ({{\o}_r}_{|U})$). Then, we require that for hypersurfaces $H^*_1, \ldots, H^*_{\bbar{n}}$ in $E_r^-$ we have:
$$\mathrm{Max} ({t_r}_{|U})  = \mathrm{Max} ({{\o}_r}_{|U}) \cap H^*_1 \cap \cdots   H^*_{\bbar{n}} \cap U$$
These hypersurfaces are uniquely determined, by the maximality of $\bbar{n}$. Note that this concept depends not just on $B_r$ but also on its position in a $t$-permissible sequence like sequence (1) above.\\
If we may take $U=W_r$ as our amenable open, we say that $B_r$ is amenable.

\smallskip

(b) Note that if, in (1), $x \in \sg(B_r)$, by the upper-semicontinuity of $t_r$ and the definitions we may find an open neighborhood  $U$ of $x$ in $W_r$ such that $t_r(x)$ is the maximum of ${t_r}_{|U}$ and $U$ is amenable for $B_r$.

\smallskip

(c) {\it The object ${B_r}''$}. Let $U$ be a amenable open for $B_r$. Then we define a new $A$-basic object 
$({B_r}_{|U})''$ as follows. Let $B_r ' = (W_r, J_r, b',E_r)$. Then, 
$({{B_r}_{|U}}) '' = (U, I''_r, b',E_r^{+})$, where 
$$I_r''=(J_r + I(H^*_1)^{b'} + \cdots + I(H^*_{\bbar n})^{b'})_{|U}, \quad
E_r^{+}= (E_r \setminus E_r ^{-})_{|U}$$ 
(concerning $E_r ^{-}$, see \ref{V:dt2.0}). By (b), given a point $x \in \sg(B_r)$ always there is a neighborhood $U$ os $x$ such that 
$({B_r}_{|U})''$ is defined.
 
 When $B_r$ is amenable and we take $U=us (B_r)$ we simply write $B_r ''$.
\end{voi}
\begin{pro}
\label{P:proB"} 
Let 
$B_0 \leftarrow \cdots \leftarrow B_r$ be a $t$-permissible sequence of $A$-basic objects, $B_j=(W_j \to S, I_j, b, E_j)$, where $B_r$ is amenable, and $B_r''$ the associated object just introduced. Then we have: \\
(a) $(B_r'')^{(0)} = (B_r^{(0)})''$ 
\\
(b)  A center $C$ for $B_r$ is $t$-permissible if and only if it is $B_r''$-permissible.
\\
(c) For every irreducible $B_r''$-center $C$, $\nu(I_r'',C)=\nu(I_r''^{(0)},C^{(0)})=b'$ (the index of $B''_r$).
\\
(d) Let $C \subset W_r$ be a $B_r$-center which is $t$-permissible (or $B_r''$-permissible, it is the same.) Consider the transformations $B_r'' \leftarrow (B_r'')_1$ and 
$B_r \leftarrow B_{r+1}$ with center $C$, and the  object $B_{r+1}''$ associated to $B_{r+1}$. Assume 
$\mathrm{max}(t_r) = \mathrm{max}(t_{r+1})$. Then, 
$(B_r'')_1=(B_{r+1})''$
\\
(e) $B_r''$ is a locally nice object.
\end{pro}
 \begin{proof} 
(a) easily follows from the definitions. (b) follows from the definition of $I''_r$, using  \ref{P:proB'}(b) and the fact that 
$\mathrm{Max} (t_r) = \mathrm{Max} ({\o}_r) \cap H^*_1 \cap \cdots   H^*_{\bbar{n}}$
 (recall that here $U=W_r$). (c) follows from (c) of \ref{P:proB'} and the fact that such a center $C$ is contained in each hypersurface  $H^*_{i}$.
 
 Concerning (d), let $B_{r+1}''=(W_{r+1},I_{r+1}'', b'', E_r^+)$. We must prove that, via the morphism 
 $p: W_{r+1} \to W_r$ (the blowing up with center a $t$-permissible center $C$, with exceptional divisor $H$),  $I_{r+1}''$ coincides with the controlled transform   $I(H)^{-b''}I_r''{\cO}_{W_{r+1}}$ $=L$. Let $H'_j$ denote the strict transform of $H_j^*$ via $p$, note that $H \in E_{r+1}^{+}$ (because ${\rm max}(t_r)={\rm max} (t_{r+1})$). Letting 
 ${\cO}':={\cO}_{W_{r+1}}$, we have:
 \\
  $L = I(H)^{-b'} I_r''{\cO}'=
 I(H)^{-b'}[I_r' + I(H_1^*)^{b'}+ \cdots + I(H_N^*)^{b'}]{\cO}' =
 \\
 I(H)^{-b'}I_r'+ (I(H)^{-1}I({H_1}^{*}))^{b'} + \cdots +(I(H)^{-1}I({H_N}^{*})^{b'}]{\cO}'=
 \\
 {I_{r+1}}' + {I(H_1')}^{b'}+ \cdots + {I(H_N')}^{b'}$
\\
and the latter is precisely $I_{r+1}''$, as desired.

(e) Consider the smallest index $s$ such that 
${\mathrm{max}}(t_s) = \ldots = {\mathrm{max}}(t_r)$. As usual, $t_j$ denotes the $j$-th $t$-function of the sequence of special fibers. Working with $B_s''$, by 
 \ref{P:proB'} (c) and \ref{V:is3.0}, 
 if $x \in \sg(B_s'')$ is a closed point we may get an open neighborhood $V$ of $x$ in $W_s$ and an element 
$h \in \Gamma(V, \Delta^{b''-1}(I_s''/S))$ defining an $A$-hypersurface on $V$. This is an adapted hypersurface (for $B_r''$ restricted to $V$). Indeed, clearly (A1) of \ref{V:is2.0} is satisfied, and (A2) is obvious because $E_s^*$ is empty. Letting $x$ vary in the (dense) set of closed points of 
 $\sg(B_s'')$  we get an open covering 
$\{ V_j \}$, $j=1, \ldots, q$ of $\sg(B_s'')$ and hypersurfaces $Z_j$ defined on $V_j$ and adapted for 
the restriction of $B_s ''$ to ${V_j}$, for all $j$. Transforming with center $C_s$, letting $V_{1j}$ be the pre-image of $V_j$ in $W_{s+1}$ and $Z_{1j}$ the strict transform of $Z_j$, by Giraud's Lemma \ref{L:gir}, $Z_{1j}$ is locally defined by an element of 
$\Delta ^{b''-1}(I_{s+1}''/S)$ and is transversal to ${E_{s+1}}^+ $ (which consists of the exceptional divisor $H_1$ of the transformation only), because $C_s \cap V_j$  is contained in $Z_j$. Thus, $Z_{1j}$ is adapted to the restriction of $B_{s+1}''$ to $V_{1j}$, for all $j$. Reiterating, we get 
a covering of $\sg(B''_r)$ by open sets $V_{(r-s),j}$, $j=1, \ldots, q$ of $W_r$ 
 and on each $V_{(r-s),j}$ an $A$-hypersurface  $Z_{(r-s),j}$, adapted to the restriction of $B_r''$ to $V_{(r-s),j}$, for all $j$. Thus, $B_r''$ is locally nice, as claimed.
\end{proof}
 \begin{voi}
\label{V:an2.0}
{\it More on the resolution algorithm over a field.} 
Now we check the assertions made in \ref{V:ag9}, by using  the theory of section \ref{S:HI} and the present one (in case $A=k$, a field). 
 
 We take  as the open set $U$ of \ref{V:ag9} an amenable neighborhood of $w$. Then the nice object $B_s''$ of \ref{V:an1.0} is defined on $U$, hence its associated homogenized object $\cH ({B_s}''):=(HB_s'')$ (which is again nice) admits an adapted hypersurface $Z_{sU}$ (or simply $Z_s$) containing $w$, defined on $U$. From the assumption ``$\dim (\ma(t_s) < d-1$'' made in \ref{V:ag7}, $Z_s$ is inductive. This is the hypersurface $Z_s$ of 
 \ref{V:ag7}, while our object $B^*_s$ is $(HB''_s)_{Z_s}$ (\ref{V:is6.0} (b)). (More properly, we should write, e.g.,  
 $B^*_{s,U}={{\cH}(({B_s}_{|U})'')_{Z_{sU}}}$ rather than $B^*_s$, we try to simplify the notation).\\
Now we check properties (i)-(iv) of \ref{V:ag7}.

(i) By \ref{V:is8.0}, 
  $\sg(B^*_s)=\sg(H{B_s}'')$. Since, by Proposition \ref{P:wide}, ${B_s}''$ and $H{B_s}''$ are W-equivalent, we obtain  $\sg(H{B_s}'')=\sg ({B_s}'')$. Combining these facts, we get (i) of \ref{V:ag7}.

(ii) By repeated application of \ref{V:is8.0} (i), (ii), from the algorithmic resolution (2) of \ref{V:ag7} (obtained by the inductive hypothesis) we get a permissible sequence 
$$(1) \qquad (H{B_s}'') \leftarrow (H{B_s}'')_1 \leftarrow \cdots $$
By the W-equivalence of $(H{B_s}'')$ and ${B_s}''$ (\ref{P:wide}), the sequence (1) induces (by using  the same centers) a permissible sequence 
$$(2)\qquad {B_s}'' \leftarrow ({B_s}'')_1  \leftarrow \cdots $$
which, by  \ref{P:proB"} (b), induces (by using again the same centers) a $t$-permissible sequence 
$$(3) \qquad {\widetilde{B}_{s}} \leftarrow \widetilde{B}_{s1} \leftarrow \cdots $$
as claimed in (ii) of \ref{V:ag7}.

(iii) By \ref{L:gir}, $us(B^*_{s+j})$ may be identified to the strict transform $Z_j$ of $Z_s$ to the scheme $us((H{B_s}'')_j)$. By the W-equivalence of ${B_s}''$ and $(H{B_s}'')$, 
$us((H{B_s}'')_j)=us(({B_s}'')_j)$. But, from our assumption 
``${\rm max}({\widetilde t_{s}})= \ldots = {\rm max}({\widetilde t_{s+q}})  $'', we get $({B_s}'')_j= {(\widetilde{B}_{r+j})}$. Since 
$Z_j = us(B^*_{s+j})$,  the fact that $us(B^*_{s+j})= {\rm Max} ({\widetilde t_{s+j}}) $ follows from \ref{P:proB"}.

(iv) Since $\sg(B^*_{s+j})$ can be identified to  ${\rm Max} ({\widetilde t_{s+j}}) \subseteq \sg(B_{s+j})$, ${\widetilde g_{j}}$ defines a function as claimed. The fact that its value ${\widetilde g_{j}}(x)$ is well-defined comes from the following observation, by restricting to a suitable amenable open neighborhood of $x$ and using the compatibility of our constructions with etale pull-backs, in particular open inclusions.

Let $B=(U \to S, I, b, E)$ be a nice basic object, $(HB)$ its associated homogenized (nice) object, $Z,\,Z'$ adapted hypersurfaces for $(HB)$, ${B}^*=(HB)_{Z} $, ${B'}^*=(HB)'_{Z'}$. Consider the resolution functions 
$g_0, \ldots, g_{m}$ of ${B}^*$ and 
$h_0, \ldots, h_{m'}$ of ${B'}^*$
 respectively (known by induction on the dimension). We claim that $g_i=h_i$, for all $i$. This means that $m=m'$, for all $i$ both $g_i$ and $h_i$ have the same domain, and they take equal values. \\
Indeed, for $j=0$, note that, by \cite{EV},  
$S_0:=\sg ({B}^*)=\sg((HB)_Z)= \sg((HB))=\sg((HB)_{Z'})= \sg ({B}'^*)$  
hence the domains of $g_0$ and $h_0$ agree. If $x \in S_0$, by \ref{T:etalez} we get etale neighborhoods 
$p:(Y,y) \to (Z,x)$ and $p':(Y,y) \to (Z',x)$ such that 
${p}^*({B}^*)={p'}^*({B'}^*):= \bbar B$. By the compatibility of the algorithm with etale morphisms, 
 ${p}^*({{g_0}})$ and ${p'}^*({{h_0}})$ agree on $\sg   ({\bbar B})$. Hence, 
$ {\widetilde {g_0}} (x) =  {p}^*({\widetilde {g_0}})(y)      =  {p'}^*({\widetilde {h_0}})(y) ={\widetilde {h_0}} (x)$. Consequently, 
${\rm Max} (g_0) = {\rm Max} (h_0)$ (say, $=C$). But $C$ is the first center we use in the resolution both for $B^*$ and ${B'}^*$. By \cite{EV}, $C$ is also a 
permissible center for $(HB)$ (or its W-equivalent object $B$). Transform these objects with center $C$, getting ${B}^* \leftarrow {{B}^*}_1$, ${B'}^* \leftarrow {{B'}^*}_1$, 
${HB} \leftarrow {{HB}}_1$, ${B} \leftarrow {{B}}_1$ respectively. We know that $us({{B}^*}_1)$ (resp. $us({{B'}^*}_1)$) is the strict transform $Z_1$ (resp. $Z'_1$) of $Z$ (resp. $Z'$) to $U_1:=us((HB)_1)=us(B_1)$ (use \ref{L:gir}). Again 
$S_1=\sg ({{B}^*}_1)=\sg((HB_1))=\sg ({{B'}^*}_1)$ (both $=\sg (B_1)$). Using again 
\ref{T:etalez}, given an arbitrary point $x_1$ in $S_1$, lying over $x_0 \in S \subset U$, we get commutative cartesian diagrams 
\begin{displaymath}
\begin{array}{ccccccccc}

{Y}&{{\stackrel{{q}_1} \leftarrow}}&{Y_1}&{}&{}&{Y}&{{\stackrel{{q}_1}      \leftarrow}}&{Y_1} \\

{\enskip{\downarrow \!{\scriptstyle{p}}}}&{}&{\enskip
{\downarrow \! {\scriptstyle{p_1}}}}&{}&{}&     
{\enskip{\downarrow \!{\scriptstyle{p' }}}}&&{\enskip{\downarrow \! 
{\scriptstyle {p' _{1}}}}}\\

{Z}&{{\stackrel{\rho_1} \leftarrow}}&{Z_1}&{}&{}&{Z'}&{{\stackrel{{{\rho}'}_1}      \leftarrow}}&{Z'_1}

\end{array}
\end{displaymath}
(where, by the compatibility of the resolution algorithm with etale pull-backs, 
${{q}_1} $ is the first transformation in the algorithmic resolution process for
${\bbar {B}} $) such that 
${p_1}^*({{B}^*}_1)={{{p'}_1}^*}({{B'}^*}_1):= {\bbar B}_1$. Then, both 
${p_1}^*({g_0})$ and ${p'_1}^*({{h_1}})$ must be the first resolution function for ${\bbar B}$, hence they must agree. Exactly as in case $j=0$ we conclude that $g_1(x_1)=h_1(x_1)$. Since $x_1$ was arbitrary in $S$, $g_1=h_1$. We may iterate this procedure, each time using \ref{T:etalez} to show that $m=m'$ and $g_i = h_i$ for all $i$.

We should check that the functions $g_j$ thus defined are compatible with etale pull-backs (we used this fact in the inductive step). But this is simple, and we leave it to the reader (or see \cite{BEV} or \cite{EV}).  
\end{voi} 

\section{Algorithmic equiresolution}
\label{S:AE} 
\begin{voi} 
\label{V:ae1.0}
In this section 
we  work with the VW-resolution algorithm (for basic objects over fields of characteristic zero) discussed in 
\ref{V:ag7}, \ref{V:ag9}  and \ref{V:an2.0}.  Henceforth this will be referred to as {\it the algorithm}. As usually, $\cA$ is the class of rings of \ref{V:bn1.0}.

We shall prove Theorem \ref{T:main}. To that effect, given a ring $A \in {\cA}$ and 
an $A$-basic object  $B=(W \to S,I,b,E)$, $S=\sp \, (A)$,  we shall introduce certain conditions 
${\cE}_{j},\, j=0,1, \ldots$, which may be valid or not. 
 Intuitively, these conditions work as follows. Let  
 $$(1) \qquad B^{(0)}:= B_0^{(0)} \leftarrow \cdots \leftarrow B^{(0)}_r$$
 be the algorithmic resolution of the  fiber $B^{(0)}$ of $B$, obtained by using algorithmic resolution functions $g_i, \, i=0, \ldots, r-1$, which yield resolution centers $C_i=\ma (g_i)\subset us(B^{(0)}_i)$. For $0 \le j < r$ the validity of ${\cE}_{j}$ means that all the operations  involved to complete the first $j-1$ steps in the algorithmic resolution (1) (i.e., what is necessary to obtain the $j$-truncation of (1)) can be extended, in a natural way, along $S=\Spec A$. Theorem \ref{T:main} will be an immediate consequence of this theory. Let us discuss these conditions more carefully.

We define  conditions $\cE _j$ for every non-negative integer. We declare them  vacuously valid for  
$j \geq r$. In particular if, in (1), $r=0$, i.e., $\sg (B^{(0)})= \emptyset$, then $B$ is algorithmically equisolvable. So, assume $r > 0$. Then, for values  $j=0, \ldots, r-1$, condition ${\cE}_{j}$, if valid, defines centers $C_i$, $i=0, \ldots, j$, where more precisely $C_0$ is a center for $B=B_{0}$, $C_1$ is a center for 
$B_1={\cT}(B_{0},C_{0}), \ldots, C_{j}$ is a center for $B_{j}={\cT}(B_{j-1},C_{j-1})$. Thus, if 
${\cE}_{j-1}$ is valid, $1 \leq j \leq r$, one has an induced permissible sequence of $A$-basic objects  
$$(2)_{j} \qquad  B_{0} \leftarrow \cdots \leftarrow B_{j}$$
These conditions ${\cE}_{j}$ will satisfy the following properties: 

\smallskip

(a) The sequence $(2)_j$ induces, by taking fiber, the $j$-truncation of the algorithmic resolution of $B^{(0)}$. 

\smallskip

(b) If $i < j$, then the sequence $(2)_i$ is the $i$-truncation of the analogous sequence $(2)_j$.

\smallskip

(c) If $f:W' \to W$ is an etale morphism  and $B'$ is the $A$-basic object induced from $B$ by pull-back,    if condition ${\cE}_{j}$ is valid for $B$ then it is also valid for $B'$  and the centers $C'_{j}$ associated to $B'$ are the pull-backs of those associated by ${\cE}_{j}$ to $B$.

\smallskip

(d) The conclusion of (c) remains true if 
 $W'=W \times _{S} S'$, where $S' \to S$ is induced by a ring homomorphism (of rings in $\cA$) $A \to A'$ and $f:W' \to W$ is the first projection (i.e., ${\cE}_j$ is stable under change of the base ring $A$). 
 
These conditions, satisfying the mentioned properties, will be defined inductively on 
$d= {\rm dim} \, B$. This will be done in the following sub-sections. 
 But, accepting the results just stated, we may prove Theorem \ref{T:main} as follows.
 \end{voi}

 \begin{voi}
 \label{V:prova}
 {\it Proof of Theorem \ref{T:main}}. Given an $A$-basic object $B$, with fiber $B^{(0)}$, let condition ${\cE}_i$ be valid for $i=0, \ldots , s-1$. Then define $e(B) := s$. There is an associated $A$-permissible sequence 
 $B=B_0 \leftarrow \cdots \leftarrow B_s$ (this is $(2)_j$ of  \ref{V:ae1.0}, with $j=s$), using centers $C_i \subset us(B_i), ~ i=0, \ldots,s-1$. These are the centers we associate to $B$ to complete the definition of the function of Theorem \ref{T:main}. The claimed properties of this function follow from properties (a), (b), (c) and (d) of conditions ${\cE}_i$ stated in \ref{V:ae1.0}.
 \end{voi}
 
 \begin{Def}
 \label{D:aleq} Given an $A$-basic object $B$ as in \ref{V:ae1.0} we say that it is {\it algorithmically equisolvable} if 
conditions ${\cE}_{j}$ are valid for every $j \ge 0$. Equivalently, 
since conditions ${\cE}_{j}$ are vacuously valid for 
 $j \geq r$,  we may  require that ${\cE}_{j}$ be valid for $j=0, \ldots, r-1$ (with $r$ as in sequence (1) in \ref{V:ae1.0}) or that in  Theorem \ref{T:main} it must be $e(B)= \ell (B^{(0)})$. 
  
  In this case, the sequence $(2)_r$ of \ref{V:ae1.0} that is determined will be called the {\it algorithmic equiresolution sequence of} $B$.
 
 Of course,  an  $A$-basic object may not be algorithmically equisolvable. For instance, the basic object of  Example \ref{E:nohay} has non-empty singular locus but it does not admit any permissible center. Hence it cannot have any equiresolution.
 \end{Def}
 Now let us precisely define conditions ${\cE}_i$.
 
\begin{voi}
\label{V:ae2.0}
{\it Defining ${\cE}_j$ when dim $(B)=1$.} Let (1) of \ref{V:ae1.0} be  the algorithmic resolution of the special fiber $B^{(0)}$. We shall define, for $0 \leq j < r$, conditions ${\cE}_{j}$, in such a way that if ${\cE}_{j}$ is valid, then the resulting sequence $(2)_j$ (of \ref{V:ae1.0}) is $\rho$-permissible (see \ref{V:det4.0'}). 
 
Start with ${\cE}_0$. Consider $B:=B_0$. Notice that necessarily max$(\o _0)>0$. In this case, the zeroth algorithmic center $C^{(0)}$ of the fiber $B^{(0)}$ is Max($t_0)$. To define ${\cE}_0$, take an open cover 
$\{U_i\}$ of Max($t_{0}$) such that each $U_{i}$ is amenable for $B$, hence $ (B_{|U_i})'' :=  B''_i={(U_i \to S, I''_i,b'', E''_i)}$ is defined (see \ref{V:an1.0}). Then, to have ${\cE}_0$ valid,  we require that, for all $i$,  
$\Delta^{b''-1}(I''_i/S)$ define a $B''_i$-permissible center $C_i$. Since the locally defined objects $B''_i$ agree on intersections (this follows from the construction in \ref{V:an1.0}), this requirement is independent of the choice of the cover and the different $C_i$ glue together to yield a well-defined center $C$ which is $B$-permissible (because, by \ref{P:proB"}, $C$ restricted to each $U_i$ is $B_{|U_i}$- permissible). This $C$ is the center ${\cE}_0$ defines. Property  (a) is true  because because $(B_{|U_i})''$ restricted to the fiber is 
$({B^{(0)}}_{|U_i})''$ and ${\sg {({B^{(0)}}_{|U_i})''}}= {\mathcal V}(\Delta ^{b-1}({I^{(0)}}_{|{U_i}}))={\rm Max} (t_0) \cap U_i$. Property (b) is clear. Moreover, $C$ is $t$-permissible (\ref{P:proB"} (b)).

Now, if $j>1$, assuming ${\cE}_s$ defined for $s < j$, we introduce condition ${\cE}_j$. Looking at $B^{(0)}_j$, there are two cases: ($\alpha$) max$(\o _j)>0$, ($\beta$) max$(\o _j)=0$. 

Consider case ($\alpha$). To declare ${\cE}_j$ valid, first we  require that conditions ${\cE}_s$ be valid, for $s < j$. Hence, we have a $t$-permissible sequence $B_0 \leftarrow \cdots \leftarrow B_{j}$. Recall that in this case the $j$-th center $C^{(0)}_j$  in (1) of \ref{V:ae1.0} is Max($t_j$).  Apply to $B_{j}$ the technique used in the case $j=0$. Namely, cover Max($t_{j}$) by amenable open sets, so that   nice basic $A$-objects 
$(B''_{j})_i= (U_i \to S,I''_{i},b''_i,E''_i)$  are defined (see \ref{V:an1.0}). To finish the definition of ${\cE}_j$ in this case   we require that the subscheme defined by 
$\Delta ^{{b''_i} -1}(I''_i/S)$ be a $(B_{j}'')_i$-center $C_{j i}$, for all $i$. As above, this is independent of the chosen cover, and these centers patch together to produce a $t$-permissible center $C_{j}$ for $B_j$, which will be the $B_j$-center determined by ${\cE}_j$.   Requirement (a) is treated as above (case $j=0$) and (b) is clear.   
 
Now consider case ($\beta$). Here, the object $B_j$ is premonomial. To have condition  ${\cE}_j$ satisfied we require that $B_j$ be monomial, and we take as the associated center the canonical monomial center. The verification of  properties (a) though (d) is simple.

\end{voi}
\begin{voi}
\label{V:ae2.1} It remains to check, in case ($\alpha$), properties (c) and (d). We are not going to present the details of this verification, either here or  in the other parts of the discussion of conditions ${\cE}_j$. In fact, the verification is either immediate or a  consequence of calculations in certain local rings (see, e.g.,  \ref{V:bn8.0} and \ref{P:ekiv}). These calculations are simple once we use the following observations.

For (c), use the identification of Proposition \ref{P:twocomp}. Concerning (d), note that if $B={(W \to {\Spec A},I,b,E)}$ is an $A$-basic object, $w \in W$ a closed point, $A \to A'$ a homomorphism in $\cA$, $W'=W \times _{A} A' \to W$ the first projection, $w' \in W'$ a point lying over $w$, $R = {\cO}_{W,w} \stackrel{\phi}{\rightarrow}{\cO}'_{W',w'}=R'$ the induced homomorphism, then: (a) an $A$-regular system of parameters 
${\bf a}=(a_1, \ldots, a_n)$ of $R$ 
maps into an $A'$-regular system of parameters 
${\bf a'}=(a'_1, \ldots, a'_n)$ of $R'$, (b) if 
 $\widehat {R}$ (resp. $\widehat {R'}$) is the completion of $R$ with respect to  ${\bf a}$ (resp. of $R'$ with respect to  ${\bf a}'$) then there is an identification of 
 ${\widehat R}$ and the power series ring $  A_1[[x_1,\ldots,x_n]]$ (resp. ${\widehat R} = A'_1[[x_1,\ldots,x_n]]$) where 
 $A_1=R/(a_1, \ldots,a_n)$ (resp. $A'_1=R'/(a'_1, \ldots,a'_n)$), with these, $\phi$ induces the natural homomorphism from 
 $A_1=R/(a_1, \ldots,a_n)$ to $A'_1=R'/(a'_1, \ldots,a'_n)$ (sending each $x_i $ to itself); we also use the corresponding known result in case $A$ and $A'$ are fields (\ref{V:ag10.1}).
 \end{voi}

In the discussion of the  case where the dimension $d$ of $B$ is arbitrary we shall need a lemma
 that we state next. Its proof will be presented in \ref{V:ae5.0}. Informally, it says that if $(HB)$ is the homogenized $A$-object associated to $B$ (nice), $Z$ and $Z'$ are inductive hypersurfaces for $(HB)$ and we assume, inductively on the dimension, that the inductive objects $(HB)_{Z}$ and $(HB)_{Z'}$ are algorithmically equisolvable, then the $0$-th centers $C_0$ and $C'_0$ of their algorithmic equiresolutions coincide: $C_0 = C'_0$. Moreover, under suitable strong permissibility assumptions similar equalities hold for the other algorithmic resolution centers associated to  $(HB)_{Z}$ and $(HB)_{Z'}$.

\begin{lem}

\smallskip

\label{L:pego}
 In the previous notation, assume conditions ${\cE}_{j}$, satisfying (a), (b), (c) of \ref{V:ae1.0}, are defined for every possible $j$ when the dimension of our basic object is $<d$. Let $B=(W \to S,I,b,E)$ be a nice $d$-dimensional basic object over $A$, $(HB):={\cH}(B)={(W \to S, \cH (I/S,b),b,E)}$ (cf. \ref{D:ho}, it is again nice), $Z$ and $Z'$  inductive hypersurfaces for $(HB)$. Let $(HB)_Z$ 
 and $(HB)_{Z'}$ be the corresponding inductive objects. Assume that both 
 $(HB)_Z$ and $(HB)_{Z'}$ are algorithmically equisolvable (i.e., they satisfy conditions  ${\cE}_j$ for all possible $j$, since their dimensions are equal to $d-1$, by our hypothesis these conditions are defined). Consider the corresponding algorithmic equiresolution sequences 
$$(1) \qquad  (HB)_{Z}= ((HB)_{Z})_{0} \leftarrow \cdots \leftarrow ((HB)_{Z})_t$$
$$(2) \qquad (HB)_{Z'}= ((HB)_{Z'})_{0} \leftarrow \cdots \leftarrow ((HB)_{Z'})_t$$ 
(Note that the length of these sequences, which dependes on the special fibers only, must be the same, by the results of \ref{V:an2.0}). 
 Assume that both (1) and (2) are strongly permissible sequences (see \ref{V:is10.0}), determined by centers $C_i \subset us(((HB)_{Z})_i)$, $C'_i \subset us(((HB)_{Z'})_i)$, $i=0, \ldots, t-1$ respectively. Let     
 $$(3) \qquad  (HB)_0 \leftarrow  (HB)_1  \leftarrow \cdots \leftarrow (HB)_t \, , $$
$$(4) \qquad  (HB)_0=(HB)'_0 \leftarrow (HB)'_1 \leftarrow \cdots \leftarrow (HB)'_t $$
be the   induced permissible sequences 
 (see \ref{V:is10.0}). By  \ref{L:gir}, if  
  $Z_i$ and $Z'_i$ are  the strict transforms of $Z$ and $Z'$ to $(HB)_i$ and $(HB)'_i$ respectively, 
 there are identifications 
 $ Z_i = us((HB)_Z)_i$, $ Z'_i = us((HB)_Z')_i$.\\
 Then, using these identifications, $C_0 = C'_0$,  hence $us(HB)_1=us((HB)'_1)$; using this fact, 
 $C_1=C'_1$ (an equality of subschemes of $us((HB)_1)$, hence $(HB)_2=(HB)'_2$), and so on, eventually getting $C_{\lambda}=C'_{\lambda}$, (as subschemes of $us((HB)_{\lambda})=us((HB)'_{\lambda})$, ${\lambda}=0, \ldots, t$. 
\end{lem}
\begin{voi}
\label{V:ae3.0}
{\it Defining conditions ${\cE}_j$ when ${\dim B}=d$, an arbitrary positive integer}. The case $d=1$ being explained, we shall proceed by induction on $d$. So, assuming the definition (having properties (a), (b), (c) and (d) of \ref{V:ae1.0}) known when the dimension of the $A$-basic object is less than $d$, we'll introduce conditions $\cE _j$ for $B$ $d$-dimensional. This will be done again by induction  on  $j$ .

Start with $j=0$. Let $B=B_{0}$ be our  $A$-basic object. Looking at the fiber $B^{(0)}$, its  ${\o}_0$ function must satisfy 
 max $({\o}_{0}) > 0$. Consider  
$M={\rm Max}(t_0)$. There are two cases:  $\dim M =d-1$ or  $\dim M < d-1$.

(I) {\it Case  $\dim M=d-1$}. Here, on the fiber $B^{(0)}$, the zeroth  algorithmic resolution center is $M(1)$, the union of the components of $M$ of codimension one.  We proceed as in the one-dimensional situation. Namely, 
consider locally defined nice $A$-basic objects $B''_i = {(U_i \to S, I''_i, b''_i, E''_i)}$ as in \ref{V:an1.0}, where the open sets $U_i$
cover $M(1)$ and, for all $i$, $U_i \cap Y = \emptyset$, for any irreducible component $Y$ of codimension $>1$ of $M$. (This is possible, since by \cite{EV} $M(1)$ is regular.) Then  we declare condition $\cE _{0}$ valid if for each $i$ the $W$-ideal $\Delta ^{b''_i - 1}(I''_i/S)$ defines a permissible center $C_i$. Again, since the considered nice basic objects agree on each intersection $U_i \cap U_{i'}$, the different $C_i$ define a $B$-center $C$. We take $C$ as the center that 
 $\cE _0$ associates to $B$. 
  From the construction of the nice object $B''_i$ (\ref{V:an1.0}), this is independent of the choice of the open cover $\{U_i\}$. 
 Conditions (a), (b), (c) and (d) of \ref{V:ae1.0} are easily checked (see \ref {V:ae2.0} and \ref {V:ae2.1}). This finishes the case  where $j=0$ and dim($M$)=$d-1$.

 (II) {\it Case $\dim M < d-1$}. 
 In this situation, the zeroth algorithmic center $C^{(0)}_{0}$ in the resolution sequence (1) of \ref{V:ae1.0} is defined inductively, on the dimension (using locally defined  basic objects ${B^{(0)}_{0}}''$, inductive hypersurfaces, homogeneization (\ref{D:ho}) and the corresponding inductive objects). Accordingly, we  
 cover $M=\max (g_0)$ by amenable opens $U_i$ ($i$ in a suitable set, $g_0$ is the $0$-th resolution function of $B_0^{(0)}$). Hence   on each $U_i$ there is a nice $A$-basic object $B''_i= {(U_i \to S, I''_i,b''_i, E''_i)}$ admitting an adapted hypersurface $Z_i \subset U_i$ (see \ref{V:an1.0}). For each $i$ take the homogenized object $(HB''_i):={\cH}(B''_i)=(U_i \to S,I^*_i,b''_i,E''_i)$ (where $I^{*}_i := {\cH}(I''_i/S,b''_i)$). The object $(HB''_i)$ is  again  nice, admitting $Z_i$ as an adapted hypersurface. From our assumption on dim($M$), by  \ref{P:proB"} (b) and \ref{V:is7.0}, $Z_i$ also satisfies (A3) of \ref{V:is2.0}, i.e., $Z_i$ is inductive. Consider the inductive object 
 $B^*_i := (HB''_i)_{Z_i}=(Z_i \to S, {\cC}(I^*_i/S,Z_i), \tilde{b_i}, \tilde{E_i})$, see 
 \ref{V:is6.0}  (a). By induction we know what it means the expression ``condition $\cE _s$ is valid for $B^*_i$, for every $i$ and $s$''. So, we declare: $\cE _0$ is valid for $B$ if the following conditions hold: (a) for all $i$, 
  $\cE _s$ is valid for $B^*_i$, for all $s \ge 0$; (b) the center $C_i$ that condition ${\cE}_0$ is associates to  $B^*_i$, is {\it strongly permissible} (\ref{V:is10.0}). Then 
  we claim that these centers $C_i$ agree on intersections, defining a $B$-permissible center $C \subset W$. These statements follow from the following :
  
  \smallskip
  
{\it Observation}:  Let $V_1$ and $V_2$ be amenable open neighborhoods of $x \in M$, such that nice basic objects $B''_1$ and $B''_2$, with inductive hypersurfaces $Z_1$ and $Z_2$, are defined on $V_1$ and $V_2$ respectively, as in \ref{V:an1.0}. Let $B^*_i:=(HB''_i)_{Z_i}$, $i=1,2$, both satisfying condition $\cE _s$, for all $s \ge 0$. Let 
$U \subseteq V_1 \cap V_2$ be an open neighborhood of $x$, $C_i$ the center that condition 
$\cE _0$ assigns to $B^*_i, ~i=1,2$. Assume $C_i$, $i=1,2$, are strongly permissible centers.  Then, 
$C_1 \cap U = C_2 \cap U$. 

To prove this observation note that, by the construction of  
\ref{V:an1.0}, ${B''_1}_{|U}={B''_2}_{|U}$ (say =$B''$). Now, both $Z_1 \cap U $ and $Z_2 \cap U$ are hypersurfaces on $U$ adapted to $B''$. Then by applying Lemma  
\ref{L:pego} (case $\lambda=0$) we get $C_1 \cap U = C_2 \cap U$, as needed.
\end{voi}
\begin{voi}
\label{V:ae4.0} 
 Now, inductively, assuming conditions $\cE _s$, satisfying properties (a), (b), (c) and (d) of \ref{V:ae1.0}, have been defined for $s < j$, $j > 0$, let's introduce condition $\cE _j$. By our inductive hypothesis (see \ref{V:ae1.0}, $(2)_{j}$) we get a permissible sequence of $A$-basic objects 
 $$ (1) \qquad B=B_0 \leftarrow \cdots \leftarrow B_j$$
inducing on fibers a sequence 
$$ (2) \qquad B^{(0)}=B_0^{(0)} \leftarrow \cdots \leftarrow B_j^{(0)}$$
(where (2) is the $j$-truncation of the algorithmic resolution of $B^{(0)}$).  Looking at the functions $\o_p$ corresponding to the sequence (2), we distinguish two cases: 
(I) max $(\o _{j})=0$, (II) max $(\o _{j})>0$. In case (I), $B_j$ is premonomial. We declare 
 condition $\cE _j$ valid if $B_j$ is monomial and we take the corresponding center $C_j$ to be the canonical center in this case (see \ref{V:det4.0}.)  
 
 In situation (II), looking at the 
 functions $t_p$ corresponding to the sequence (2), and letting $M:={\mathrm {Max}}(t_j)$ and $d=\dim(B)$, we distinguish the following two possibilities: ($\alpha$) $\dim M <d-1$, ($\beta$) $\dim M =d-1$.
 
Consider ($\alpha$) first. This is the situation where to define the $j$-th  center $C^{(0)}_{0}$ in the algorithmic resolution of $B^{(0)}$  we use induction on the dimension (via local associated homogenized nice basic objects, inductive hypersurfaces, inductive objects). 
Take the index $q$ such that  
  ${\mathrm {max}(t _{q})}= {\mathrm {max}(t _{q+1})} = \cdots = {\mathrm {max}  (t _{j})}$ 
but ${\mathrm {max}(t _{q-1})}> {\mathrm {max} (t _{q})}$. % and, moreover, $\dim (M_q) < d-1$, with $M_q:=\ma(t_q)$.   
 To declare condition $\cE j$ valid, first we require that conditions $\cE _s$ hold, for $s < j$. In particular, 
  condition $\cE _q$ is valid.
  Take an open cover $\{V_{iq}\}$ of ${C^{(0)}}_q$ (teh $q$-th algorithmic center in (1) of \ref{V:ae1.0}) such that on each $V_{iq}$ we have a nice $A$-basic object $B_{iq}'':={({B_q}_{|U_i})''}$ with adapted hypersurface $Z_{iq} \subset V_{iq}$ (i.e., $V_{iq}$ is amenable, see \ref{V:an1.0}). From  
 dim($M$)$<d-1$ we see that   $Z_{iq}$ is inductive (i.e., (A3) of \ref{V:is2.0} holds, see \ref{V:is7.0}) Next we take, for each $i$, the homogeneous $A$-basic object 
 $(HB_{iq}''):={\cH}({B_{iq}''})$  
  (again nice), admitting $Z_{iq}$ as an inductive hypersurface. Consider next, for each index $i$, the inductive object $B^* _{iq}:=(HB_{iq}'')_{Z_{iq}}$ and 
the algorithmic equiresolution sequence 
$$ (3) \quad B^*_{iq} \leftarrow B^*_{i q+1} \leftarrow \cdots \leftarrow B^*_{ij} \leftarrow \cdots $$
(obtained by the validity of condition $\cE _q$) with centers 
$C_{ip} \subset us(B^*_p)$, $q \leq p \leq j$. (It reaches at least level $j$ because of the equalities 
${\rm max}(t_q) = \cdots = {\rm max}(t_j)$). Finally, we  say that $B$ satisfies condition $\cE _j$ if 
 the sequence (3) is strongly permissible (\ref{V:is10.0}), for all $i$.
 Then there is an induced permissible sequence 
{$(HB_{iq}'') \leftarrow \cdots \leftarrow (HB_{ij}'') $}. Let $\widetilde{B_{iq}}$ denote the restriction of $B_q$ to $ V_{iq}$. By the W-equivalence of $\widetilde{B_{iq}}$ and $(HB_{iq}'')$ and 
 repeated application of \ref{P:proB"}, 
we obtain (by using the same centers) an induced $t$-permissible sequence
$$(4) \quad \widetilde{B_{iq} } \leftarrow \widetilde{B_{i,q+1} } \leftarrow \cdots \leftarrow \widetilde{B_{ij} }$$
From the validity of $\cE_p$ for $p < j$, by induction we may assume that by varying $i$ the centers $C_{ip}$ glue together, to produce globally defined centers $C_p$, $q \leq p < j$, so that the resulting sequence coincides with (1). More precisely, we obtain $C_q \subset us(B_q)$, next $C_{q+1} \subset us(B_{q+1})$, where $B_{q+1}=\cT (B_q,C_q)$, and so on. Finally, we have 
$B_j = \cT(B_{j-1},C_{j-1})$. 
In this way, 
$us(\widetilde{B_j}) $ any be identified to an open $V_{ij} \subset us(B_j)$, the center 
$C_{ij}$ of $B^*_{ij}$ to a closed subscheme of $V_{ij} \subset us(B_j)$. 
 Arguing as in \ref{V:ae3.0}(II), using 
 Lemma \ref{L:pego} (now in case $\lambda=j$) we see  that these agree on intersections. By the strong permissibility condition on the sequence (3), we obtain a $t$-permissible center $C_j$ of $B_j$. This $C_j$ is the center 
$\cE _j$ attaches to $B_j$.
 Again by Lemma \ref{L:pego}, this center is independent of the choice of the cover and the adapted surfaces selected, so $C$ is well defined. 
 
 Situation $(\beta)$ is treated exactly as case ``$\dim M=d-1$ '' in \ref{V:ae3.0}. 
 
 In both cases $\alpha$ and $\beta$ property (a) of \ref{V:ae1.0} follows from the fact that our construction over $A$ induces on the fiber the appropriate algorithmic center in the resolution (1) of \ref{V:ae1.0}, (b) of \ref{V:ae1.0} is immediate and (c), (d) are easily verified by using the remarks of \ref{V:ae2.1}.
 
So, we have completed the definition of conditions $\cE _j$, and the proof of Theorem \ref{T:main} (see \ref{V:prova}) is complete.  
 \end{voi}
 \begin{exa}
 \label{E:unosolo}
 A basic object over $A$ may have many equiresolutions, but only one of these can be the algorithmic equiresolution. Consider the following example. Here, $k$ denotes a characteristic zero field, $A=k[\epsilon]$, where ${\epsilon}^2=0$, $R=A[x,y]$ ($x,y$  indeterminates), $S=\Spec (A)$, $W=\Spec(R)$, $W \to S$ is induced by the natural homomorphism $A \to R$. Let $B=(W \to S, (y^2,x^3),2, \emptyset)$. Then 
 $B^{(0)}=(\Spec (k[x,y],(y^2,x^3),2, \emptyset)$ and Sing($B^{(0)}$)=$V(x,y)$ (the origin). If we blow-up the origin $0$ of $\Spec (k[x,y])$ the resulting transform $B^{(0)}_1$ will satisfy Sing($B^{(0)}_1$)=$\emptyset$, and clearly this is the algorithmic resolution of 
 $B^{(0)}$. Concerning $B$, if (for $\lambda \in k$) $C_{\lambda} \subset W$ is the subscheme defined by the ideal $I(C_{\l}):=(y,x-{\lambda}{\epsilon}) \subset R$, then $C_{\l}$ is a permissible center. Indeed, if we let 
 $x':=x-{\l}{\epsilon}$, then 
 $I=(y^2,{x'}^3+3{\l}{\epsilon}{x'}^2)$ and   $I(C_{\l})=(y,x')$. Now we immediately see that 
 $\nu (I,C)=\nu (I^{(0)},C^{(0)})=2$. Clearly $C_{\l}$ induces the origin $0$ on $B^{(0)}$, for all $\l$. If $B_{1{\l}}$ is the transform $B$ with center $C_{\l}$, then $\sg(B_{1{\l}})$ is nonsingular, for all $\l$. In fact, its restriction to the non trivial affine open of the blowing-up is 
 $(\Spec (A[x',y] \to S,(y^2,x'+3{\lambda}{\epsilon}),2,H)$ (with $H={\mathcal V}(x')$), and we see that 
 $\sg (B_{1{\l}})={\emptyset}$. Thus, for any $\l$ we obtain an equiresolution of $B$. \\
 For what values of $\l$ will this be algorithmic, that is will condition $\cE _{0}$ hold? We are in the case where (in the notation of \ref{V:ae3.0}) we have codim($M) >1$, so we have to use the inductive object. 
 Note that $B$ is a nice $A$-basic object, $Z$ defined by the ideal $(y)A[x,y]$ is an inductive  hypersurface. Let us study condition ${\cE}_{0}$ using $Z$. Here, $\Delta(I/S)=(y,x^2)$, $\cH(I/S)=(y^2,x^2y,x^3)$, 
 $\cC(\cH(I/S),Z)=(x^3)A[x]$. Hence, $B_Z=(\Spec (A[x] \to S,(x^3),2,{\emptyset})$.
By induction, to have condition $\cE _{0}$ satisfied, we need a permissible center for $B_Z$. To study permissiblity for $C_{\l}$ (defined by $I(C_{\l})=(x-{\l}{\epsilon}))$, note that  $(x-{\l}{\epsilon})^3=(x^3-3{\l}{\epsilon}x^2)$. From this fact we  see
 that for ${\l}\not=0$ we get 
 $2 = \nu (J,C_{\l})<\nu (J^{(0)},{C_{\l} ^{(0)})}=3$, so that the center $C_{\l}$ is not permissible, while for ${\l} =0$ we get 
  $\nu (J,C_{0})=\nu (J^{(0)},{C_{0} ^{(0)})} =3$, so that $C_0$ is permissible. 
So, this way the only  algorithmic center we get is  $C_0$, in agreement with the general theory
 \end{exa}
\begin{voi}
\label{V:ae5.0}
{\it Proof of Lemma \ref{L:pego}}. We proceed by induction on $t$. Let $t=0$, i.e., we have a single $A$-basic object $(HB)$. Let $C$ (resp. $C'$) be the algorithmic center of 
$(HB)_{Z}$ (resp. $(HB)_{Z'}$). Firstly, we claim that $C^{(0)}=C'^{(0)}$ (equality of fibers). Indeed, by (a) in \ref{V:ae1.0}, these are the algorithmic zeroth centers of the fibers 
$(HB^{(0)})_{Z^{(0)}}$ and $(HB^{(0)})_{Z'^{(0)}}$ respectively. But  both centers, regarded as subschemes of $W^{(0)}:=us(HB^{(0)})$
 must be the (unique) algorithmic zeroth center of $(HB^{(0)})$ (by the construction when the base is a field). Hence they agree: $C^{(0)}=C'^{(0)}$. \\
 Since the set of closed points is dense, to show the equality $C=C'$ it suffices to show that if $y$ is any closed point in $C^{(0)}={C'}^{(0)}$, then  
  the restrictions of $C$ and $C'$ to a neighborhood of $y$ coincide. \\ 
 By using \ref{T:etalez} we find etale morphisms $\pi, {\pi}'$ from a scheme $Y$ into $Z$ and $Z'$ respectively, such that 
 $\pi ^{*} ((HB)_Z) = {\pi'}^{*} ((HB)_{Z'})$. By (c) of \ref{V:ae1.0}, both 
${\pi}^{-1}(C)$ and 
  $\pi '^{-1}(C')$ must be the algorithmic zeroth center that condition $\cE_0$ assigns to 
  ${\pi} ^*((HB)_Z))$ and ${\pi'} ^*((HB)_{Z'})$ respectively. Since these $A$-basic objects are equal, ${\pi}^{-1}(C)={\pi '}^{-1}(C')$, whence $C=C'$ near $y$, as claimed.\\
  The inductive step  (transition from $t$ to $t+1$) is accomplished with an argument similar to that used for $t=0$, applying the inductive hypothesis and again  Theorem \ref{T:etalez}.
 \end{voi}
 \begin{voi}
 \label{V:ae5.1}
 {\it On generalized basic objects}. Initially, we intended to use strictly  the resolution process for basic objects (over a field) discussed in \cite{EV} or \cite{BEV}. In it there is a fundamental inductive step, where one replaces a basic object $B$ first by a nice object $B''$ (locally defined) and then this one by an object $B_Z$ with underlying scheme $Z$, a regular hypersurface defined on $us(B'')$. By induction we have a resolution function for $B_Z$, which allows us to define (locally) a resolution function for $B$. Alternately, we have an algorithmic permissible center for $B_Z$ which produces (locally) an algorithmic center for $B$. But this construction is local and moreover, given $B''$, there are many adapted hypersurfaces $Z$, so there is a problem of patching, if we try to get a globally defined resolution function (or algorithmic permissible resolution center). (Similar considerations apply to a basic  object $B_j$ which appears in a suitable permissible sequence of basic objects, see \ref{V:is8.0}). 
  In these references to solve this problem the authors 
  use  {\it generalized basic objects}. In a suitable sense these are, locally, basic objects, but it is possible to define for them global permissible centers, globally defined $t$-functions (see \ref{V:ag7}), and so on. Once this is verified, the mentioned patching problem is easily solved. Key results to implement this approach are: (a) Hironaka's trick (see \cite{BEV}, section 21) (b) the fact that formula (1) of  \ref{V:is8.0} holds or (essentially equivalently), (c) that a center $C$ is permissible for $B$ if and only if it is permissible for $B_Z$. As we saw, the statement analogous to (c) working over $A \in \cA$ rather than a field is not true (see \ref{E:accidenti} and \ref{E:accidento}). Hence it does not seem possible to adapt this approach to the situation where we work over an artinian ring, necessary for  Theorem \ref{T:main}. This difficulty was overcome 
   by abandoning generalized basic objects and using some techniques from \cite {W} instead. This leads to the VW-resolution algorithm, that we have used. However, many aspects of the theory of generalized basic objects, including Hironaka's trick,  can be adapted to the case of objects over $A \in \cA$. Since we do not make use of these results, we do not present the details here.
 \end{voi}

  \section{Families of ideals and varieties}
\label{S:FV} 
  
  It is known that, working over a field, our algorithm of resolution for basic objects induces algorithms for principalization of ideals and resolution of embedded varieties (see \cite{BEV}).  In this section we present a brief review of these facts and a description of an analogue  when we work over an Artin ring. Throughout, $A$ will denote a ring in the class $\cA$ of \ref{V:bn1.0}, $S:=\Spec (A)$.
 \begin{voi}
 \label{V:ideals} {\it Families of ideals}. Given an Artin ring $A \in {\mathcal A}$, an {\it  idealistic triple} (or {\it id-triple}) over $A$ is a 3-tuple  
 ${\mathcal T}=(p: W \to S,I,E)$,  
   $S= \Spec (A)$, such that $(W \to S,E)$ is an $S$-pair (\ref{V:bn5.0}) and $I$ is a never-zero $W$-ideal (\ref{V:ag1}). 
   
   There is natural notion of  (closed) fiber, by reducing modulo $r(A)$. This fiber is an id-triple over $k=A/r(A)$,  ${\mathcal T}^{(0)}=(W^{(0)} \to {\Spec k},I^{(0)},E^{(0)})$. We might call an id-triple over $A \in {\mathcal A}$ a {\it family of triples (over fields) parametrized by $S$} or an {\it infinitesimal deformation} of the triple ${\mathcal T}^{(0)}$. If $A=k$ is a field, often we shall write 
   $(W,I,E):= (W \to {\Spec k},I,E)$.
   
   Given and id-triple  ${\mathcal T}=(W\to S,I,E)$ and a subscheme $C$ of $W$ which is a permissible center for the underlying $S$-pair  $(W\to S,E)$ (\ref{V:bn5.0}), we may define the {\it transform} of $\mathcal T$ with center $C$. This is the id-triple 
    ${\mathcal T}_1=(W_1 \to S,I'_1,E_1)$, where $(W_1 \to S,E_1)$ is the transform of the $S$-pair $(W \to S, E)$ with center $C$ (\ref{V:bn9.1}) 
 and $I'_1=I{\cO}_{W_1}$ (the total transform of $I$ to $W_1$, \ref{V:bo2.0}). 
 
 An $A$-resolution, or {\it equiprincipalization} of $\mathcal T$ is a sequence 
 ${\mathcal T}={\mathcal T}_0 \leftarrow {\mathcal T}_1{\leftarrow  \cdots \leftarrow \mathcal T}_r$ of id-triples, with  
${\mathcal T}_i=(W_i \to S,I_i,E_i)$, ${\mathcal T}_{i+1}$ the transform of 
${\mathcal T}_i$ with a permissible center, such that  
$I'_r=I(H_1)^{c_1} \ldots I(H_s)^{c_s}$ (where $E_r=(H_1, \ldots , E_s$)), for suitable (locally constant) integral exponents $c_i \ge 0$, $i=i, \ldots, s$. By taking fibers, such an $A$-principalization induces a principalization sequence for ${\mathcal T}^{(0)}$ (see \cite{BEV}, Theorem 2.5). 
  
 It is known that when the base is a field $k$ (of characteristic zero) our algorithm of resolution for basic objects induces an algorithm for {\it principalization} of id-triples. Namely, given the id-triple 
 $\mathcal T=(W,I,E)$ over a field $k$, one considers the basic object 
 $B_0= (W,I,1,E)$ and applies the algorithm to $B_0$, getting a resolution 
 $B=B_0 \leftarrow \cdots \leftarrow B_r$, $B_i=(W_i,I_i,1,E_i)$. Just by dropping the entry $b=1$ in each basic object $B_i$, we get the desired principalization. See \cite{BEV}, Parts I and II, for details. Henceforth this principalization process will be referred to as {\it the algorithmic principalization  of} $\mathcal T=(W,I,E)$.
 
 Now, relative to the VW-algorithm, we introduce a 
 notion of {\it algorithmic equiresolution}, or {\it equiprincipalization}, for id-triples over $A \in {\cA}$. Namely, given such an id-triple 
 $T=(W\to S,I,E)$, with fiber  $\mathcal T^{(0)}=(W^{(0)},I^{(0)},E^{(0)})$, we 
  say that $\mathcal T$ is {\it algorithmically equiprincipalizable} if conditions ${\mathcal E}_0, \ldots, {\cE}_{r-1}$ are valid for the basic object $B=(W \to S, I,1,E)$. Alternatively, we could demand that conditions   ${\mathcal E}_j$ be valid for all possible $j$, since they are vacuously hold for $j \ge r$. So, if $\mathcal T$ is  equiprincipalizable, we have an algorithmic equiresolution  
 $B=B_0 \leftarrow \cdots \leftarrow B_r$, $B_i=(W_i,I_i,1,E_i)$, of $B$.
 As before (when we worked over a field), by dropping throughout the entry $b=1$, we obtain an $A$-principalization of $\mathcal T$ inducing by taking fibers the algorithmic principalization of ${\cT}^{(0)}$. This will be called the {\it algorithmic equiprincipalization of $\mathcal T$}.
 
 It is easy to state and prove a theorem analogous to \ref{T:main} for $A$-triples, $A \in \cA$. We leave this task to the reader.
 
 We could have defined the notion of {\it family of ideals parametrized by $\Spec (A)$, $A \in \cA$} as a pair $(p:W \to S,I)$, with $p$ and $I$ as above, and essentially repeat what was done above. But since in an $A$-resolution process the exceptional divisors that appear must be considered, it seems more reasonable to work from the outset with id-triples instead. 
 \end{voi}

\begin{voi}
\label{V:var1.0}
Working over a characteristic zero field $k$, 
 a pair $\mathcal X = (X, W)$ where $W$ is a scheme, smooth over $k$, and $X$ is a reduced, equidimensional subscheme of $W$ will be called  an {\it embedded variety }. 
 
 A {\it resolution} of $\cX$ is a proper, birational morphism $f:W' \to W$, with $W'$ smooth, such that: 
   (i) the exceptional locus of $f$ is the  union of regular hypersurfaces 
 $H_1, \ldots, H_n$ with normal crossings, (ii) the strict transform $X'$ of $X$ to $W'$ is regular, and has normal crossings with $H_1, \ldots, H_n$, (iii) $f$ induces an isomorphism $X'- f^{-1}(\Sigma) \to X- {\Sigma}$, where $\Sigma$
 is the singular locus of $X$.
 
 As explained in sections (2.4) and (5.8) of \cite{BEV} , our algorithm for resolution of basic objects induces an algorithm for resolution of embedded varieties. Indeed, consider the basic object $B=(W,I(X),1,\emptyset)$ and its corresponding algorithmic resolution:
 $$B=B_0 \leftarrow {B}_1 \leftarrow \cdots \leftarrow {B}_r           $$
 obtained via resolution functions $g_0, \ldots, g_{r-1}$, taking values in a totally ordered set ${\Lambda}^{(d)}$, $d=\dim(W)$ (see \ref{V:ag7}). We write 
 $B_i = (W_i, I_i, 1, E_i)$, for all $i$.  A property of our algorithm says that  $g_0$ is constant, say = $a \in {\Lambda}^{(d)}$ on $W \setminus \sg (X)$ and there is a unique index $n$ (depending on $B$, hence on $\cX$) such that max($g_n$)=$a$.  It turns out that the strict transform $X_n$ of $X$ to $W_n$ is a union of components of ${\mathrm {Max}}(g_n)$ (i.e., the $n$-th center in the algorithmic resolution process), hence it is regular, having normal crossings with $E_n$.  
 We shall denote the index $n$ above by $\eta({\cX})$. 
 \end{voi}

\begin{voi}
\label{V:var2} {\it Relative $A$-varieties and their resolutions}.  If $A \in \cA$ (\ref{V:bn1.0}), a scheme $X$ together with a  flat morphism of finite type 
$f:X \to S=\Spec A$ will be called a {\it relative $A$-scheme}. The morphism won't be specified when it is clear from the context. If the (only) fiber 
$X^{(0)}$ is reduced and equidimensional, we shall talk about a {\it relative  $A$-variety}. In that case, if $U$ is the open set of points where 
$X^{(0)}$ is smooth over $k=A/r(A)$, the induced morphism $X_{|U} \to S$  is smooth over is smooth (because it is flat, with smooth fiber).  Note that $X_{|U}$ is a Cohen-Macaulay scheme (use 21.C, page 154, in \cite{Mm}). Let us write $\cS (X/S) := X_{|U}$.

A {\it resolution} of an  a relative   $A$-variety  $f:X \to S$ is a proper morphism $\phi : X' \to X$ such that: (i) $f \phi : X' \to S$ is smooth, (ii) $\phi$ induces a resolution of fibers morphism ${\phi}^{(0)}:{{X'}^{(0)} \to X^{(0)}}$ (i.e., ${\phi}^{(0)}$ is proper, birational and an isomorphism off the singular locus of $X^{(0)}$). 
\end{voi}
 
 \begin{voi}
 \label{V:var2.0}
 {\it Relative embedded $A$-schemes and varieties}. 
  As usually, $A$ is an Artin ring in $\cA$, $S = \Spec (A)$. 
  
  (a) A {\it relative embedded $A$-scheme} (or an embedded scheme, flat over $S$) is a pair $\cX = {(X, p:W \to S)}$,  with $p$ smooth, $X$ a closed subscheme of $W$, such that 
    the  morphism $q: X \to S$ induced by $p$ is flat. 
  
  If in addition $X^{(0)}$ is reduced and equidimensional, we talk about a {\it relative embedded $A$-variety}. In this case 
   $(X^{(0)}, W^{(0)})$ is an embedded variety over $k=A/r(A)$.

 (b) We say that an the relative embedded $A$-variety $\cX$ is {\it equisolvable} if there is a proper morphism $\psi : W' \to W$ such that: 
 (i) the composition $p \psi : W' \to S$ is smooth, (ii) $\psi $ induces a proper birational morphism of fibers ${W'}^{(0)} \to W^{(0)}$, (iii) If ${\tilde X} = {\cS (X/A)}$ (the largest open subscheme of $X$ smooth over $S$, see \ref{V:var2}) and $X'$ is the scheme-theoretic closure of ${\psi}^{-1}({\tilde X})$ in $W'$, then the induced morphism 
 $X' \to X$ is a resolution of the relative $A$-variety $q:X \to S$. 
 
 We have used the term {\it relative}  
 by analogy with the well established terminology concerning divisors flat over a base scheme, see \cite{Mu}, page 72. 
 \end{voi}

\begin{thm}
 \label{T:emva} Let $\cX = {(X, p:W \to S)}$ be a relative $A$-embedded variety, 
  $\cX ^{(0)} = (X^{(0)}, W^{(0)})$ its fiber. Assume  $B:=(p:W \to S,I(X),1,\emptyset)$, 
   satisfies conditions ${\cE}_0, \ldots,{\cE}_{q}$, where $q=\eta({\cX}^{(0)})$ (\ref{V:ae6.0}). Then, 
   $\cX$ is equisolvable in the sense of \ref{V:var2.0} (b). 
  \end{thm}
  \begin{proof}
  The validity of these conditions implies the existence of a permissible sequence 
 $$B=B_0 \leftarrow \cdots \leftarrow B_q \leftarrow B_{q+1}$$
 $B_i=(W_i,I(X)_i,1,E_j)$,  with centers $C_i \subset W_i$, $i=0, \ldots q$, inducing on special fibers the $(q+1)$-truncation of the algorithmic resolution sequence of the fiber $B^{(0)}=B_0^{(0)}$ (with centers $C_i^{(0)}, ~ i=0, \ldots, q$). By \ref{V:var1.0} the strict transform $X^{(0)}_q$ of $X$ to $us(B_s^{(0)})$ is the (disjoint) union of components $Y_1^{(0)}, \ldots, Y_s^{(0)}$ of the center $C_q^{(0)}$. Let $Y_i$, $i=0,\ldots,s$ be the component of the center $C_q$ inducing $Y_i^{(0)}$ on the special fiber (i.e., modulo $r(A)$) and 
 $\psi:W_q \to W_{0}$  the naturally induced morphism. Using the fact that ${\cS (X/A)}$ is a Cohen-Macaulay scheme and remark (b) in the proof of Proposition \ref{P:ekiv},  one sees that the scheme -theoretic closure of $ {\psi}^{-1}(\cS(X/A))$
 in $W_q$ is equal to $Y_1 \cup \ldots \cup Y_s:=X_q$. Since $C_q$ is smooth over $S$ (via the restriction of the natural projection $W_q \to S$)  $X_q$ is smooth over $S$ and      one readily checks  that the  morphism $\psi:W_q \to W$ defines an equiresolution of 
 $\cX$ in the sense of \ref{V:var2.0} (b). 
 \end{proof}
\begin{voi}
 \label{V:var2.1}
In the notation of Theorem \ref{T:emva}, when conditions   ${\cE}_0, \ldots,{\cE}_{q}$,  $q=\eta({\cX}^{(0)})$, hold for $B=(p:W \to S,I(X),1,\emptyset)$ we say that the relative embedded $A$-variety $\cX$ is {\it algorithmically equisolvable}, and call the equiresolution obtained in the proof of \ref{T:emva} its {\it algorithmic equiresolution}. 
 \end{voi}
 \begin{voi}
 \label{V:var3} We use the notation of \ref{T:emva} and its proof. If $\cX$ is is algorithmically equisolvable the algorithmic equiresolution of \ref{T:emva} has some additional properties, namely: (a) the  morphism $\Psi:W_q \to W$ is a composition of blowing-ups with centers smooth over $S$, (b) the exceptional divisor $D$ of $\psi$ is a union of hypersurfaces with normal crossings (\ref{V:bn4.0}), (c) $X_q$ has normal crossings with $D$.
 \end{voi}

\begin{voi}
\label{V:ae6.0}
To finish this chapter we indicate how our notions of equiresolution (for basic objects, ideals or embedded schemes), given when we work over $S=\Spec(A)$, $A \in \cA$, naturally induce similar notions when working with families parametrized by an arbitrary noetherian scheme  $T$. Consider, e.g., a {\it basic object over a scheme $T$}, that is a four-tuple $B={(p:W \to T,I,b,E)}$, with $p$ a smooth morphism, $I$ a never-zero ideal of ${\cO}_W$, $b$ a positive integer, $E$ a finite sequence of distinct hypersurfaces of $W$ with normal crossings. Let $t \in T$,  $R_{t}:={\cO}_{T,t}$ and $P_t$ the maximal ideal of $R$. Hence, for any non-negative integer $m$, $R_{t,m}:=R/{P_{t}}^{m+1} \in {\cA}$. The family $B$ induces an 
$R_{t,m}$-basic object $B_{t,m}=({p_{t,m}:W_{t,m} \to S_{t,m}},I_{t,m}, b,E_{t,m})$ ($S_{t,m}=\Spec(R_{t,m})$) where the morphism $p_{t,m}$ is obtained by base change
 (via the natural morphism $S_{t,m} \to T$), the hypersurfaces in $E_{t,m}$ come, by pull-back, from those in $E$, and $I_{t,m}=I {\cO}_{W_{t,m}}$. We say that $B$ is {\it algorithmically  equisolvable at $t$} if, for every integer $m \ge 0$, the induced family $B_{t,m}$ is an equisolvable $R_{t,m}$-basic object, in the sense of 
 \ref{D:aleq}.
  Finally we say that $B$ is  {\it algorithmically equisolvable} if it is algorithmically  equisolvable at $t \in T$, for every $t \in T$. 
 
 Similarly, essentially by substituting in our previous work (in \ref{V:ideals} and \ref{V:var2.0}) the base $S=\Spec (A)$ by an arbitrary noetherian scheme $T$, we  introduce the notions of family of id-triples and family of embedded schemes, parametrized by $T$. With the notation above, we naturally obtain induced families over $R_{t,m}$, at each point $t \in T$ and we say that a family of id-triples  is equiprincipalizable  at $t$ is the induced family over $R_{t,m}$ is so, for all $m \ge 0$.  Finally we say that it is equiprincipalizable  if it is equiprincipalizable  at every $t \in T$. 
 
 Following \cite{EV} or \cite{BEV} we say that a family of embedded schemes $(X,W \to S)$ is {\it algorithmically  equisolvable} if its associated family of id-triples $(W \to S, I(X), b, \emptyset)$ is algorithmically  equiprincipalizable. 
 
 In case the parameter space is a smooth algebraic $k$-scheme $T$ ($k$ a characteristic zero field) these notions are closely related  to those studied in 
  \cite{EN} or \cite{BEV}, section 10. This will be discussed in a subsequent article. We hope that that  the present theory will also have  applications similar to those considered in \cite{EN}, section 4 (Hilbert schemes).
\end{voi}

\section{Appendix 1: Review of useful results}
\label{S:A}
 In this appendix we collect a number of basic algebraic and geometric results that are used in the paper. Probably most of them are well known, but we prove those for which we could not find appropriate references in the literature.

\begin{voi}

\label{V:ap1.0}

(a) We shall use the notation and terminology of \ref{V:bn1.0}. Thus, 
${\mathcal A}$  denotes the collection of artinian local rings $(A,\cM)$ such that the residue field $k=A/{\cM}$ has characteristic zero. Any ring in ${\mathcal A}$ is necessarily a complete $k$-algebra, $k=A/r(A)$.

\smallskip

(b) We shall be primarily concerned with the following situation: $(A,{\cM})$ is an artinian ring in $ {\mathcal A}$, $(R,{\mathcal N})$  a local noetherian $A$-algebra, essentially of finite type (via a local homomorphism 
$A \to R$), $S=\Spec (A)$. Let  $Z=\Spec (R)$ and  
$\pi : Z \to S$ be the induced morphism.  The closed fiber is isomorphic to $\Spec (R^{(0)})$, where 
$R^{(0)}:= R/{\cM}R  $ 

\smallskip

(c) Often it will be the case that $R$ is a ${\mathcal N}$-smooth
 $A$-algebra, or 
${\mathcal N}$-smooth over $A$. This is usually defined in terms of a homomorphism lifting property (see \cite{M}, page 213, or \cite{And}, Def. 14, p. 222.) Equivalently, this means that that the morphism $\pi$ above is flat, with its closed fiber geometrically regular, see \cite{And}, Thm. 18, p. 224. Under our assumption that the field $A/{\cM}$ has zero characteric, this just means: ``the closed fiber is regular'', that is `` $R^{(0)}$ is regular''. 

\end{voi}

\begin{lem}
\label{L:regse}
With notation as in \ref{V:ap1.0} (b), assume moreover that $R^{(0)}$ is a regular ring, of dimension $n$, 
$a^{(0)}_1,\ldots, a^{(0)}_n$ is  a regular system of parameters of $R^{(0)}$, and for all $i$ let $a_i \in {\mathcal N} = r (R)$ be such that $a_i$ induces $a^{(0)}_i$ via the canonical homomorphism $R \to R^{(0)}$. Then: 
$(i)$  $a_1,\ldots, a_n$ is a regular sequence in $R$, $(ii)$ $R$ is a Cohen-Macaulay ring, 
$(iii)$ If $R$ is a ${\mathcal N}$-smooth $A$-algebra then, for any indices  
$1 \leq i_1 < \ldots < i_r \leq n$,  
$R/(a_{i_1}, \ldots, a_{i_r})R$ is a ${\mathcal N}$-smooth $A$-algebra.

\end{lem}
\begin{proof} We show that 
$a_1,\ldots, a_n$ is a regular sequence by induction on $n$, the case $n=0$ (i.e., $A=R$) being trivial. It suffices to check that $a:=a_1$ is a regular element of $R$. In fact, if this is the case, then our hypotheses applies to $R/(a_1)$ and the images of 
$a_2,\ldots, a_n$ in this ring, and we use induction.

We shall see that $a$ is regular in 
$R$ by induction on    
$\dim (A)$ (dimension as a vector space over $k=A/{\cM}$). The case where $\dim (A) = 0$ is trivial. For the inductive step, we shall use the well-known fact that the maximal ideal ${\cM}$ contains an element $\epsilon \not= 0$ such that $\epsilon \, {\cM} = 0$. Let $A' := A/(\epsilon) A$, 
${\phi} = R \to R' := R/(\epsilon) R$ be the canonical homomorphism and $b$ the image of $a_1$ in $R'$. By induction assumption $b$ is a regular element. Let $\alpha$ be an element of $R$ such that $a \, {\alpha} = 0$, ${\alpha}'$ its image in $R'$. Then, 
$b {\alpha}' = 0$, which (by the regularity of $b$) implies that ${\alpha}'=0$. Hence, ${\alpha} \in K:={\rm Ker} \, (\phi) = (\epsilon)\, R$. Note that, by the property of the element $\epsilon$, 
$K=(\epsilon)R$ is naturally a $k$-module, as such isomorphic to $R^{(0)}$. By this isomorphism, $\alpha \in K$ corresponds to $c \in R^{(0)}$, and 
$a {\alpha}$ to $a^{(0)}_1 \, c$. Since $a^{(0)}_1 \not= 0$ in the integral domain $R^{(0)}$ (because $a^{(0)}_1$ is part of a regular system of parameters), $c=0$, hence $\alpha = 0$, as desired. This proves $(i)$

($ii$) The fact that $R$ is Cohen Macaulay follows from the equalities ${\rm dim} (R) = {\rm dim} (R^{(0)})=n$ and the existence of the regular sequence $a_1, \dots, a_n$, that we just verified. 

($iii$) Concerning the smoothness, to begin with $R$ is $A$-flat by the "lifting relations criterion" (\cite{Ar} p. 11, the proof presented there, for polynomial rings, works more generally and yields the result we need here.) In fact, we are dealing with regular sequences, whose relations are trivial. Next, the only closed fiber is a regular variety, since 
$a^{(0)}_1,\ldots, a^{(0)}_n$ 
was a regular system of parameters in the regular local ring $R^{(0)}$,  
which gives us smoothness.    
\end{proof}

\begin{pro}
 \label{P:divi}
 Let $f:X \to Y$ be a morphism of noetherian schemes, $D \subset X$ an effective Cartier divisor (which we identify to a closed subscheme of $X$ locally defined by a non-zero divisor), $U=X - D$, $f':D \to Y$ and $f_U: U \to Y$ the morphisms induced by $f$ by restriction. Assume both $f'$ and $f_U$ are smooth. Then, $f$ is smooth (i.e., flat with geometrically regular fibers.).
\end{pro}

\begin{proof}. In \cite{N}
 it is shown that the flatness of $f'$ and $f_U$ implies that $f$ is flat. Let us check that all the geometric fibers are regular. So, consider such a fiber $X_y$, $y$ a geometric point of $Y$. Here, with obvious notation, $D \cap X_y \subset X_y$ is a Cartier divisor, let $U_y = X_y -D$. Note that $D \cap X_y $ (resp. $U_y$) can be identified to the fiber of $f'$ (resp. $f_U$) at $y$, hence is regular. Then the regularity of $X_y$ is a consequence of the following lemma, finishing the proof.
\end{proof}

\begin{lem}
\label{L:inhreg} Let $(R, \mathcal N)$ be a noetherian local ring, $a \in \mathcal N$ a regular  element of $R$ (i.e., a non-zero divisor), suppose $R/(a)$ is a regular ring. Then, $R$ is regular.
\end{lem}

\begin{proof} Let dim($R$) $= d$, $R':=R/\mathcal N$, 
${\mathcal N}'$ the maximal ideal of $R'$. By \cite{AM}, Cor. 12.18, dim($R'$)$=d-1$, The natural homomorphism of vectors spaces (over the common residue field $k$) 
${\mathcal N}/{\mathcal N}^2 \to {\mathcal N}'/{\mathcal N}'^2$ is onto and its kernel is generated by $a + {\mathcal N}^2$, hence 
 we have ${dim}({\mathcal N}/{\mathcal N}^2) \leq {dim}({\mathcal N}'/{\mathcal N}'^2) + 1$, where $dim$ means dimension as  a vector space. But by the assumed regularity of $R'$, the right hand side is equal to 
$ {\rm dim}(R') +1 = {\rm dim}(R) -1 + 1 = {\rm dim}(R) = d$. Hence ${\rm dim}(R) = {dim}({\mathcal N}/{\mathcal N}^2)$ and $R$ is regular.
\end{proof}
\begin{pro}
\label{P:compsm}
The  notation is as in 1.1(a), we asume that $R$ is ${\mathcal N}$-smooth over $A$. Let 
$a_1, \ldots, a_n$ be elements of $R$ inducing a regular system of parameters 
$a'_1, \ldots, a'_n$ of the regular local ring $R^{(0)}$ and $I=(a_1, \ldots, a_r)R$, $r \leq n$ . 
 If 
$p: Z' \to Z$ is the blowing-up of $Z$ with center $I$ and $\pi ':Z' \to S$ the composition of $p$ and $\pi$,   
then $\pi'$ is a smooth morphism.
\end{pro}

\begin{proof}.  
Let $R_i:= R[a_1/a_i,\cdots, a_n/a_i]$.
 By the usual local description of the blowing-up, it suffices to show that if we regard $R_i$ as an $A$-algebra via the composition homomorphism $A \to R \to R_i$, then the resulting morphism $U_i:= \Spec (R_i) \to S$ is smooth, for all 
$i=1, \ldots, n$. To simplify the notation take $i=1$. Let $E \subset Z'$ be the exceptional divisor of our blowing-up, $E_i := E \cap U_i$ (note that that $U_i$ is naturally identified to an open of $Z'$). By Proposition \ref{P:divi}, it suffices to show that the induced morphisms from $U'_i:=U_i - E_i$ and $E_i$ to $S$ are both smooth. Since $U'_i$ is isomorphic to $Z-V(I)$, and $Z$ is smooth over $S$,  $U'_i \to S$ is smooth. Concerning, $E_i \to S$, since 
$a_1, \dots, a_r$ is a regular sequence in $R$, then we have an isomorphism of graded rings 
${\rm gr}_I(R) = (R/I) [T_1, \ldots, T_r]$, where the $T_i$'s are indeterminates. 
The Proj of this is $E$ and, ``dehomogenizing'', $E_i$ is isomorphic to 
$\sp ( (R/I) [t_2, \ldots, t_r])$ (see \cite{K}, p. 152), where  $t_i = T_i/T_1$, $i = 2, \ldots ,r$. Since $(R/I) [t_2, \ldots, t_r]$ is a polynomial ring over $R/I$ and $R/I$ is smooth over $A$ (by Lemma \ref{L:regse}, iii), the smoothness of the projection $E_i \to S$ follows.
\end{proof}

\smallskip

 On the following proposition we use the notation of \ref{V:bn1.0}-\ref{V:bn3.0}.
\begin{pro}
\label{P:preries}
Let $W \to S=\Spec (A)$ be a smooth morphism, $w \in W$, $a_1, \ldots, a_n$ a partial system of $A$-regular parameters of $R=\cO _{W,w}$, $I=(a_1, \ldots, a_n)R$. 
%A'=R/I$.    $\cN = r(R)$. 
Let $\hat{R}$ be the $I$-completion of $R$. Then, $\hat{R}$ is isomorphic to a power series ring $A'[[x_1, \ldots, x_n]]$, so that the isomorphism sends $a_i$ into $x_i$, $i=1, \ldots, n$ and $A'$ is isomorphic to $R/I$.  
\end{pro}

 \begin{proof}
 (i) First we shall prove this result with the added assumption that $R$ contains a subring $A'$ isomorphic to $R/I$ via the canonical quotient homomorphism. In this case $R$ is not necessarily local, and $a_1, \ldots, a_n$ may be assumed to be just a regular sequence.
 
 Consider the ring homomorphism  $\psi : A'[x_1,\ldots,x_n] \to R$ (where the elements $x_1, \ldots, x_n$ are algebraically independent over $A'$) such that $\psi(b)=b$ if $b \in A'$ and 
 $\psi(x_i)=a_ i$ for all $i$. We claim that the induced homomorphism of completions, with respect to the ideals $(x_1, \ldots, x_n)$ and $(a_1, \ldots, a_n)$ respectively, is the desired isomorphism. It suffices to show that for all positive integer $j$ the homomorphism  
 $\psi _n :  A'[x_1,\ldots,x_n]/(x_1, \ldots, x_n)^{j} \to R/I^j A'$ induced by $\psi$ is bijective. By taking quotients this will be true if the induced homomorphism 
 $(x_1, \ldots, x_n)^{j} /(x_1, \ldots, x_n)^{j+1}  \to I^{j}/I^{j+1}$ is bijective, for all positive integer $j$. But the later statement is true by the isomorphism of graded rings 
 $R/I[x_1, \ldots, x_n] \approx {\mathrm gr}_I(R)$ (proved, e.g., in \cite{K}, page 152).
 
 (ii) Let us consider now the general case. First let us check that the completion $\hat{R}$ contains a subring $A'$ mapping isomorphically onto $\hat{R}/\hat{I} \approx R/I$ via the quotient map (where $\hat{I}:=I \hat{R}$). To see this, consider the commutative diagram 
 $$\begin{CD}
 A @>>> {R/I^2}\\
 @VVV      @VVV\\
 A' @>{\alpha _1}>>{R/I}
   \end{CD} $$
(recall that $R$ is an $A$-algebra). Note that by \ref{L:regse}, $A' \approx R/I$ is smooth over $A$, hence the is a homomorphism ${\alpha}_2:A' \to R/I^2$ making the resulting augmented diagram commutative. Thus, we obtain a commutative diagram:
$$\begin{CD}
 A @>>> {R/I^3}\\
 @VVV      @VVV\\
 A' @>{\alpha _2}>>{R/{I^2}}
   \end{CD} $$
   As before, the smoothness of $A'$ over $A$ gives us a homomorphism 
   ${\alpha}_3:A' \to R/I^3$ making the resulting augmented diagram commutative. Reiterating, we get a system of compatible homomorphisms 
  ${\alpha}_j:A' \to R/I^j$ for all positive integer $j$, which yields a homomorphism 
  $A' \to \hat{R}$ inducing an isomorphism $A' \approx  \hat{R}/\hat{I}$. Thus the image of $A'$ (still denoted by $A'$) is the desired subring of $\hat{R}$. 
  
  (iii) Now apply the result of (i) to $\hat{R}$. If $R^*$   denotes its completion with respect to the ideal $(a_1, \ldots , a_n)$, then $R^*$ is isomorphic to a power series ring $A'[[x_1, \ldots, x_n]]$, so that the isomorphism sends $a_i$ into $x_i$. But since $\hat{R}$ is 
$(a_1, \ldots , a_n)$-complete, the natural homomorphism $\hat{R} \to R^*$ is an isomorphism. Thus, $\hat{R}$ has the desired property, and the proposition is proved.  
 \end{proof}
 
 Next we present a result on blowing-ups that we  use several times.
 
 \begin{pro}
 \label{P:bloui}
 Let $f:X \to T$ be a morphism of schemes, $C \subset X$ a close subscheme, flat over $T$, $J=I(C)$, such that $J_x \subset {\cO}_{X,x}$ is generated by a regular sequence , for all $x \in C$. Let $T' \to T$ be a morphism, 
 $X'=X \times _T T'$, $p:X' \to X$  the natural projection, $C'=C \times _T T'=p^{-1}(C)$, $X_1 \to X$ and $X_1' \to X_1$ the blowing-ups with centers $C$ and $C'$ respectively. Then, there is a natural isomorphism $X_1' = X_1 \times _X X'$ 
 \end{pro}
 
 \begin{voi}
 In particular, we may take as $T'$ a closed point $t$ of $T$ and the natural morphism 
 $\Spec (k(t)) \to T$. By the proposition, we may identify the blowing -up of $ f^{-1}(t)$ with center $f^{-1}(t) \cap C$ with ${f_1}^{-1}(t)$, where $f_1:X_1 \to T$ is obtained by composition.
 
 By using the definition (or construction ) of the blowing-up given in \cite{H}, page 163, Proposition \ref{P:bloui} is an easy consequence of the following algebraic lemma. In it, if $B$ is a ring and $I$ and ideal of $B$, we write 
 ${\cP}_{B}(I):=B \oplus I \oplus I^2 \oplus \cdots $
  (a graded $B$-algebra). 
 \begin{lem}
 \label{L:bla}
 Let $R \to B$ be a homomorphism of  rings, $I \subset B$ an ideal, generated by a regular sequence, such that $B/I$ is $R$-flat. Let $R \to R'$ a ring homomorphism, $B':=R' \otimes _R B$ (naturally a $B$-algebra) and $I':=IB'$. Then,
  ${\cP}_B \otimes R'={\cP}_{B'}(I')$. 
 \end{lem}
 \begin{proof} The contention follows if we prove that $I^n \otimes _R R' = (I')^n = (IB)^n$, for all $n \ge 0$. But this follows if $B/I^n$ is $R$-flat, for all $n$. Indeed, from the exact sequence of $R$-modules
 $$0 \to I^n \to B \to B/I^n \to 0$$
 by tensoring with  $R'$ over $R$ we get, from the flatness of the $R$-module $B/I^n$, that the sequence 
 $$ 0 \to I^n \otimes _R R' {\stackrel{\phi _n} \longrightarrow} B' \longrightarrow (B/I^n) \otimes _r R'  \to 0 $$
 is exact. So, ${\phi _n}$ is injective and $I^n \otimes _R R' = {\rm Im}{(\phi}_n)=(IB')^n$, as needed.
 
 To see that $B/I^n$ is $R$-flat (for all $n \ge 0$), from the fact that $I$ is generated by a regular sequence, we obtain canonical isomorphisms 
 $$gr _I(B) = B/I \oplus I/I^2 \oplus I^2/I^3 \oplus \cdots  = (B/I)[T_1, \ldots, T_r]$$ 
 (a polynomial ring, see \cite{K}, Proposition 5.10). Thus, $I^n/I^{n+1}$ is a finite free $B/I$-module, for each $n$. Since $B/I$ is $R$-flat, it follows that  $I^n/I^{n+1}$ is $R$-flat, for all $n \ge 0$. By using the exact sequences 
 $$0 \to  I^n/I^{n+1} \to  B/I^{n+1} \to B/I^{n}\to 0$$
and induction, we obtain that  $B/I^{n}$ is $R$-flat for all $n \ge 0$, as desired. 
 \end{proof}
 
 \end{voi}

\providecommand{\bysame}{\leavevmode\hbox to3em{\hrulefill}\thinspace}


\begin{thebibliography}{10}

\bibitem{And} M. Andr\'e \emph{Homologie des alg\`ebres commutatives}, Springer-Verlag, Berlin (1974)

\bibitem{A} M. Artin \emph{Algebraic  Approximations of Structures over Complete Local Rings}, Publ. Math\'ematiques I.H.E.S. 36 , 23-58 (1969)

\bibitem{Ar} M. Artin \emph{Deformations of singularities}, Tata Institute, Bombay (1976)

\bibitem{AM} M. Atiyah and I. G. Macdonald\emph{Introduction to Commutative Algebra}, Addison-Wesley, Reading (1969)

\bibitem{BEV} A. Bravo, S.~Encinas and O. Villamayor \emph{A simplified proof of desingularization and applications}, Rev. Mat. Iberoamericana 21, 349-458 (2005)

\bibitem{BM} E. Bierstone and P. Milman \emph{Canonical desingularization in characteristic zero by blowing up the maxium strata of a local invariant}, Invent. Math. 128, 207-302 (1997)

\bibitem{BMM} E. Bierstone and P. Milman \emph{Desingularization algorithms I. The role of exceptional divisors }, Moscow Math. J. 3, 751-805  (2003)

\bibitem{Cu} S. Cutkovsky \emph{Resolution of singularities}, Graduate Studies in Mathematics 63, Am. Math. Soc. Providence R.I. (2004)

\bibitem{EH} S. Encinas and H. Hauser \emph{Strong resolution of singularities in characteristic zero}, Comm. Math. Helv. 77, 821-845 (2002)

\bibitem{EN} S.~Encinas, A. Nobile and O. Villamayor \emph{On algorithmic equi-resolution and stratification of Hilbert schemes}, Proc. London Math. Soc 86, 607-648 (2003)

\bibitem{EV}
S.~Encinas and O. Villamayor \emph{A course on constructive desingularization and equivariance}, in H. Hauser, J. Lipman, F. Oort and A. Quir\'os, editors, {\it Resolution of singularities}, Progress in Mathematics 181, Birkhauser, Boston (2000)

\bibitem{G} J. Giraud \emph{Sur la th\'eorie du contact maximal}, Mat. Z. 137, 285-310 (1974)

\bibitem{H} R. Hartshorne \emph{Algebraic Geometry}, Springer-Verlag, New York (1977)

\bibitem{Hir} H. Hironaka \emph{Resolution of singularities of an algebraic variety over a field of characteristic zero, I-II}, Ann. of Math. Soc 79, 109-326 (1964)

\bibitem{Hiro} H. Hironaka \emph{Idealistic exponent of a singularity}, in Jun-Ichi Igusa editor, \emph{ Algebraic Geometry, the John Hopkins Centennial Lectures}, Johns Hopkins Univ. Press, Baltimore, 52-125 (1977)

\bibitem{Ko}  J. Kollar \emph{Resolution of singularities-Seattle lecture}, preprint, \\arXiv:math.AG/0508332v1

\bibitem{Kol}  J. Kollar \emph{Lectures on resolution of singularities},Princeton University Press, Princeton (2007) 

\bibitem{K} E. Kunz \emph{Introduction to Commutative Algebra and Algebraic Geometry}, Birkhauser, Boston (1985)

\bibitem{Mm} H. Matsumura \emph{Commutative Algebra}, W. A. Benjamin, Inc., New York (1970)

\bibitem{M} H. Matsumura \emph{Commutative ring theory}, Cambridge University Press, Cambridge (1989)

\bibitem{Ma} K. Matsuki \emph{Notes on the inductive algorithm of resolution of singularities of S. Encinas and O. Villamayor}, preprint, arXiv:math.AG/0103120

\bibitem{Mu} D. Mumford \emph{Lectures on curves on an algebraic surface}, Princeton University Press, Princeton (1966)

\bibitem{N} A. Nobile \emph{A note on flat algebras}, Proc. A.M.S. 64, 206-208  (1977)

\bibitem{S} M. Schlessinger \emph{Functors of Artin rings}, Trans. Amer. Math. Soc. 130, 208-222 (1968)

\bibitem{T}  B. Teissier \emph{R\'esolution simultan\'ee}, in M. Demazure, H. Pinkham and B. Teissier editors, \emph{S\'eminaire sur les sigularit\'es des surfaces}, Lecture Notes in Mathematics 777, Springer, Berlin, 71-81 (1980)

\bibitem{V1} O. Villamayor \emph{Constructiveness of Hironaka's resolution}, Ann. Sci. \'Ecole Norm. Sup. 22, 1-32 (1989)

\bibitem{V2} O. Villamayor \emph{Patching local uniformizations}, Ann. Sci. \'Ecole Norm. Sup. 25, 629-677 (1992)

\bibitem{W} J. Wlodarczyk \emph{Simple Hironaka Resolution in Characteristic Zero}, Journal of the A.M.S., 18, 779-822 (2005)

\bibitem{Za1} O. Zariski \emph{Studies in equisingularity, I}, Am. J. Math 87, 972-1006 (1965)

\bibitem{Za2} O. Zariski \emph{Studies in equisingularity, II},Am. J. Math 90, 961-1023 (1968)  

\end{thebibliography}
\end{document}